\documentclass[11pt, leqno]{amsart}
\usepackage[utf8]{inputenc}
\usepackage{geometry}
\usepackage{etoolbox}
\usepackage{lipsum, bm}
\usepackage[english]{babel}
\usepackage{csquotes}
\usepackage{mathrsfs, enumerate}
\usepackage{mathtools, tikz, pgfplots, subcaption, mathdots}
\usepackage{xcolor, multicol}
\usepackage{yfonts, float}
\usetikzlibrary{patterns}
\usepackage{thmtools}
\usepackage[makeroom]{cancel}

\usepackage{tocvsec2}
\tikzset{cross/.style={cross out, draw=black, minimum size=2*(#1-\pgflinewidth), inner sep=0pt, outer sep=0pt},
	cross/.default={2pt}}
\usetikzlibrary{shapes.misc}

\usepackage[breaklinks=true]{hyperref}
\usepackage{cleveref}

\newcommand{\supp}[0]{\mathrm{supp}}

\newcommand{\height}[0]{\mathrm{ht}}
\newcommand{\wt}[0]{\mathrm{wt}}
\newcommand{\ad}[0]{\mathrm{ad}}

\usepackage{amsmath, amsthm, amssymb}

\theoremstyle{plain}
\newtheorem{theorem}{Theorem}[section]
\newtheorem{lemma}[theorem]{Lemma}
\newtheorem{prop}[theorem]{Proposition}
\newtheorem{cor}[theorem]{Corollary}
\newtheorem{thmx}{Theorem}[section]

\theoremstyle{definition}
\newtheorem{definition}[theorem]{Definition}

\newtheorem{example}[theorem]{Example}

\newtheorem{question}{Question}
\newtheorem*{p-psp}{parabolic-PSP}
\newtheorem{corollary-C}[theorem]{Corollary to Theorem C}

\newtheorem{observation}[theorem]{Observation}

\numberwithin{equation}{section}

\newtheoremstyle{problem}{5pt}{5pt}{}{}{\normalfont}{\textbf{:}}{.5em}{}
\theoremstyle{problem}
\newtheorem{problem}{\textbf{Problem}}
\newtheorem{note}{\textbf{Note}}

\begin{document}

\title[Weights and characters of highest weight modules]{Weights and characters of highest weight modules}
	\author{Souvik Pal  and G. Krishna Teja}
	
	\begin{abstract}
		Let $\mathfrak{g}=\mathfrak{g}(A)$ be any Borcherds--Kac--Moody $\mathbb{C}$-Lie algebra (BKM LA) for BKM-Cartan matrix $A$, with Cartan subalgebra $\mathfrak{h}$.
        Let $V$ denote a highest weight $\mathfrak{g}$-module, with top weight $\lambda\in \mathfrak{h}^*$ (not necessarily in the domninant integral cone $P^+$).
        The non-integrable simples $V= L(\lambda)$ in [Naito, {\it Trans. Amer. Soc.} 1995] are widely studied beyond integrable simple $L(\nu)s,\ \nu \in P^+$.
     We introduce and study: 1) A weight cone $P^{\pm}=\big\{\mu\in \mathfrak{h}^*\ \big|\ \mu(\alpha_i^{\vee})\in \frac{A_{ii}}{2}\mathbb{Z}_{\geq 0}\text{ for all simple co-roots }\alpha_i^{\vee}\big\}$; note Weyl vector $\rho\in P^{\pm}\setminus P^+$. 2) The resulting (novel) non-integrable simple $L(\lambda)s, \ \lambda \in P^{\pm}\setminus P^{+}$; their Chevalley--Serre (CS) type relations (which are, in fact, complementary to those of integrable $L(\nu)$s); 3) Higher length CS type relations in any highest weight module under the name ``holes". 
     
     Using these, we obtain explicitly and uniformly, (notably) Weyl-orbit typed formulas for weight-sets of: all simples $L(\lambda)$s ($\forall$ $\lambda\in \mathfrak{h}^*$) and all quotients of parabolic Verma modules along imaginary directions.
     This generalizes and extends in one stroke, such formulas over Kac--Moody (KM) $\mathfrak{g}$, of all $L(\lambda)$ by Khare [{\it Trans. Amer. Math. Soc.} 2017] and Dhillon--Khare [{\it Adv. Math.} 2017], and also of all $V$ by Khare--Teja recently; which used parabolic and {\it higher order Verma} modules. 
  
We obtain Weyl--Kac--Borcherds type character formulas for  $L(\lambda) \text{ for } \lambda\in  P^{\pm}$, over negative rank-2 $\mathfrak{g}$'s; by exploring Verma module embeddings. 
We obtain character of every highest weight module $V$ for  $\lambda=\rho$ in negative $A$-type cases.  
All our results \big(even weight-formulas of infinitely many integrable $V$, for a given $\lambda \in P^{+}$\big) and ingredients are novel to the best of our knowledge.
	\end{abstract}\bigskip
   \maketitle
	\settocdepth{section}
	\tableofcontents
	\allowdisplaybreaks
    \section{Introduction} \label{Section 1}
    Throughout the paper, we work over complex numbers $\mathbb{C}$; and $\mathbb{R}\supset\mathbb{Z}\supset \mathbb{N}$ denote resp. the sets of reals, integers, natural numbers.
    We fix $\mathfrak{g}=\mathfrak{g}(A)$ 
to be the Borcherds--Kac--Moody Lie algebra (BKM LA) for an arbitrary fixed BKM-Cartan (BKMC) matrix $A_{\mathcal{I}\times \mathcal{I}}$.
We index the Dynkin graph nodes for $\mathfrak{g}$ also by $\mathcal{I}$.
Recall for BKMC $A$:\\
\hspace*{0.5cm}i) $A_{ii}\in \{2\}\sqcup\mathbb{R}_{\leq 0}$ (non-positive reals); \quad ii) $A_{ij}=0\iff A_{ji}=0$;\\  \hspace*{0.3cm}iii)  $A_{ij}\in \mathbb{R}_{\leq 0}$ for $i\neq j\in \mathcal{I}$;\quad iv) $A_{rj}\in \mathbb{Z}_{\leq 0}$ (non-positive integers) for $A_{rr}=2$ and $j\neq r\in \mathcal{I}$.
\begin{definition}
Corresponding to the diagonal entries of $A$, we define and fix the three {\underline{node-types}}:\\  $\mathcal{I}\ \   =\ \ \big(  \overset{\textbf{Real nodes}}{\mathcal{I}^+\ := \{i \ |\  A_{ii} =2\}} \big)\ \ \bigsqcup\   \ \big( \overset{\textbf{Heisenberg nodes}}{\mathcal{I}^0\ :=\{i\ |\ A_{ii}=0\}}\big) \ \  \bigsqcup\ \ \big(\overset{\textbf{Negative/Real-like nodes}}{\mathcal{I}^-:\ =\{i\ |\ A_{ii}\in \mathbb{R}_{<0}\}}\big)$.
\end{definition}\noindent
For a realization of $\mathfrak{g}$, we fix: $A=[A_{ij}=\alpha_{j}(\alpha_{i}^{\vee})]_{i,j\in \mathcal{I}}, \text{ Cartan subalgebra } \mathfrak{h},$ $\text{ simple roots }\Pi=  \{ \alpha_i \ |\  i \in \mathcal{I} \}\subset \mathfrak{h}^*,\ \text{simple co-roots } \Pi^{\vee}=\{ \alpha_i^{\vee} \ |\  i \in \mathcal{I} \}\subset \mathfrak{h}$.
Recall, $\dim \mathfrak{h}= |\mathcal{I}|+ \text{nullity of }A$.
$\mathfrak{g}$ has the root system $\Delta\subset \mathfrak{h}^*$, and the Weyl group $W$ generated by real-simple reflections $s_i$, $i\in\mathcal{I}^+$
\big(we deal with reflections for imaginary roots, see  Lemma \ref{Lemma (2,2) sol. characterization}\big).
The Chevalley generators, triangular decomposition, universal enveloping algebra of $\mathfrak{g}$ are
 $\{ e_i, f_i, \alpha_i^\vee \ | \ i \in \mathcal{I} \}$,  $\mathfrak{g}=\mathfrak{n}^+\oplus \mathfrak{h}\oplus \mathfrak{n}^-$, $U(\mathfrak{g})$ resp.

R. Borcherds \cite{Borcherds J. Alg, Borcherds Liesuper} introduced generalized Kac--Moody (KM) LAs -- called BKM LAs -- generalizing his Monster LA, or Monster vertex operator algebras (modules over affine KM LAs introduced in \cite{Lepowsky E8, Lepowsky Vertex operators}); the latter play a central role in string theory and conformal field theory in theoretical physics, helping in quantum chromodynamics and quantum gravity \cite{Witten}.
This was for studying and settling the Conway and Norton's moonshine conjecture(s) \cite{Conway--Norton, Thompson 2} for the (largest) monster sporadic simple group \cite{Fischer baby Monster, Leon baby Monster, Griess Inventions}, using the no-ghost theorem in string theory.
The Weyl--Kac--Borcherds character formulas \eqref{Eqn WKB character formula}, and their denominator identities shed more light on the structure of monster group, and modular functions and forms; e.g. Koike--Norton--Zagier identities.

BKM LAs are crucial settings of contragredient LAs studied by Kac--Kazhdan \cite{Kac--Kazhdan}.
BKM LAs include: 1) {\it Semisimple} and KM LAs under $\mathcal{I}=\mathcal{I}^+$; 2) Heisenberg LAs for $A=0$ and $\mathcal{I}=\mathcal{I}^0$; 3)~{\it partially commutative} LAs \cite{Duchamp, Lyndon}.
BKM $\mathfrak{g}$'s are constructed similar to KM LAs by {\it Serre relations}: $f_i^{1-A_{ij}}f_j=0$ $\forall$ $i\in \mathcal{I}^+$ and $[f_i, f_j]=0$ 
 when $A_{ij}=0$.
 For $i\in \mathcal{I}$, $\mathfrak{g}_{\{i\}}:=\mathrm{span}_{\mathbb{C}}\big\{ e_i,\alpha_i^{\vee}, f_i\big\}$ is isomorphic to : $\mathfrak{sl}_2(\mathbb{C})$ (traceless $2\times 2$ $\mathbb{C}$-matrices) if $i\in \mathcal{I}^+\sqcup \mathcal{I}^-$, or the 3-dim. Heisenberg LA if $i\in \mathcal{I}^0$.     
The representation theories of these subalgebras differing significantly from one another (Lemma \ref{Lemma rank-1 BKM rep. theory}) also makes the theory for BKM LAs interestingly subtler; seen for instance in the technicalities and applications of character formulas \eqref{Eqn WKB character formula}.
Let the negative part of $\mathfrak{g}$ be $\mathfrak{n}^-$.

This paper contributes to study weights, characters and presentations for arbitrary highest weight modules (we abbreviate as h.w.m.s) $V$'s over BKM $\mathfrak{g}$'s; these $V$'s need not be integrable or simple.
We discuss here, motivations and applications to our study. 
We elaborate on our main results with duly addressing several guiding questions in Section \ref{Section main results}.
The reader interested in them may proceed to Subsections \ref{Subsection wt-formula theorems}--\ref{Subsection character for L(lambda nice)}.
All the necessary concepts that we introduce are recorded in Section \ref{Section 2}.
\begin{definition}
Fix a BKM LA $\mathfrak{g}=\mathfrak{g}(A)$.
The classical dominant integral weight-cone over contragredient $\mathfrak{g}$'s, for the crucial integrable simple h.w.m.s over $\mathfrak{g}$, is given by
\begin{equation}\label{Eqn classical cone}
P^+:= \{ \lambda\in\mathfrak{h}^*\ \big|\  \lambda(\Pi^{\vee})\subset \mathbb{R}_{\geq 0} \text{, and } \lambda(\alpha_i^{\vee})\in \mathbb{Z}_{\geq 0}\ \
\forall\ i\in \mathcal{I}^+\}.
\end{equation}
Let $M(\lambda)$ and $L(\lambda)$ be the Verma $\mathfrak{g}$-module and simple h.w.m. with highest weight (h.w.) $\lambda\in \mathfrak{h}^*$.
Let $M(\lambda)\twoheadrightarrow V$ denote the h.w.m. $V$ with h.w. $\lambda$, and $V_{\mu}$ denote its $\mu$-th weight space $\forall$ $\mu\in \mathfrak{h}^*$.
\end{definition}

The first set of problems (among many others solved in the paper) on weights' side are below.
\begin{problem}\label{Question wt-form?}
Fix a BKM $\mathfrak{g}$ and any $\lambda\in \mathfrak{h}^*$.
a) What are all the weights of the $L(\lambda)$s, even for $\lambda\in P^+$? 
b) What are the weight-sets of non-simple $\mathfrak{g}$-h.w.m.s $V\twoheadleftarrow M(\lambda)$?
\end{problem}
We recall that even the weights of integrable simple h.w.m.s are seemingly unwritten in the literature; and that of arbitrary simple h.w.m.s were recently computed (only) up to KM settings in a series of works by Khare and Dhillon \cite{Khare_Ad, Khare_Trans} that we shall elaborate up on soon. 
Also, recall that we have the celebrated Weyl--Kac--Borcherds (WKB) character formulas \eqref{Eqn WKB character formula} for integrable simple $L(\mu)$'s and sophisticated extended Kazhdan--Lusztig theory for non-integrable simples (see Naito \cite{Naito 1}).  
However, their alternating nature and the involvement of computationally hard Kazhdan--Lusztig (KL) polynomials do not easily yield explicit answers to Problem a).

Our Theorem \ref{Theorem weight formula for nice modules} in Subsection \ref{Subsection wt-formula theorems} {\it explicitly} and {\it uniformly} describes weight-sets $\wt V$ for: 1) all simple h.w.m.s $ L(\lambda) \ \forall \ \lambda\in \mathfrak{h}^*$; 2) a large class of $V$'s which includes all h.w.m.s when every simple root in $\mathfrak{g}(A)$ is imaginary, i.e. when all $A_{ij}\leq 0$. 
It generalizes in one stroke, the formulas arising from series of works by Khare et.al. \cite{Chari_JPAA, Khare_JA, Khare_Trans, Khare_Ad, Teja--Khare} over KM $\mathfrak{g}$, ranging from weight-hulls of parabolic Vermas, to finer weight-sets of all $L(\lambda)$'s and $V$'s.
Motivations and applications of the study of weights up to KM settings from the literature are as follows.
\begin{note}\label{Note wt. form. motivations} Khare et al. (Theorem \ref{Theorem wt-formula for all V in KM setting} \eqref{Eqn wt-formula for all simple in KM setting}) solved Problem~\ref{Question wt-form?} a) over KM $\mathfrak{g}$, by first proving the following results for key objects  parabolic Verma modules, using integrability \eqref{Eqn Integrability} (size-1 holes) : \\
1)~$\wt L(\lambda)$ is the set of lattice points in its hull $\forall$ $\lambda\in \mathfrak{h}^*$; occuring as multiplicity-free character formula building on \cite{Brion, Postnikov}; has branching-fashioned slice-decompositions \eqref{Eqn wt-formula for all simple in KM setting} crucial for Khare et.al.
(These primarily solve some questions of D. Bump, M. Brion, J. Lepowsky, quoted in \cite{Khare_Ad}).\\
2)~The formulas in 1) have applications to shapes of hulls of $\wt V$'s for $V\twoheadleftarrow M(\lambda)$, which are important in convex geometry: determination of their shapes, polyhydrality, closedness (\cite[Theorem 2.5]{Khare_Ad}); faces and inclusions between them (\cite{Khare_Ad, Khare_Trans}); extremal rays (\cite[Corollary 4.16]{Dhillon_arXiv}) etc.\\
3) To extend classical results in 2) for shapes of {\it root and Weyl polytopes} from \cite{Satake, Borel, Cellini} to infinite hulls of $\wt L(\lambda\notin P^+)$; also for the equivalence of faces with recent combinatorial subsets of weights in \cite{Chari_Adv, Chari_JPAA, Khare_JA, Khare_AR}, which was established in \cite{WFHWMRS}.
For minimal descriptions for $\wt L(\lambda)$s, see \cite{Teja-ArXiv, Teja-Fpsac}.
\end{note}
More recently, the second author and Khare \cite{Teja--Khare} computed {\it five weight-formulas} for every h.w.m. over semisimple and KM LAs.
In that paper, they used novel tools of {\it Holes} and revealed {\it Higher order Verma modules} along with certain {\it Weyl semigroups}; also a previously unwritten weight-formulas for universal family of {\it Parabolic Verma modules}.
We record their definitions in the next section.
Our inspirations to problem b) above also include developing all those tools for BKM case.

We now move to the character-side of h.w.m.s $V$'s. Over BKM LAs, study of non-integrable simple h.w.m.s $L(\mu)$ for $\mu\notin P^+$, was pioneered by Naito \cite{Naito 1, Naito 2} for weights $\mu$'s that appear on the exponentials in the classical WKB formula \eqref{Eqn WKB character formula}.
Let $\Delta^+$ be the subset of positive roots of $\Delta$.
\begin{problem}
   The top weights of simple h.w.m.s considered by Naito \cite{Naito 1, Naito 2}, cover the whole integral weight lattice, up to KM cases.
   It is natural to explore if this phenomenon holds good over BKM LAs.
   Naito's works do not cover simples $L(\rho)$ for crucial  Weyl vector(s) $\rho$ (half-sum of positive roots in finite type, and ally for Harish-Chandra's linkage, for the WKB characters and denominator identities).
   Note $\rho\notin P^+$.
To what extent, extended KL-theory developed by Naito in BKM case, describes the characters of generic simple h.w.m.s with integral h.w.s (also for $\rho$)?
\end{problem}
Naito's works do not cover a  class of simple h.w.m.s, which  includes $L(\lambda)$'s for  $\lambda\in -P^+$ in negative settings of $\mathcal{I}=\mathcal{I}^-$.
We explain this in view point of rank-1 theory (Lemma \ref{Lemma rank-1 BKM rep. theory}) in the next section.
We introduce a novel $P^{\pm}$ weight cone in Definition \ref{Defn integrability of lambda} for those outlandish h.w.s, and initiate a study of the Vermas $M(\lambda)$'s and their quotients for $\lambda\in  P^{\pm}$.
Those simple h.w.m.s are universal for (and by-products of) our weight analysis.
Also, they were previously not studied in the literature, which we learnt from M. Wakimoto and S. Viswanath.

 $P^{\pm}$ refines classical dominance and integrality (Observation \ref{Observation signed-cone importance}), and the concepts of integrability and Chevalley--Serre presentations in simple h.w.m.s.
Indeed, $P^{\pm}$ captures Weyl vectors $\rho$  (see Observation \ref{Observation sign-dim int wts and Weyl vetor}), that is non-dominant and non-integral particularly when $\mathcal{I}^-\neq \emptyset$. 
\begin{note}\label{Note -rho}
Kac--Kazhdan \cite{Kac--Kazhdan} quotes studying $L(-\rho)$ and the simplicity of $M(-\rho)$ over (non-semisimple) contragredient $\mathfrak{g}(A)$'s to be important.
They conjectured on characters of $L(-\rho)$ using its free imaginary root-directions.
This was settled in KM cases (e.g. \cite{Hayashi}), and seemingly not beyond.
To our knowledge, the study on the complementary side of weights and characters of h.w.m.s with h.w. $+\rho$ -- particularly, the structure of $M(+\rho)$ --  has been unexplored; even in rank-2 when $\mathcal{I}^-\neq\emptyset$.
\end{note}

The second-half of the paper computes WKB (\eqref{Eqn WKB character formula}) type character-formulas and presentations in Theorems \ref{Theorem C}--\ref{Theorem D character of V(rho)} (in Subsection \ref{Subsection character for L(lambda nice)}) over two ``negative'' settings of $\mathfrak{g}(A)$ introduced in \eqref{Eqn negative settings of A}, for the following modules: \
1) Several non-integrable simple h.w.m.s with h.w. from $P^{\pm}$ (complementary to simples studied by \cite{Naito 1}), over rank-2 $\mathfrak{g}(A(b,a,c,d))$ (in fact, we study the structure of their Verma
covers).
2) All h.w.m.s $V$'s with h.w.s $\lambda=\rho$, in ``negative type-$A$ case'' of $\mathfrak{g}(A(n))$.\allowdisplaybreaks 
\begin{note} \label{Note all h.w.m.s over -ve An case are higher order Vermas}
We record here quick examples of higher order Vermas, before their definition. 
 Namely, all h.w.m.s $V$ with h.w. $\lambda=0$, when $\mathfrak{g}=\mathfrak{sl}_2(\mathbb{C})\oplus \cdots \oplus \mathfrak{sl}_2(\mathbb{C})$ ($A_1\times \cdots \times A_1$-type); reminiscent of monomial ideals in polynomial algebras.
 An interesting consequence of Theorem \ref{Theorem D character of V(rho)} is that the same assertion is true in the negative BKM setting of $A(n)$ in \eqref{Eqn negative settings of A} for $\lambda=\rho$.
 Namely, every h.w.m. $V\twoheadleftarrow M(\rho)$ is a higher order Verma; we determine all of their characters.
It might be interesting to explore if this is true for other $\lambda$'s in $A(n)$ case, as well as for $\lambda=\rho$ outside $A(n)$ case.
   \end{note}
\begin{equation}\label{Eqn negative settings of A}
\textbf{(A)}\ \ 
\underset{\large \mathcal{I}=\mathcal{I}^-= \{1,2\},\ \ a,b,c,d\in \mathbb{N}}{   A(b,a,c,d):=\begin{bmatrix}
   -b & -a\\
   -c & -d
\end{bmatrix}}.  \qquad   \textbf{(B)} \ \  A(n):=\underset{ \large  \mathcal{I}=\mathcal{I}^-=\{1, \ldots, n\} }{\begin{bmatrix}
             -2 & -1 & 0 & 0 & \cdots & 0  \\
             -1 & -2 & -1 & 0& \cdots & 0\\
    \  \vdots &  & \ \ \ \ \ \  \ddots & & \cdots  & 0 \\
             0 & 0 & \cdots & -1&  -2 & -1\\
             0 & 0 & \cdots & 0 & -1 & -2 
           \end{bmatrix}_{\mathclap{\ \ \  \ n\times n}}}.
\end{equation}

These two settings (involving only $\mathcal{I}^-$, i.e. negative simple roots) are fundamental for building our understanding of simple h.w.m.s from $P^{\pm}$, and for their Verma module structures which was again unexplored.
They include the crucial (negative) analogue $\left(\begin{smallmatrix}
    -2 & -1\\ -1 & -2
\end{smallmatrix}\right)$ of classical $\mathfrak{sl}_3(\mathbb{C})$; indeed we observe parallels of it to $\mathfrak{sl}_3(\mathbb{C})$.
Our character formulas for $L(\rho)$ are novel even in that rank 2 case; they also remind us of Weyl semigroups (\eqref{Eqn Weyl semigroup}) for higher order Vermas in \cite[Theorem F]{Teja--Khare}.

It might be worth recalling that when $A_{ij}\neq 0$ $\forall$ $i,j\in \mathcal{I}=\mathcal{I}^-$, $\mathfrak{n}^+ \text{ and } \mathfrak{n}^-$ are free LAs (e.g. \cite{Elizabeth}); with Witt's necklace formula for their root-multiplicities.
We also recall that, lower rank settings are prominent in the literature for studying: maximal vectors in Verma modules \cite{Malikov}, KL-theory in KM setting \cite{Wallach--Rocha}, identities using Weyl--Kac characters \cite{Feingold}, KM LA embeddings \cite{Naito subalgebras}, and word maps on (BKM) LAs \cite[Section 5]{Bracket width}. 
Let $\mathfrak{g}$ be symmetrizable, with its standard invariant form $(.,.)$.

Fix any $\lambda \in P^{\pm}$.
Our strategy for characters $\mathrm{char}(V)$'s of h.w.m.s $M(\lambda)\twoheadrightarrow V \twoheadrightarrow L(\lambda)$ (also for $\lambda=\rho$), is to compute solutions $\mu$ to {\it Norm equality} \eqref{Eqn norm equality with invariant form} and their $c(\mu)\in \mathbb{Z}$ in principle formula:
    \begin{equation}\label{Eqn char formula principle 1 from norm equality}
             \mathrm{char}(V)\ = \ \sum\limits_{\mu \text{ s.t. }\eqref{Eqn norm equality with invariant form}}c(\mu) \mathrm{char}\big(M(\lambda)\big)\ = \  \frac{\sum\limits_{\mu \text{ s.t. }\eqref{Eqn norm equality with invariant form}}c(\mu)e^{\mu}}{R\ = \ \prod\limits_{\alpha\in \Delta^+}\big(1-e^{-\alpha}\big)^{\dim \mathfrak{g}_{\alpha}}} \quad\text{wherein}\vspace*{-3mm} \end{equation}
        \begin{equation}\label{Eqn norm equality with invariant form}
                \ (\lambda+\rho \ , \ \lambda+\rho)\  = \  \ (\mu+\rho\ , \ \mu+\rho ).\vspace*{-7mm}\hspace*{3cm}
            \end{equation}
\begin{problem}\label{Question dot orbit = norm equality sol.?} 
Let $\mathfrak{g}$ be of finite or affine type (or even  be of type $A$).
(a) Is the solution-set $\{\mu\in \mathfrak{h}^*\ |\lambda\succeq  \mu \text{ satisfy } \eqref{Eqn norm equality with invariant form}\}$  equal to the {\it dot-orbit} $(W\bullet \lambda)\cap[\lambda-\mathbb{Z}_{\geq 0}\Pi] $, or more generaly  the {\it strong linkage dot-orbit} of $\lambda$  (\cite[Section 4]{Kac--Kazhdan} and \cite[Proposition 2]{Naito BGG 2}) by Kac--Kazhdan equation \eqref{Eqn dot action by any positive root} below?
   (b) Is this true for $\lambda\in P^+$, or even $\lambda=0$?
   We aim to explore them for $\lambda\in  P^{\pm}$.
\end{problem}
No answers were known previously to Problem \ref{Question dot orbit = norm equality sol.?} even in finite type, according to some experts including M. Wakimoto.
Theorem \ref{Theorem D character of V(rho)} gives a
 negative answer to Problem \ref{Question dot orbit = norm equality sol.?}(a) when $\lambda=\rho$.
We do not know of an answer to (b).
Our Proposition \ref{Prop maxl vect} constructs maximal vectors in $M(\lambda)_{\mu}$, for solutions $\mu$ to \eqref{Eqn norm equality with invariant form} in some cases, following their count's bounds in \cite[Section 4]{Kac--Kazhdan}.
\begin{note}\label{Note KK unique solution}
a) Crucial Lemma 3.3 in \cite{Kac--Kazhdan} studies Kac-Kazhdan {\it dot-action} condition \eqref{Eqn dot action by any positive root} with unique solution(s) $\mu\precneqq \lambda$.
(See Note \ref{NOte KK maxl vect bounds} in Subsection \ref{Subsection character for L(lambda nice)}.)
b) But a characterization, or even example cases for such $\mu$'s seem to be unknown.
c) For our goal of $P^{\pm}$-simples' characters, our Proposition \ref{Prop number theory} characterizes unique solutions in ``interior of $\wt M(\lambda)$'' to those equations, in some cases.
\end{note}
We conclude this section with another motivation for us to study weights.
Namely, their equivalence with characters (when they are close to multiplicity-free characters), for the following special class of h.w.m.s $V$'s: i)~all $V$'s over (Heisenberg LA type) $\mathfrak{g}(0_{\mathcal{I}\times \mathcal{I}})$; ii) {\it 2-nd order Verma modules} over some rank-3 KM $\mathfrak{g}$'s and in type $A_3$ in \cite{Teja--Khare}.

\section{Elements in our study of weights and characters in Borcherds case}\label{Section 2}
We begin by developing our first ingredient, the notion of ``full'' dominant integral cone $P^{\pm}$ in Definition \ref{Defn integrability of lambda} (strengthening $P^+$-cone), guided by the following ``complete $\mathfrak{sl}_2$-theory'' picture.
\begin{lemma}\label{Lemma rank-1 BKM rep. theory}
Fix any rank-1 BKM or contragredient $\mathfrak{g}= \mathbb{C}\{e_i, \alpha_i^{\vee}, f_i\}$, for $\mathcal{I}=\{i\}$, and $\lambda\in \mathfrak{h}^*=\mathbb{C}\alpha_i$.
\big($M(\lambda)$ is the polynomial ring $\mathbb{C}[f_i]$.\big) $f_i^n M(\lambda)_{\lambda}= M(\lambda)_{\lambda-n \alpha_i}$ are maximal (i.e., killed by $\mathfrak{n}^+$) for $n\in \mathbb{N}$ $\iff$ $n\Big(\lambda(\alpha_i^{\vee})-\frac{A_{ii}}{2}(n-1) \Big)=0\ \iff\ n= 
\begin{cases}
\frac{2}{A_{ii}}\lambda(\alpha_i^{\vee})+1,\ &\text{ if }A_{ii}\neq 0,\\    
\text{arbitrary,} &\text{ if }A_{ii}\ = \ \lambda(\alpha_i^{\vee})\ = \ 0.
\end{cases}$
So $M(\lambda)$ over $\mathfrak{g}_{\{i\}}$ is simple $\iff \lambda(\alpha_i^{\vee})\notin \frac{A_{ii}}{2}\mathbb{Z}_{\geq 0}$.
Here, $\mathbb{Z}_{\geq 0}$ are non-negative integers.
\end{lemma}
\begin{example}\label{Example all Vectors in Heiesnberg-Verma are maximal} (1) For $\mathfrak{g}\big([0]_{\mathcal{I}^0\times \mathcal{I}^0}\big)$, any vector lying in $M(0)$ is maximal.
(2) For $M(-4)$ over $\mathfrak{g}\big([-2]_{\{1\}\times  \{1\}}\big)$, note $f_1^5 M(-4)_{-4}$ are maximal, and moreover $f_1^5 L(-4)=0$; so $\dim L(2\alpha_1)< \infty$.
\end{example}
\begin{definition}\label{Defn integrability of lambda} We define over any BKM or even contragredient $\mathfrak{g}$,  
\begin{equation}
P^{\pm}\ :=\  \Big\{ \lambda\in \mathfrak{h}^*\ \Big| \ \lambda(\alpha_i^{\vee})\in {\tiny \frac{A_{ii}}{2}}\mathbb{Z}_{\geq 0}\ \forall\ i\in \ \mathcal{I}\  =  \mathcal{I}^+\sqcup\mathcal{I}^0\sqcup \mathcal{I}^-\Big\}.
\end{equation}
 In reverse, for any $\lambda\in \mathfrak{h}^*$, its {\it maximum integrability } is $J_{\lambda}:= \Big\{i \in \mathcal{I}\ \big|\ \lambda(\alpha_i^{\vee})\in \frac{A_{ii}}{2} \mathbb{Z}_{\geq 0} \Big\}$ (also see  Definition \eqref{Eqn Integrability}). 
Note that $\lambda(\alpha_j^{\vee})=0$ $\forall$ $j\in J_{\lambda}\cap \mathcal{I}^0$.
Quotients along $\mathcal{I}^-$ in $M(\lambda)$s (as in Example \ref{Example all Vectors in Heiesnberg-Verma are maximal}), even $L(\lambda)$ with $J_{\lambda}\cap \mathcal{I}^-\neq \emptyset$, were not seen in the literature to our knowledge.
\end{definition}
\begin{observation}\label{Observation sign-dim int wts and Weyl vetor}
   (1) Neither $P^+\not\subset P^{\pm}$ nor $P^{\pm}\not\subset P^+$, see Example \ref{Counter examples for slice-decomp.}.
(2) This signed-version $P^{\pm}$ of $P^+$ is seen in integrable setting with $\rho$, and not beyond.
$\rho(\alpha_i^{\vee})=\frac{A_{ii}}{2}$ $\forall$ $i$, and so $\rho\in P^{\pm}\setminus P^+$.
  \end{observation}
  \begin{observation}\label{Observation signed-cone importance}
   1)$P^{\pm}$-directions reveal the maximal vectors that are lost in during $V \twoheadleftarrow M(\lambda)$ from the 1-dim. weight spaces in $M(\lambda)$. 
   Essentially, the {\it holes} in $V$ are the weights of those lost vectors (see Definitions \ref{Defn holes in KM setting}, \ref{Defn holes in BKM setting}); in fact, we show that holes solely determine $\wt V$.\\
2) $P^{\pm}$-directions reveal {\it Chevalley-Serre (CS) type presentation} relations and there by all the sets $\wt L(\lambda)$ in Corollary \ref{Corollary to Theorem A}); we develop those CS relations through points 2$'$)--2$''''$) below.\\
2$'$) To begin, CS relations in $L(\lambda)$, $\lambda\in P^+$ are- $f_i^{\lambda(\alpha_i^{\vee})+1}L(\lambda)_{\lambda}=0\ \  \forall\ i\in \mathcal{I}^+ ; \ \ \ f_j^{1}L(\lambda)_{\lambda}=0$ for $j\in \mathcal{I}^0\sqcup \mathcal{I}^-$ with $\lambda(\alpha_j^{\vee})=0$.
\begin{equation}\label{Eqn Integrability}
\underset{\text{(Real CS-type relations)}}{2'') \text{ Integrability}  \text{ of any }V\twoheadleftarrow M(\lambda):} \qquad I_V:=\Big\{i\in \mathcal{I}^+ \Big| \lambda(\alpha_i^{\vee})\in \mathbb{Z}_{\geq 0}, \ f_i^{\lambda(\alpha_i^{\vee})+1 }V_{\lambda}=0\Big\}.\quad
\end{equation}
Note $\text{over KM }\ \mathfrak{g},\  I_{L(\lambda)} = J_{\lambda} \ \supseteq I_V \ \forall \  0\neq V \twoheadleftarrow M(\lambda)$.
$V$ is integrable over the parabolic subalgebra $\mathfrak{p}_{I_V}:=\mathfrak{n}^+\oplus \mathfrak{h}\oplus$\big(the subalgebra $\langle f_j\ |\ j\in I_V \rangle$ generated by $\{f_{j}\ |\ j\in I_V \}$\big).
 So~$\wt V \text{ and character }\\ \mathrm{char}(V)$ of $V$ are invariant under the parabolic subgroup $W_{I_V}:=\langle s_i\ |\ i\in I_V \rangle
 \subseteq W$.\\
 2$'''$) In BKM case,  \cite{Naito 1, Naito 2} mark the perspective-shift from the widely treated (for $A$ symmetrizable, in \cite{Kac--Kazhdan, Khare_Ad}) relations $f_h^1 L(\lambda)_{\lambda}=0$, $\lambda\in P^+$ in 2$'$)--2$''$).
Namely, to their higher-lengthed $\prod_{h\in H}f_h \cdot L(\lambda)_{\lambda}=0$, which are included in holes \eqref{Eqn holes in KM setting}, \eqref{Defn holes in BKM setting}.\\
 2$''''$) Now, in generic $L(\lambda)$ -- notably in $L(\rho)$ (see Theorem \ref{Theorem D character of V(rho)}) -- there are more CS type relations:\qquad 
    $\displaystyle f_j^{\frac{2}{A_{ii}}\lambda(\alpha_i^{\vee})+1} L(\lambda)_{\lambda}\ = \ 0$ \ $\forall$\ $i\in J_{\lambda}\cap \mathcal{I}^-$. This holds by Lemma \ref{Lemma rank-1 BKM rep. theory}, and it generalizes \eqref{Eqn Integrability}. 
\end{observation}
In KM case, $L(\lambda)$ has full CS relations iff $\lambda\in P^+$. 
In this perspective, over 
contragredient $\mathfrak{g}$, $L(\lambda)$ has full CS relations iff $\lambda\in P^{\pm}$.
Recall, $e_i, f_i$ act locally nilpotently on $L(\lambda\in P^+)$ $\forall$ $i\in \mathcal{I}^+$.

The next tool is the concept of holes, recorded in KM case first; see Definition \ref{Defn holes in BKM setting} in BKM case.
\begin{definition}[\cite{Teja--Khare}]\label{Defn holes in KM setting} Fix KM $\mathfrak{g}$.
The set of \textbf{holes} in $V\twoheadleftarrow M(\lambda)$ \big(\text{larger commutative CS-relations}\big):  \begin{align}\label{Eqn holes in KM setting}
\begin{aligned}
   \mathcal{H}_V\ :=\  \left\{ H \subseteq J_{\lambda} \ \bigg|\ H \text{ is {\bf independent}, }  \prod\limits_{h\in H}f_h^{\boldsymbol{\lambda(\alpha_h^{\vee})+1}}V_{\lambda} \ = 0 \right\}.\ \ \ 
    \end{aligned}
    \end{align}
    \end{definition}
   Despite KL theory for characters of non-integrable $L(\lambda)$'s in finite type (\cite{Kashiwara, Kazhdan--Lusztig}), and KM (\cite{Wallach--Rocha, Kashiwara KL Conj Sym-KM}) and BKM (\cite{Naito 1, Naito 2}) cases,
$\mathrm{char}\big(L(+\rho)\big)$s were not explored to our knowledge and also not in \cite{Naito 1, Naito 2}.
\begin{observation}
In our terminology, Naito computes $\mathrm{char}\big(L(\mu)\big)$, for $\mu$ in the $W$ dot-orbits of weights $\lambda-\sum_{h\in H}\alpha_h$ in $M(\lambda\in P^+)$s -- hole-typed as in 2$'''$) -- for independent  $H \subseteq \{i\in \mathcal{I}^0\sqcup \mathcal{I}^-\ |\ \lambda(\alpha_i^{\vee})=0\}\subseteq J_{\lambda}$.
So, holes do appear in a disguise in those works from BKM case on-wards, along imaginary directions; but only as singletons in real directions, leading to $W$ symmetry.
We show applications of working with combinations of such holes: 1) Weight formulas (particularly, freeness in weights, by Proposition \ref{Proposition Minkowski difference formula}). 
2) Certain symmetry in characters as in finite type \big(e.g. finitely many exponentials in $\mathrm{char}(L(\mu\in P^{\pm}))$ as in Theorems \ref{Theorem C}, \ref{Theorem D character of V(rho)}; also see Lemmas \ref{Lemma bddness of sols for two negative nodes}, \ref{Lemma (2,2) sol. characterization}\big). 
Following Lemma \ref{Lemma rank-1 BKM rep. theory} and \cite{Naito 2}, it might be interesting to study such $L\big(\lambda - \sum_{h\in H}\big(\frac{2}{A_{hh}}\lambda(\alpha_h^{\vee})+1\big)\alpha_h\big)$ for independent $H\subseteq J_{\lambda}\cap \mathcal{I}^-$.
For these and following the works of Kac--Kazdhan \cite[Section 4]{Kac--Kazhdan}, our Theorem \ref{Theorem C} and Proposition \ref{Prop maxl vect} (in Subsection \ref{Subsection character for L(lambda nice)}) initiate writing the WKB type character formulas for $L(\lambda)$, and maximal vectors in $M(\lambda)$ when $\lambda\in P^{\pm}$.
Holes also appear in the WKB character formulas, on whom our formulas for $\mathrm{char}\big(L(\lambda\in P^{\pm})\big)$s build-on.
\end{observation}
Let us denote the weight-set of a $\mathfrak{h}$-weight module $M$ by $\wt M\ := \ \big\{ 
\mu\in \mathfrak{h}^*\ \big|\ 
\mu\text{-weight space } M_{\mu}:=\big\{m\in M\ \big|\  hm=\mu(h)m\ \forall\ h\in \mathfrak{h}\big\}\neq 0 \big\}$; and so $M= \bigoplus_{\mu\in \wt M}M_{\mu}$.
And let us denote the formal-character of such $M$ by $\mathrm{char}M := \sum_{\mu\in \wt M}\dim M_{\mu}e^{\mu} $.

In continuation to the discussion on holes, we now record the WKB celebrated character formula.
\begin{align}\label{Eqn WKB character formula}
\begin{aligned}[t]
&{\bf Weyl}-{\bf Kac}-{\bf Borcherds}\ {\bf (WKB)}\\
&{\bf \ \  character \ formula} \quad \text{\cite{Kac book, Wakimoto}}\end{aligned}   {\bf :}\ \ 
\mathrm{char} L(\lambda)\ = \sum_{w \in W}
\frac{(-1)^{\ell(w)}\ \  w{\bf S_{\lambda}}}{e^{\rho} \hspace*{-2mm}\prod\limits_{\alpha \in \Delta^+}
(1 -e^{-\alpha})^{\dim(\mathfrak{g}_{\alpha})}}\  \forall\ \lambda\in P^+.
\end{align}
The ingredients above are: i) $\rho\in P^{\pm}$;
ii)  (holes) $\emptyset \subseteq \text{ independent } I\subseteq$ $ \mathcal{I}^0\sqcup \mathcal{I}^-$ with $\lambda(\{\alpha_{i}^{\vee} \ |\  i\in I\})=\{0\}$, on which ${\bf S_{\lambda}} \ :=\  e^{\lambda+\rho}\    \sum_{I} (-1)^{|I|}  e^{-\sum_{i\in I}\alpha_i }$ runs on.
The {\it dot-action} \eqref{Eqn dot action by any positive root} of holes alongside $W$'s was central to study $L(\mu\notin P^+)$, their characters, their Garland--Lepowsky, BGG type resolutions (also of parabolic Vermas) \cite{Naito 1}--\cite{Naito BGG 2}.
Specializing ${\bf S_{\lambda}}$ to $e^{\lambda+\rho}$ $(\mathcal{I}=\mathcal{I}^+)$, \eqref{Eqn WKB character formula} yields the Weyl--Kac character formula -- which includes  Macdonald identities \cite{Kac Macdonald identities} in affine case for $\lambda=0$.

\begin{definition}\label{Defn integrables}
Fix a BKM $\mathfrak{g}$, $\lambda\in P^+$ and $M(\lambda)\twoheadrightarrow V$. 
$V$ is said to $\mathfrak{g}$-integrable iff $f_i^{\lambda(\alpha_i^{\vee})+1}V_{\lambda}=0$ $\forall$ $i\in \mathcal{I}^+$.
We can have $\mathfrak{g}$-integrable $V$ which are $\mathfrak{n}^-_{\mathcal{I}^-\sqcup \mathcal{I}^0}$-free (the subalgebra generated by $\{f_j\ |\ j\in \mathcal{I}^-\sqcup \mathcal{I}^0 \}$).
So we can have infinitely many integrable h.w.m.s for $\lambda\in P^+$: 
$L^{\max \textnormal{ int}}(\lambda):= \frac{M(\lambda)}{\Big\langle f_j^{\lambda\big(\alpha_j^{\vee}\big)+1}M(\lambda)_{\lambda}\ \Big| \ j\in  \mathcal{I}^+  \Big\rangle} \ \twoheadrightarrow \frac{L^{\max\textnormal{ int}}(\lambda)}{\Big\langle f_i^{n_i}L^{\max \textnormal{ int}}(\lambda)_{\lambda} \Big|\ n_i= 1 \text{ or }  \infty \text{ if }i\in \mathcal{I}^- \  ;\ n_i \in \mathbb{N}\cup \{\infty\}\ \text{if }i\in \mathcal{I}^0 \Big\rangle}\ \twoheadrightarrow\ L(\lambda)$.
\end{definition}
The above WKB character formula is only for integrable $V$ that is simple.
In the context of h.w.m.s up to BKM settings, characters are seemingly unknown beyond those of: 1) $L(\lambda\in P^+)$; 2) the simples in \cite{Naito 2}; 3) parabolic Vermas (for non KM cases).
We expect that the weight and character analysis in this paper to help the study of characters of the above mentioned integrable h.w.m.s $L^{\max \text{int}}(\lambda\in P^+) \twoheadrightarrow V\twoheadrightarrow L(\lambda\in P^+)$ \big(see Corollary \ref{Corollary to Theorem A} of Theorem \ref{Theorem wt-formula by composition series simples for nice V} for all their weights\big).

We record the important weight-formulas in aforementioned works of Khare et al. and those for all $\wt V$'s in \cite{Teja--Khare}, over KM $\mathfrak{g}$'s.
These needed the novel concept of Holes, Higher order Vermas (which also featured in \cite{Log-conc}) defined in \eqref{Eqn defn higher order Vermas} in Section \ref{Section main results}, {\it Slice-formulas} of parabolic (1-st order) Vermas \eqref{Eqn wt-formula for all simple in KM setting}.

Let $L_{_{J_{\lambda}}}(\mu)$ be the simple h.w.m over $\mathfrak{p}_{J_{\lambda}}$ with h.w. $\mu\in \mathfrak{h}^*$, and
 ` $\bullet$ ' be the {\it dot}-action of $W$ on $\mathfrak{h}^*$.
We define involutions $s_H:=\prod_{h\in H}s_h\ \in W$ for independent $H \subseteq \mathcal{I}^+$.
See Subsection \ref{Subsection 2.1} for more notations.
\begin{theorem}[{\cite{Khare_JA, Dhillon_arXiv, Teja--Khare}}]\label{Theorem wt-formula for all V in KM setting} Fix any KM LA $\mathfrak{g}$, $\lambda\in \mathfrak{h}^*$ and a h.w. $\mathfrak{g}$-module $M(\lambda)\twoheadrightarrow V\neq 0$. 
\begin{align}\label{Eqn wt-formula for all simple in KM setting}
\begin{aligned}
    \wt L(\lambda)\ = \ \wt M(\lambda, J_{\lambda})\ = \ \underset{\textbf{Integrable slice decomposition}}{\bigsqcup\limits_{(c_i\in \mathbb{Z}_{\geq 0})_{i\in J_{\lambda}^c}} \wt L_{J_{\lambda}}\Big( \lambda- \sum_{i\in J_{\lambda}^c}c_i\alpha_i\Big)}.
    \end{aligned}\qquad
    \begin{aligned}  &\textnormal{\cite[Theorem D]{Khare_JA}}\\
    &\textnormal{\cite[Section 2.1] {Dhillon_arXiv}}
    \end{aligned}
\end{align}
\begin{equation}\label{eqn wt-formula for all V in KM setting}
    \wt V \ = \   \wt \mathbb{M}(\lambda,\ \mathcal{H}_V)\ = \ \underset{\quad(\text{by  }\textbf{local Jordan--Hölder series factors} \text{ in }V)  }{\bigcup\limits_{\text{independent }H\notin \mathcal{H}_V}\wt L(s_H\bullet \lambda)}.\qquad \textnormal{\cite[Theorems A--C]{Teja--Khare}}
\end{equation}\end{theorem}
\section{Main results}\label{Section main results}
\subsection{Weights of h.w.m.s $V$ over BKM $\mathfrak{g}$, by $P^{\pm}$-structured~holes}\label{Subsection wt-formula theorems}
  Let `$\preceq$' be the usual partial order on $\mathfrak{h}^*$.
   If $x=\sum_{i\in \mathcal{I}}c_i\alpha_i$ for $c_i\in \mathbb{C}$, then let $\supp(x):=\{i\in \mathcal{I}\ |\ c_i\neq 0\}$.  
\begin{question}\label{Question on wt-formulas for simples}
 (a) Does $\wt L(\lambda)$ for $\lambda\in P^{\pm}$ admit Weyl-orbit formulas \eqref{Eqn Weyl orbit wt-formula in KM setting} (\cite[Proposition 11.2]{Kac book})?
    \begin{equation}\label{Eqn Weyl orbit wt-formula in KM setting}
        \text{Over KM }\mathfrak{g},\  \wt L(\lambda\in P^+) \ = \ W\big\{ \mu\preceq \lambda\ \big|\  \mu\in P^+,\ \  \supp(\lambda-\mu) \text{ is {\it non-degenerate} w.r.t. }\lambda\big\}.
    \end{equation}
 (b) Are simple $L(\lambda\in P^+)$s or  $L(\lambda\in P^{\pm})$s atomic for weight considerations; 
 i.e. does any of them yield integrable slice decompositions for arbitrary $\wt L(\lambda)$ similar to \eqref{Eqn wt-formula for all simple in KM setting}?
  (c) Notably do $\wt L(\lambda\in P^+)$s admit such decompositions/approximations via $\wt L(\lambda\in P^{\pm})$s, or vice-versa?
\end{question}
Question (b) has a negative answer and we do not know the answer to (c) (see Corollary \ref{Corollary no integtrable Slice-decomp.} in Subsection \ref{Subsection 2.2})
Our first result Theorem \ref{Theorem weight formula for nice modules} solves Question \ref{Question on wt-formulas for simples}(a) positively, with determining all h.w.m.s $V$ (all $L(\lambda)$, parabolic Verma typed $V$,...) with fundamental weight-formula \eqref{Eqn Weyl orbit wt-formula in KM setting}; using:  
    \begin{definition}[{\textbf{Holes} over BKM $\mathfrak{g}$}]\label{Defn holes in BKM setting}\ 
		[a] \underline{The hole-set of $M(\lambda)\twoheadrightarrow V\neq 0$} is a collection of tuples~:
		\[
		\mathcal{H}_V:=
		\left\{\ 
		(H,m_H)\  \text{ satisfying following conditions (H1)--(H3) }
		\right\}.
		\]
		\begin{itemize}
			\item[(H1)] Independent $\emptyset\neq H\subseteq J_{\lambda}$ (Definition \ref{Defn integrability of lambda});\ \  so $\dim M(\lambda)_{\mu}=1$ $\forall$ $\mu \in \lambda-\mathbb{Z}_{\geq 0}\{\alpha_h\ |\ h\in H\}$.
			\item[(H2)] Power-map/sequence $m_H: H\rightarrow \mathbb{Z}_{\geq 0}$\  satisfies $m_H(h)=|\frac{2}{A_{hh}}\lambda(\alpha_h^{\vee})|+1$ $\forall$ $h\in H\cap(\mathcal{I}^+\sqcup\mathcal{I}^-)$.
           Example \ref{Example all Vectors in Heiesnberg-Verma are maximal}(1) explains the arbitrariness of $m_H(h) >\lambda(\alpha_h^{\vee})=0$ $\forall \ h\in H\cap\mathcal{I}^0$.
			\item[(H3)] $V_{\lambda-\sum\limits_{h\in H}m_H(h)\alpha_h}\ = \ \bigg(\prod\limits_{h\in H}f_h^{m_H(h)}\bigg)V_{\lambda}\ = \ \{0\}$; \quad so, $\lambda-\sum_{h\in H}m_H(h)\alpha_h \notin \wt V$.
                        \end{itemize}
      [b] $\mathrm{Indep}(J_{\lambda})\ :=\ \big\{(H, m_H)\ \big|\ H, m_H   \textnormal{ satisfy (H1) and (H2)} \big\}$.
     \ \  
            [b$'$] \underline{Hole-set} : Any $\mathcal{H}\subseteq \mathrm{Indep}(J_{\lambda})$.
            \smallskip\newline
            [c] \underline{The set of minimal holes, in a hole-set $\mathcal{H}\subseteq \mathrm{Indep}(J_{\lambda})$} : \\ 
 $\mathcal{H}^{\min}:=\big\{ (H, m_H)\in \mathcal{H}\ \big|\ \text{ there is no }\ (H,m_H) \neq (J,m_J)\in \mathcal{H}\text{ with } m_J(j)\leq m_H(j)\ \forall\ j\in J\subseteq H \big\}$.\smallskip\newline
[d] \underline{Nice hole-sets $\mathcal{H}$} :\ \  Hole-set $\mathcal{H}$ with \ \ $(H,m_H)\in \mathcal{H}^{\min}\ \ \implies  \begin{cases}
		\text{either}\ H\subseteq \mathcal{I^-}\sqcup\mathcal{I}^0 \ (\text{of any size}),\\
	 \text{or}\ H\subseteq \mathcal{I}^+ \text{ and }|H|=1.		\end{cases}$         
             \end{definition}
            \begin{note}\label{Note imaginary reflections in hole-weihts}
               Consider $\Big(\prod_{h\in H}f_h^{m_H(h)}\Big)M(\lambda)_{\lambda}$ for $(H, m_H)$ as in (H1) and (H2).
               It is maximal, and when $\mathcal{I}=\mathcal{I}^+\sqcup \mathcal{I}^-$, its weight is the ``dot-conjugate of $\lambda$ by $s_H:=\prod_{h\in H}s_{\alpha_h}$'' (the product of reflections $s_{\alpha_h}$,  possibly $A_{hh}<0$).
               For $\lambda\in \mathfrak{h}^*$, $\mathrm{Indep}(J_{\lambda})= \mathcal{H}_{L(\lambda)}\sqcup \{(H, 
               0\text{-map})\ |\ H \text{ as in } (H1)\}$. Moreover, $\mathcal{H}_{L(\lambda)}$ is nice (Definition \ref{Defn holes in BKM setting}[d]), and $\mathcal{H}_{L(\lambda)}^{\min}\ = \ \Big\{ \Big(\{i\},\ 1+\frac{2}{A_{ii}}\lambda(\alpha_i^{\vee})\Big)\ \Big|\ i\in J_{\lambda}\cap(\mathcal{I}^+\sqcup\mathcal{I}^-) \Big\}$ $\sqcup \big\{\big(\{i\}, \ 1\big)\ \big|\ i\in J_{\lambda}\cap \mathcal{I}^0 \big\}$. 
     Due to the lack of slice-formulas \eqref{Eqn wt-formula for all simple in KM setting}, which makes our proofs interestingly subtler, we prove Theorem~\ref{Theorem weight formula for nice modules} (in Subsection \ref{Subsection Theorem A proof}) directly for all $V$ in its hypothesis, and not first for simples.
     This deviates from the standard approach of going initially from simples to any $V$ which featured in the works of Khare et al (which is not applicable in the BKM case).     
            \end{note}
	\begin{thmx} \label{Theorem weight formula for nice modules}
		Fix any BKM LA  $\mathfrak{g}$ and $\lambda\in\mathfrak{h}^*$. Then the weight-set of $\mathfrak{g}$-h.w.m. $0\neq V \twoheadleftarrow M(\lambda)$ with holes $\mathcal{H}_V$ {\it nice} \big(Definition \ref{Defn holes in BKM setting} [d]\big) -- including $V=$ any simple $L(\lambda)$ for $\lambda\in \mathfrak{h}^*$ is given by
	\begin{equation}\label{Weight formula for nice modules}
		\wt V \ = \ W_{I_V} \big\{\mu \preceq \lambda\ |\ 
		\mu \text{ satisfies conditions } (C1) \text{ and }(C2) \text{  below}\big\}.
		\end{equation}
        \begin{itemize}
			\item[(C1)] $\mu(\alpha_i^{\vee})\in \mathbb{Z}_{\geq 0}$ $\forall\ i\in I_V$; see Definition \ref{Eqn Integrability}.\ \ \ \ \big(Note $I_V=\{i\in \mathcal{I}^+\ |\ \big(\{i\},\lambda(\alpha_i^{\vee})+1\big)\in \mathcal{H}_V\}$.\big) 
			\item[(C2)] Suppose $\mu=\lambda-\sum_{j\in J}c_j\alpha_j\precneqq\lambda$ \ \ \ \ for  $J=\supp(\lambda-\mu)\neq \emptyset$ and $c_j\in \mathbb{Z}_{>0}$ $\forall$ $j\in J$. \\
			Let $J=J_1\sqcup\cdots\sqcup J_n$ \ \ \ \ be the decomposition into (connected) Dynkin subgraph components.\\
			There exist nodes $j_k\in J_k$ for each $1\leq k\leq n$, \ \ \ \  s.t. the following element $\mu'$ lies in $\wt V$.
			\begin{equation}\label{Eqn canonical weight for mu}
			\underset{\big(\text{defined for }\mu\ \text{or} \ (c_j)_{j\in J}\big)}{\mu'}\ := \ \lambda-\sum_{k \text{ s.t. }|J_k|=1}c_{j_k}\alpha_{j_k} - \sum_{k \text{ s.t. }|J_k|>1}\alpha_{j_k}\ \ \in\ \wt V.
			\end{equation}
           \end{itemize}
       
	\end{thmx}
    After integrable simples, the well-studied h.w.m.s are parabolic Verma modules $M(\lambda,J)$, defined by parabolic induction:\ 
For $\lambda\in \mathfrak{h}^* \text{ and } \ J\subseteq J_{\lambda}\cap \mathcal{I}^+$,  
$M(\lambda, J) :=U(\mathfrak{g})\otimes_{U(\mathfrak{p}_J)} L_J(\lambda)$ \big(\cite{Lepo parabolics, GaLe, KuBGG, Naito BGG 2}\big).
Up to over symmetrizble KM/BKM $\mathfrak{g}$,  $M(\lambda, J)=  \frac{M(\lambda)}{\Big\langle f_j^{\lambda(\alpha_j^{\vee})+1}M(\lambda)_{\lambda}\ \Big| \ j\in J \Big\rangle}$.
\begin{cor}
(1) Theorem \ref{Theorem weight formula for nice modules} yields weight-sets of all quotients of parabolic Verma $\mathfrak{g}$-modules $V'=\frac{M(\lambda, J)}{N_{V'}}$, for submodules $N_{V'}\subset M(\lambda ,J)$ that are generated by weight vectors made of imaginary $f_i$.
 In particular, we found weights of all h.w.m.s $V$ over $\mathfrak{g}(A)$ when all $A_{i,j}\leq 0$, i.e., when $\mathcal{I}=\mathcal{I}^-\sqcup \mathcal{I}^0$; such settings of $A$ were notably studied in the literature \cite{Elizabeth, Duchamp}.
\smallskip\\ (2) Theorem \ref{Theorem weight formula for nice modules} reveals (parabolic Verma type / minimal $\mathcal{I}^+$-holes singletoned) higher order Verma $\mathbb{M}(\lambda, \mathcal{H})$ for nice $\mathcal{H}$,  and their weights.
For $\lambda\in \mathfrak{h}^*$ and hole-set $\mathcal{H}$ (Definition \ref{Defn holes in BKM setting}[b$'$]), we define:
        \begin{align}\label{Eqn defn higher order Vermas}
  \mathbb{M}(\lambda,\ \mathcal{H})\ \ :=\ \   \frac{M(\lambda)}{\left\langle \prod\limits_{h\in H}f_h^{m_H(h)}M(\lambda)_{\lambda}\ \bigg|\ (H, m_H)\in \mathcal{H} \right\rangle}\qquad
  \bigg(\begin{aligned}[] \mathbb{M}(\lambda,\ \mathcal{H})\ = \ \mathbb{M}(\lambda,\ \mathcal{J})\\
  \forall\ \ \mathcal{H}^{\min} \subseteq \mathcal{J}\subseteq \mathcal{H}
  \end{aligned}\bigg).
  \end{align}
 \end{cor}
 \begin{example}[Universality of higher order Vermas \& holes]
(a) $\mathcal{H}_{M(\lambda)}= \emptyset$. $\mathcal{H}_{V=0}$ and $\mathcal{H}_{L(\lambda)}$ equal the \textit{upper closures} of $\{\ \emptyset\ \}$  and resp. $\{\ \{j\}\ |\ j\in J_{\lambda} \}$.\ \  
(b) $0^{\text{th}}$-order: $\mathbb{M}(\lambda, \mathcal{H}=\emptyset)=M(\lambda)$.
(c) $1^{\text{st}}$-order: $\mathbb{M}\big(\lambda, \ 
\{\ \{j\}\ |\  j\in \mathcal{I} \} \big)=L(\lambda)$ for $\lambda\in P^+$ up to symmetrizable BKM $\mathfrak{g}$, and the parabolic Verma module $M(\lambda, J)$ (with h.w. $\lambda$ and \ integrability $J$\big) equals $\mathbb{M}\big(\lambda,\ \{\ \{j\} \ |\ j\in J  \} \big)$. \end{example}
	\begin{cor}\label{Corollary to Theorem A}
 For any simple $L(\lambda)$, $\lambda\in \mathfrak{h}^*$, (C2) for \eqref{Weight formula for nice modules} is equivalent to non-degeneracy \eqref{Eqn Weyl orbit wt-formula in KM setting}!
 Namely, $\wt L(\lambda)$ is the set of all $\mu\preceq \lambda$ that satisfy the following conditions:
 \begin{itemize}
     \item $\mu \text{ satisfies (C1) w.r.t. } I_{L(\lambda)}$.
     \item $J_t \text{ is a component in }J_{\lambda}       \text{as in (C2) }\implies 
                  \exists \ j_t\in J_t \text{ s.t. }         \lambda  (\alpha_{j_t}^{\vee})\neq 0$.
                  \item $\{j_t\}=J_t \implies  j_t\in  \mathcal{I}^+\sqcup \mathcal{I}^- \text{ and }\lambda- \Big(\textstyle{\frac{2}{A_{j_t j_t}}}\lambda\big(\alpha_{j_t}^{\vee}\big)+1\Big)\alpha_{j_t}\nsucc \mu$.
 \end{itemize}
                This formula seems un-written for $L(\lambda\in P^+)$s over BKM $\mathfrak{g}$; not in this basic form for non-integrable $L(\lambda)$s and for parabolic Vermas even in finite type. 
Theorem \ref{Theorem weight formula for nice modules} yields (distinct) weight-sets for all integrable h.w.m.s in the poset for each $\lambda\in P^+$ in Definition \ref{Defn integrables}.
	\end{cor}
 In a followup work contained in \cite{BKM ppr}, we extended and proved Theorem \ref{Theorem weight formula for nice modules} for all h.w.m.s $V$.
   
    \begin{thmx}\label{Theorem wt-formula by composition series simples for nice V} 
Fix any BKM $\mathfrak{g}$  and $\lambda\in \mathfrak{h}^*$, and any integrable h.w.m. $M(\lambda)\twoheadrightarrow V$.  \big(So $\mathcal{H}_V$ is nice, and holes in it along $\mathcal{I}^-\sqcup \mathcal{I}^0$-directions can have arbitrary sizes.\big)
Then-                    \begin{align}\label{Eqn wt-fromula via comp. series for V in BKM case}
             \begin{aligned}  \wt  V\  = \ \qquad \quad  \bigcup\limits_{\mathclap{(H, m_H)\ \in\  \mathrm{Indep}(J_{\lambda})\ \setminus \ \mathcal{H}_V}} \qquad \wt L\bigg( \lambda-\sum_{h\in H}m_H(h)\alpha_h\bigg).
            \end{aligned}
\quad 
\left(\begin{aligned}
&\textnormal{As in }\eqref{eqn wt-formula for all V in KM setting}, \textnormal{ these simples occur in}\\
& \text{all local Jordan--H{\" o}lder series of }V  .
\end{aligned}\right)
\end{align}
    \end{thmx}
    Our formulas are : i)~Uniform across all $\mathfrak{g}$'s and nice h.w.m.s.
   ii) Explicit and cancellation-free.
   iii)~Occur as various parabolic $W$-subgroups' orbits of weights, which lie below either $\lambda$ or its maximal 1-dim. dot-conjugates.
  The simple-factors in \eqref{Eqn wt-fromula via comp. series for V in BKM case} vary from those studied in Vermas in \cite{Naito BGG 1, Naito BGG 2}, along imaginary directions.
The analogue of Kac's local Jordan--Hölder series \cite[Lemma 9.6]{Kac book} for $V$ over BKM LAs seems unstudied in literature.
So, to strengthen \eqref{Eqn wt-fromula via comp. series for V in BKM case} following \cite{MDWF} and \cite[Theorem C]{Teja--Khare}, it may be interesting to solve Question \ref{Question composition series simple with a given weight in any V}, using the \underline{key weight $\mu'$} for $\mu$ in \eqref{Eqn canonical weight for mu}.
\begin{question}\label{Question composition series simple with a given weight in any V}
   Which simple in a local composition series of $V$, contains (any) given $\mu\in \wt V$? \end{question} 
   Following \cite{MDWF}, the first step to solve this problem is exploring it in the fundamental case of $V=M(\lambda)$ in BKM setting, for $\lambda$ in $P^+$ or in $P^{\pm}$.
   Our Theorem \ref{Theorem wt-formula by composition series simples for nice V} explores it for integrable (non-simple) $V$'s, and we aim to completely solve this problem for all $V$ in a future work.
\subsection{On Vermas $M(\lambda\in P^{\pm})$, characters of simples $L(\lambda\in P^{\pm})$ and $L(\rho)$}\label{Subsection character for L(lambda nice)}
From $P^{\pm}$-perspective, simple $L(\lambda\in P^+)$s are -- when thought of as modules with fewer singleton-ed real hole-relations -- parabolic or nice higher order Verma typed, possessing WKB character formulas.
Indeed when $\mathcal{I}^+=\emptyset$ and $\lambda\in P^+$ with $\lambda(\alpha_i^{\vee})>-\frac{A_{ii}}{2}$ $\forall$ $i$ (so $J_{\lambda}^c\neq \emptyset$) one sees by Kac--Kazhdan equation \eqref{Eqn dot action by any positive root} below, that $L(\lambda)=M(\lambda)$.
 Our Theorems \ref{Theorem C} and \ref{Theorem D character of V(rho)} study by-products $M(\lambda)\twoheadrightarrow  L(\lambda)$ for $\lambda\in P^{\pm}$ of our weight analysis; concerning which, recall even $\mathrm{char} \big(L(+\rho)\big)$ is not known beyond over symmetrizable case (also see the discussion below Definition \ref{Defn holes in KM setting}).
An interesting case of $J_{\lambda}^c\neq \emptyset$, is $\lambda= -\rho$ with $J_{\rho}^c=\mathcal{I}^+\sqcup \mathcal{I}^-$ discussed in Note \ref{Note -rho}.  
Our primary goals in this subsection are:
\begin{question}\label{Question characters of nice simples}
Do $L(\lambda)$ for $\lambda\in P^{\pm}$ possess WKB character type formulas? 
If so, what would be the analogue of the Weyl group and the length function in them?
\end{question}
Recently, in \cite[Theorems E--G and Proposition 2.14]{Teja--Khare}, the authors initiated the study of higher order Verma modules $\mathbb{M}(\lambda, \mathcal{H})$ -- whose yield all $\wt V$'s --  and their Weyl--Kac type characters, using~: i) Semigroups $W_{\mathcal{H}}$ to enumerate character numerators; $\text{e.g. in } A_1\times\cdots \times A_1$ types
\begin{equation}\label{Eqn Weyl semigroup}
 (W_{\mathcal{H}}, \star ) := \Big\{s_{H_1}\star\cdots \star s_{H_r}:= s_{H_1\cup\cdots \cup H_r}\ \Big|\ H_1,\ldots,H_r\in\mathcal{H}^{\min} \Big\}\ \  \subseteq W.
\end{equation}
ii) Positive solutions to Problem \ref{Question dot orbit = norm equality sol.?} above in $A_3$-type, for the two $2^{\text{nd}}$-order Vermas $\mathbb{M}(0, \mathcal{H})$'s.

We now explore such formulas for $L(\lambda\in P^{\pm})$ over rank-2 BKM LAs (Theorem \ref{Theorem C}), and notably for characters of every $V\twoheadleftarrow M(\rho)$ over $\mathfrak{g}\big(A(n)\big)$, addressing Question~\ref{Question characters of nice simples} and Problem \ref{Question dot orbit = norm equality sol.?}.

Hereafter, let $A$ be symmetrizable. 
  We resort to the following praised technique of Kac of working with a decisive central Casimir element, 
employed in WK and WKB character formulas for integrables;
 to bypass (unavailable) Harish-Chandra's theorems for the center of $U(\mathfrak{g})$. 
 \begin{observation}
 (1) By  \cite[Proposition 9.8]{Kac book} and \cite[Proposition 2.70]{Wakimoto}, Question \ref{Question characters of nice simples} reduces to the problem of determining coefficients $c(\mu)$'s in \eqref{Eqn char formula principle 1 from norm equality} for $V=L(\lambda)$ for each $\lambda\in P^{\pm}$.
(2) Problem \ref{Question dot orbit = norm equality sol.?} is a subquestion to them.
It has the following easy partial answer which was helpful in WKB characters' derivation:  $\{\mu\prec \lambda\ |\ P^+\ni \mu \text{ satisfies }\eqref{Eqn norm equality with invariant form}\} = \{\lambda\}$ when $\lambda\in P^+$.
    \end{observation}
We focus on settings (A) and (B) of $\mathfrak{g}(A)$'s, defined in \eqref{Eqn negative settings of A} of our Introduction. 
We fix $\lambda\in P^{\pm}$ below, and set $M_i\ :=\  \frac{2}{A_{ii}}\lambda(\alpha_i^{\vee})  +1\ \in \mathbb{N} \ \ \forall\ i\in \mathcal{I}^-\sqcup \mathcal{I}^+,\  \text{with} \ M_i=1 \ \forall\ i\in \mathcal{I}^0$.

Over $\mathfrak{g}\big(A(b,a,a,b)\big)$ for $a,b\in \mathbb{N}$,\ \ \ 
    $\mu \ =\  \lambda-X\alpha_1-Y\alpha_2\ \ \text{ for } \ (X, Y)\in \mathbb{Z}_{\geq 0}^2 \  \text{ satisfies }\ \eqref{Eqn norm equality with invariant form}$ \ iff 
\begin{equation}\label{Eqn symmetric and same lengths case and any a}
\vspace*{-0.5cm}
			\ X^2\ + \ Y^2\ - \  M_1 X\ - \ M_2Y \ +\ \frac{2a}{b}XY\ = \ 0.
\end{equation}
\smallskip \\
Over $\mathfrak{g}\big(A(n)\big)$, and for $\lambda=\rho, \ \mu=\rho -\sum_{i=1}^n X_i\alpha_i$ \big(analogous to \eqref{Eqn symmetric and same lengths case and any a}\big)  \eqref{Eqn norm equality with invariant form} simplifies as
\begin{align}
           X_1^2 + \cdots + X_n^2 - 2(X_1+\cdots + X_n)+ (X_1X_2+\cdots  X_iX_{i+1}\cdots +X_{n-1}X_n) = 0 \label{Eqn norm equality for rho 1} \ \ \ \iff \\ (X_1-1)^2 + (X_1+X_2-1)^2+\cdots+ (X_{n-1}+X_n-1)^2 + (X_n-1)^2
-(n+1)\ =0.\hspace*{0.6cm}     \label{Eqn norm equality for rho 2}     
\end{align}\allowdisplaybreaks
\begin{note}\label{Note Easy holes' maximal vectors} 
$f_i^{M_i}M(\lambda)_{\lambda}= M(\lambda)_{\lambda-M_i\alpha_i}$ are maximal. 
And $(0,..,0,M_i,0,..0)$ satisfies \eqref{Eqn symmetric and same lengths case and any a}--\eqref{Eqn norm equality for rho 2}. 
Their solutions in $\mathbb{Z}_{\geq 0}^n$ are finitely many by Lemmas \ref{Lemma bddness of sols for two negative nodes}, \ref{Lemma negative type A norm equality equaion}; we shall note such parallels to $A_2$ case in due course.
 $\mathrm{char}\big(L(\lambda\in P^{\pm})\big)$s and $\mathrm{char}\big(L(\rho)\big)$ lead us to the journey of studying solutions' properties.
\end{note}
 \begin{thmx}\label{Theorem C}
				Fix any $\mathfrak{g}\big(A(b,a,a,b)\big)$, $a,b\in \mathbb{N}$, and $\lambda\in  P^{\pm}$ with powers $M_1,M_2\in \mathbb{N}$ as in \eqref{Eqn symmetric and same lengths case and any a}.
                Then $L(\lambda)$ has the following character formulas in settings (I)--(III) (in view of  Lemmas \ref{Lemma maximal vect in solution (1,n) for case (1, M2)}, \ref{Lemma (2,2) sol. characterization}):\\  
                (I) Assume $(X=1,Y=1)$ satisfies \eqref{Eqn symmetric and same lengths case and any a}.
               
			\begin{equation}\label{Eqn char formula with missing norm equatlity solution in (1,1) case}
			\text{If } \min \{M_1, M_2\} =1,	\qquad\qquad  	\mathrm{char} \big(L(\lambda)\big)\ = \ \frac{e^{\lambda}- e^{\lambda-M_1\alpha_1}-e^{\lambda-M_2\alpha_2}}{R}.\hspace*{3cm} 
					\end{equation}
	\begin{equation}\label{Eqn char formula with non-missing norm equatlity solution in (1,1) case}
\text{If }M_1,M_2\geq 2,\qquad \qquad\qquad\quad   	\mathrm{char} \big(L(\lambda)\big)\ = \ \frac{e^{\lambda}- e^{\lambda-M_1\alpha_1}-e^{\lambda-M_2\alpha_2}-e^{\lambda-\alpha_1-\alpha_2}}{R}.\hspace*{1cm} 
	\end{equation}
    (II) Fix any Weyl vector $\lambda=\rho$; so $M_1=M_2=2$.
    We see three differing cases of solution-sets  in  the ``interior'' $\mathbb{N}\Pi$ to \eqref{Eqn symmetric and same lengths case and any a} :  $\{(1,1)\}$ (iff $a=b$), $\{(1,2), (2,1)\}$ (iff $b=4a$) and $\emptyset$ (in all other cases).
   Then  $\mathrm{char}\big(L(\rho)\big)$ is given by either \eqref{Eqn char formula with non-missing norm equatlity solution in (1,1) case} (iff $a=b$) or \eqref{Eqn char formula with missing norm equatlity solution in (1,1) case} (in all other cases).
    \smallskip\\
    (III) Let  $(X=2,Y=2)$ satisfy \eqref{Eqn symmetric and same lengths case and any a}.
    Then four cases Lemma \ref{Lemma (2,2) sol. characterization}(A)--(D) arise, and the respective numerators of the character formulas in them (ignoring the usual denominator $R$): 
\begin{align*}
\text{Case (A)}\ : & \hspace*{1.5cm}   e^{\lambda}- e^{\lambda-M_1\alpha_1}-e^{\lambda-M_2\alpha_2}-2e^{\lambda-2\alpha_1-2\alpha_2}\\  \text{Case (B)}\ : & \hspace*{1.5cm}  e^{\lambda}- e^{\lambda-M_1\alpha_1}-e^{\lambda-M_2\alpha_2}- e^{\lambda-k_1\alpha_1-\alpha_2} - e^{\lambda-\alpha_1-k_2\alpha_2}-2e^{\lambda-2\alpha_1-2\alpha_2}\\  \text{Cases (C)\ \&\ (D)} \ : & \hspace*{1.5cm}  e^{\lambda}- e^{\lambda-M_1\alpha_1}-e^{\lambda-M_2\alpha_2}
\end{align*}
    \end{thmx}
    \noindent
We elaborate on follow-up questions, which are addressed in Theorem \ref{Theorem D character of V(rho)} in one stroke for $\lambda=\rho$.
\begin{observation}\label{Observation missing solution points in character numerator}
0) Theorem \ref{Theorem C}(I) applies to infinitely many cases: of  any $\lambda\in  P^{\pm} \leftarrow\!\rightarrow (M_1\geq 2,M_2)\in \mathbb{N}^2$ for $A\big(2k,\ (M_1+M_2-2)k,\ (M_1+M_2-2)k,\ 2k\big)$ $\forall$ $k\in \mathbb{N}$.
 1) Full $W\bullet\lambda$ -- satisfying \eqref{Eqn norm equality with invariant form} -- appears in the numerator of \eqref{Eqn WKB character formula}, with $c(\cdot)$-coefficients $\pm 1$ \big(also 0 for possible solutions $\mu\notin W\bullet \lambda$ to \eqref{Eqn norm equality with invariant form} when $\lambda\in P^+$\big).
2) But in \eqref{Eqn char formula with missing norm equatlity solution in (1,1) case}, the exponential $e^{\lambda-\alpha_1-\alpha_2}$ for solution $(1,1)$ to \eqref{Eqn symmetric and same lengths case and any a}, did not appear in the numerator; this happened in (II) at $b=4a$ and in (III) (C)--(D).
Furthermore, $c$-coefs. in (III)(A)--(B) of $e^{\lambda-2\alpha_1-2\alpha_2}$ are $-2$.
These did not arise for $\mathrm{char}(L(\rho))$.
   \end{observation}
 \begin{question}\label{Question for simple's presentations}
    (1) For $A(2)$, can we characterize solutions to \eqref{Eqn symmetric and same lengths case and any a} for all $M_1,M_2$?
  (2)~Following Observation \ref{Observation missing solution points in character numerator}, which solutions of \eqref{Eqn symmetric and same lengths case and any a} appear in the numerator of $\mathrm{char}\big(L(\lambda)\big)$? 
(3) Do $L(\lambda)$s in Theorem \ref{Theorem C} have non CS type relations?
(4) In the spirit of \cite{Naito BGG 1, Naito BGG 2} and \cite{Shushma} (on weights and resp. roots with simple root heights $\leq 1$), it might be interesting to explore characters over higher rank $\mathfrak{g}(A_{\mathcal{I}^-\times \mathcal{I}^-})$ when \eqref{Eqn norm equality with invariant form} has solution-tuples $\big(0,.., 0,1,..,1,0,..,0\big)$ of only 0s and 1s.  
This also motivated our Setting-B ($\lambda=\rho\  \text{ i.e. }  M_i=2\  \forall\ i\in \mathcal{I}=\mathcal{I}^-$), with solution-tuples of 2s and 0s. 
 \end{question}
 Recall \eqref{Eqn norm equality with invariant form} is the necessary condition for $M(\lambda)_{\mu}$ to contain maximal vectors, which our Proposition \ref{Prop maxl vect} focuses on for Question \ref{Question for simple's presentations}(2).  
For sufficiency, we have the Harish-Chandra--BGGs' theory over contragredient $\mathfrak{g}$'s \cite{Kac--Kazhdan, Naito BGG 2}. 
It shows maximal vectors in $M(\lambda)$ to lie within the spaces of weights $\mu=\lambda- \sum_{i=1}^k n_i\beta_i$ in the ``strong-dot orbit of $\lambda$'', where $\beta_1,\ldots, \beta_k\in \Delta^+$ satisfy-
\begin{align}\label{Eqn dot action by any positive root}
\underset{(\text{as referred to in \cite[Theorem 4.2]{Malikov}})}{\textbf{Kac--Kazhdan equation :}} \qquad  \bigg( \lambda+\rho-\sum_{t=1}^i n_t\beta_t\ ,\ 2\beta_{i+1}\bigg) = n_{i+1}(\beta_{i+1}, \beta_{i+1})\ \ \forall\ i. 
\end{align}
\begin{note}\label{NOte KK maxl vect bounds}
a) Let $\mu=\lambda-\beta_1$ be the unique solution to \eqref{Eqn dot action by any positive root} in its ``interval'' $\mu\preceq \cdots \precneqq \lambda$.
Then \cite[Proposition 4.1]{Kac--Kazhdan} estimates maximal vectors' dimensions in $M(\lambda)_{\mu}$ to be  $=\sum_{t\in \mathbb{N}}\dim \mathfrak{g}_{\frac{\beta_1}{t}}$, and asks for their constructions (it seems to be unresolved even when $\lambda-\mu$ is just $\beta_1$).
a$'$) For non-minimal $\mu$'s, it is a lower bound.
b) For $\mathcal{I}=\mathcal{I}^-$ in rank-2 (by the freeness of $\mathfrak{n}^-$), any solution $\mu$ to \eqref{Eqn norm equality with invariant form} \big(by step \eqref{Eqn norm equality simplified}\big) satisfies \eqref{Eqn dot action by any positive root}.
b$'$) Following b), our Proposition \ref{Prop maxl vect} constructs  maximal vectors for non-minimal $\mu$'s in cases below, and prove the attainment of the count-bounds in a$'$).
\end{note}
\begin{prop}\label{Prop maxl vect}
Fix any $\mathfrak{g}, \lambda, M_1,M_2$ as in Theorem \ref{Theorem C}. 
We explicitly compute the maximal vectors in the weight spaces in $M(\lambda)$ of the following solutions $(X,Y)$ to \eqref{Eqn symmetric and same lengths case and any a}, in order to show- \smallskip\\
 (a) ``$A_{-2}$-theory'' : $(X,Y)=\big(M_1,\ n = M_2-\frac{2a}{b}M_1 \in \mathbb{Z}_{\geq 0}\big)$  $\implies$ $f_2^n f_1^{M_1} M(\lambda)_{\lambda}$ are maximal.\smallskip\\ 
(b) $(X,Y) = (1,n)\text{ or } (n, 1)$  $\implies$  $M(\lambda)_{\lambda-X\alpha_1-Y\alpha_2}$ has a unique maximal vector (up to scalars). 
(c) $(X,Y)=(2,2)$ $\implies$ $M(\lambda)_{\lambda-X\alpha_1-Y\alpha_2}$ has 2  linearly independent maximal vectors.\\
In (b)--(c), these counts -- of $\mathfrak{g}$-homomorphisms of Vermas, i.e. $\dim \mathrm{Hom}_{\mathfrak{g}}$ $\big(M(\lambda-X\alpha_1-Y\alpha_2),\ M(\lambda)\big)$ -- equal $\dim \mathfrak{g}_{\alpha_1+n\alpha_2}=1$ and resp. $\dim \mathfrak{g}_{\alpha_1+\alpha_2}+ \dim \mathfrak{g}_{2\alpha_1+2\alpha_2} =2$; as mentioned in Note\ref{NOte KK maxl vect bounds}b$'$).
\end{prop}

\noindent (a) is reminiscent of the usual monomial type construction of maximal vectors, e.g. $f_2^{[0\boldsymbol{-}\alpha_1](\alpha_2^{\vee})+1} f_1\\   M(0)_0$ in $A_2$ case; see Lemma \ref{Lemma maximal vect in solution (1,n) for case (1, M2)} and Note \ref{Note negative Sl3-theory, string reversal}.
\big(For examples of cases in (a), see Proposition \ref{Prop number theory}.\big)
This procedure does not help in proving parts (b)--(c), and we solve equations on (Poincaré--Birkhoff--Witt) ordered words on Lyndon root basis. 
The seminal work in \cite{Malikov} constructs maximal vectors in affine cases, using complex powers of root vectors.
Concerning part (c), our follow-up work to \cite{BKM ppr} shows maximal vectors' counts for solutions $ (X,Y)= (2,n)$ $\forall$ $n\geq 2$, to be $\lceil \frac{n+1}{2}\rceil$.
\begin{cor}\label{Corollary presentations gfor nice simples}
   In Theorem \ref{Theorem C}, for (I)\eqref{Eqn char formula with missing norm equatlity solution in (1,1) case}, (II) at $a\neq b$, and (III)(C)--(D), $L(\lambda)$ have expected CS-type presentations.
   And the 4 remaining cases, (I)\eqref{Eqn char formula with non-missing norm equatlity solution in (1,1) case}, (II) at $a=b$ and (III)(A)--(B), involve all the maximal vectors computed above.
   E.g. case (I)\eqref{Eqn char formula with non-missing norm equatlity solution in (1,1) case} negatively solves Question \ref{Question for simple's presentations}(3) by:
   \begin{equation}\label{Eqn non CS presentation in rank 2}
        L(\lambda)\  = \ \frac{M(\lambda)}{\left\langle f_j^{\lambda-M_j\alpha_j}m_{\lambda} \ , \  2af_1f_2m_{\lambda} +b(M_2-1)[f_2,f_1]m_{\lambda}  \ \Big|\ j\in \{1,2\}   \right\rangle} .   \end{equation}
\end{cor}
For $\lambda\in P^{\pm}$ and solutions $\mu$ in Proposition \ref{Prop maxl vect} of Equation \eqref{Eqn norm equality with invariant form}, $M(\lambda)_{\mu}$ has at least one maximal vector; supporting Problem \ref{Question dot orbit = norm equality sol.?}(a). 
 This is not true in general, by Theorem \ref{Theorem D character of V(rho)} for $\lambda=\rho$; see Observation \ref{Observation solutions with no maximal vectors in M(rho)} in Section \ref{Section 6 for rho}.
\begin{question}\label{Question maxl vect appearing in presentations}
 (1) To narrow down the exponential solution-powers in $\mathrm{char}(V)$, can principle \eqref{Eqn char formula principle 1 from norm equality} (for all h.w.m.s) be strengthened with solutions of \eqref{Eqn dot action by any positive root} instead of \eqref{Eqn norm equality with invariant form}? 
(2) Concerning maximal vectors, suppose $\mu$ is a minimal solution to \eqref{Eqn norm equality with invariant form}; i.e., $\mu\not\prec \mu'$ for any other solution $\mu'$ to \eqref{Eqn norm equality with invariant form}. 
When $\mathcal{I}=\mathcal{I}^-$, do the non-CS-type maximal vectors in $M(\lambda)_{\mu}$ -- similar to $(1,1)$-weighted vector in \eqref{Eqn non CS presentation in rank 2} -- appear in presentations of $L(\lambda\in P^{\pm})$?
(3) Does its answer address Question \ref{Question for simple's presentations}(2)?
 \end{question}
We will re-visit Question \ref{Question maxl vect appearing in presentations} (2)--(3).
Recall Note \ref{Note all h.w.m.s over -ve An case are higher order Vermas} on higher order Vermas. 
Below, when $\mathcal{I}=\mathcal{I}^-$, we omit the power maps in holes which are canonical (similar to in KM cases of $\mathcal{I}=\mathcal{I}^+$).  
 \begin{thmx}\label{Theorem D character of V(rho)}
     Let $A(n)$ be the negative $A_n$-type BKMC matrix in Setting-B, $\mathcal{I}=\{1,\ldots,n\}$ and $\lambda=\rho$ the Weyl vector (unique in this case). 
     Fix any h.w.m. $M(\rho)\twoheadrightarrow V$ over $\mathfrak{g}\big(A(n)\big)$.\\
(a) Let $\mu=\rho-\sum_{i=1}^nX_i\alpha_i$ be a solution to \eqref{Eqn norm equality for rho 1} or \eqref{Eqn norm equality for rho 2}.
Then $V_{\mu}$ contains a non-zero maximal vector, iff $(X_1,\ldots,X_n)$ has an independent support and has the form
\begin{equation}\label{Eqn solutions of os and 2s for rho in type A} 
       (.., 0, \underset{i_1}{2}, 0,.., 0,  \underset{i_2}{2}, 0,.., 0,\underset{i_k}{2}, 0,.. )\qquad \text{for non-adjacent }\ 1\leq i_1< \cdots  <i_k\leq n.
       \end{equation}
       (b) $V=\mathbb{M}(\rho , \mathcal{H}_V)$.
       So every h.w.m. $V$ (not only simples) here admit explicit presentations by their minimal holes. 
       In particular, $\displaystyle L(\rho)=\frac{M(\rho)}{\left\langle 
f_1^2M(\rho)_{\rho},\ldots, f_n^2M(\rho)_{\rho} \right\rangle}$ is parabolic Verma typed.\\
       (c) $\mathrm{char}\big(L(\rho)\big)\ =\  \mathrm{char}\big(\mathbb{M}\big(\rho, \ \big\{  
\{1\},\ldots, \{n\} \big\} \big)\big)$, and more generally, we have
\begin{equation}\label{Eqn char V is char of M(rho, HV)}
\mathrm{char}\big(V\big)\ = \ \mathrm{char}\big(\mathbb{M}\big(\rho, \ \mathcal{H}_V\big)\big)\ = \ \frac{\sum\limits_{\mu \text{ as in }\eqref{Eqn solutions of os and 2s for rho in type A}}\sum\limits_{\substack{H_1,\ldots,H_r\in \mathcal{H}_V^{\min} \text{ s.t. }\\ H_1\cup\cdots \cup H_r=\supp(\rho-\mu)} } (-1)^r e^{\mu}}{R}.    
\end{equation}
\end{thmx} 
Formula \eqref{Eqn char V is char of M(rho, HV)} for all $V\twoheadleftarrow M(\rho)$, resembles WKB characters  \eqref{Eqn WKB character formula} (for $\mathcal{I}=\mathcal{I}^-$): both run over all the unions of minimal holes in the h.w.m.s in them!  
It also resembles \cite[Formula (7.20)]{Teja--Khare},  upon using appropriate ``negative-Weyl semigroup''.
 Theorem \ref{Theorem D character of V(rho)} is proved in Section \ref{Section 6 for rho}.
 We turn to minimal solutions of \eqref{Eqn norm equality with invariant form} in Question \ref{Question maxl vect appearing in presentations}(2), motivated by:
 i) The dot-orbits in Lemma \ref{Lemma (2,2) sol. characterization}(D), where no solution in the positive cone $\mathbb{N}\Pi$ appears in the character numerator or in presentation denominator for $L(\lambda)$. 
 ii) A similar absence in \eqref{Eqn char V is char of M(rho, HV)}, of a majority of solutions outside \eqref{Eqn solutions of os and 2s for rho in type A}; they are also not minimal by characterization Lemma \ref{Lemma at least 2 blocks in d-solutions}.
 At the same time, Equation \eqref{Eqn char V is char of M(rho, HV)} for $L(\rho)$, and even \eqref{Eqn WKB character formula} for $L(\lambda\in P^+)$,  contains non-minimal solutions to \eqref{Eqn norm equality for rho 2} \big(e.g. $(2,0,...,0,2)$\big) with non-adjacent 2s. This phenomenon did not take place for $\lambda=\rho$ over $\mathfrak{g}\big(A(2)\big)$.

We digress and recall \cite[Theorem II]{Naito 2}, which extends to BKM setting the Jantzen's character sum formula.
It was explored for the quotients of Verma $M(\lambda)$ by its submodules that are isomorphic to $M(\lambda-\beta)$ for domestic imaginary roots (\cite{Wakimoto}) $\beta\in W \Pi_{\mathcal{I}^-}$, when $A_{ij}\neq 0$ $\forall$ $i,j\in \mathcal{I}=\mathcal{I}^+\sqcup \mathcal{I}^-$.
\begin{observation}\label{Note KK Eqn unique sol section 3}
The simplest of aforementioned quotients, is the unique solution $\lambda-\beta_1$ case quoted in Note \ref{Note KK unique solution} a).
Now, as a first step towards Question \ref{Question maxl vect appearing in presentations}, we explore in $A(b,a,c,d)$ settings unique solutions $\mu=\lambda-\beta_1$ of \eqref{Eqn norm equality with invariant form} inside the positive cone $\lambda-\mathbb{N}\Pi$.
Observe that \eqref{Eqn norm equality with invariant form} always has two solutions $\lambda-M_i\alpha_i$; we study it when it has totally three solutions.
  In view of shortcomings in Note \ref{Note KK unique solution} b) (and to showcase the non-triviality of this problem), we provide a complete number-theoretic characterization for the desired $\mu$'s in Proposition \ref{Prop number theory} over $\mathfrak{g}\big(A(2)\big)$:
 \end{observation}
   \begin{prop}\label{Prop number theory}
Fix (any) $M_1\leqslant  M_2 \in \mathbb{N}$ with $d = \text{gcd}(M_1,M_2)$.
Then in $\mathbb{N} \times \mathbb{N}$, the equation $X^2 + Y^2 - M_1X - M_2Y + XY = 0$ has a unique solution -- $(X,Y)=$ either $\big(M_1, M_2-M_1\big)$ or $\Big(\frac{2}{3}M_1, \frac{2}{3}M_1\Big)$ resp. when $M_1<M_2$ or $M_1=M_2$ -- iff all the following conditions are satisfied :
\begin{enumerate}
\item The prime factors of $d$ are $2, 3$, or of the form $6k-1$ for some $k \in \mathbb{N}$. 
\item $M_1 = M_2$ $\implies$ $3$ divides $d$.
\item  $M_2 \notin \{M_1, 2M_1\}$ $\implies$ $3$ does not divide $d$ and $\ \dfrac{M_1^2 + M_2^2 - M_1M_2}{d^2}$ is a prime of the form $6m+1$ for some $m \in \mathbb{N}$.
\end{enumerate}
\end{prop}\noindent
Among early cases of (unique) solutions $(M_1,\ M_2-M_1>0)$ in Proposition \ref{Prop number theory}, we see infinitely many examples $(M_1=1, M_2-1\in \mathbb{N})$ for $M_2=4,6,7,9,13,15,16,...$
This could perhaps be the first interesting rank-2 setting, along $\mathcal{I}^-$-directions and also for non-domestic/alien imaginary $\beta = M_1\alpha_1+(M_2-M_1)\alpha_2\in \Delta^+$, to explore the Jantzen's character sum formulas for aforementioned quotients of $M(\lambda\in P^{\pm})$ by $M(\lambda-\beta)$.
 We conclude this section with the following question.
\begin{question} Can one characterize $\big(a,b,c,d \ ; \ M_1, M_2 \big)$, for which \eqref{Eqn symmetric and same lengths case and any a} has a unique solution in $\mathbb{N}\times \mathbb{N}$; extending Proposition \ref{Prop number theory} which was for $\big(1,2,1,2 \ ; \ M_1, M_2 \big)$?\end{question} 
  \subsection{Organization of the paper} 
1) In Section \ref{Section 4}: Subsection \ref{Subsection 2.1} lists some important notations. 
Subsection \ref{Subsection 2.2} studies root strings of real and imaginary roots through weights of h.w.m.s $V$, and concludes showing the failure of slice-formulas (key players in \cite{Khare_Ad, Teja--Khare}) in BKM case.
Subsection \ref{Subsection 2.3} extends to BKM case, the ``full weights if no holes'' result of \cite{Teja--Khare}, using freeness $J_{\lambda}^c$ in $\wt V$.\\
2) In Section \ref{Section 3}:
We build Weyl-orbit type weight-formulas (of integrable simple h.w.m.s $L(\lambda)$, known up to KM case), and 
prove Theorem \ref{Theorem weight formula for nice modules} in Subsection \ref{Subsection Theorem A proof}; extending those formulas for all nice h.w.m.s $V$ (including all $L(\lambda)$ and all parabolic Verma $V$).
The key step in this is all possible Minkowski weight-decompositions of h.w.m.s in Proposition \ref{Proposition Minkowski difference formula} in Subsection \ref{Subsection Minkowski difference wt.-form.} (following \cite{MDWF}). 
The second formula \eqref{Eqn wt-fromula via comp. series for V in BKM case} locates simples $L(\lambda')$ in a local Jordan--H{\"o}lder series of $V$, whose weight-sets patch-up to the whole of $\wt V$.
This completes our study of weight-side.
Our goals after it are, characters of new simples $L(\lambda)$, and maximal vectors and their multiplicities in $M(\lambda)$s, for $\lambda\in P^{\pm}$.
We pursue them in Settings A and B in \eqref{Eqn negative settings of A}, in Sections \ref{Section 5}, \ref{Section 7}  and resp. in Section \ref{Section 6 for rho}.\\
4) In Section \ref{Section 5}: 
We meticulously study properties of solutions to norm equality  \eqref{Eqn norm equality with invariant form} in Setting A, and then focus on those of Kac--Kazhdan equation \eqref{Eqn dot action by any positive root}.
We note parallels to finite type (Lemmas \ref{Lemma bddness of sols for two negative nodes}, \ref{Lemma (2,2) sol. characterization}).
We prove Theorem \ref{Theorem C} and Proposition \ref{Prop maxl vect} -- on $\mathrm{char}L(\lambda\in P^{\pm})$ and resp. maximal vectors in $M(\lambda\in P^{\pm})$ in spaces of solutions in Theorem \ref{Theorem C} -- progressively in Subsections \ref{Subsection Theorem C (I)--(II) proof}--\ref{Subsection Theorem C (III) proof}.\\
5) In Section \ref{Section 6 for rho}: 
We work in Setting B, on h.w.m.s $V$ with h.w. the Weyl vector $\rho\in P^{\pm}\setminus P^+$.
We meticulously study and compare solutions of \eqref{Eqn norm equality with invariant form} and \eqref{Eqn dot action by any positive root} in Proposition \ref{Proposition solutions with maximal vector for L(rho)}.
We prove Theorem \ref{Theorem D character of V(rho)}, for the character of every $V \twoheadleftarrow M(\rho)$.
6) Section \ref{Section 7} proves Proposition \ref{Prop number theory},  determining all $\lambda\in P^{\pm}$ for which \eqref{Eqn dot action by any positive root} has unique solutions in the ``interior'' $\lambda-\mathbb{N}\Pi$.

\subsection*{Acknowledgments}
 We thank A. Khare and R. Venkatesh, for encouraging us to explore weight-formulas over BKM LAs; and Shrawan Kumar for his encouraging comments and for suggesting important references.
 We thank M. Wakimoto for his invaluable feedback on our study of new simple $L(\lambda)s \ (\lambda \in P^{\pm})$.
 We are grateful to Supravat Sarkar for helping us to complete the proof of Proposition \ref{Prop number theory}.
We immensely thank K. Hariram, P. Karmakar and B. Sury for invaluable number-theoretic ideas to study norm-equalities' solutions for character formulas, and for help in sage programming.
This work was mostly done when the first author was an NBHM Postdoc (fellowship Ref. No. 0204/9/2024/R\& D-II/2965) at IISc Bangalore. 
The second author was also supported by an NBHM Postdoctoral fellowship (Ref. No. 0204/16(8)/2022/R\&D-II/11979). 

    \section{Preliminaries on freeness in weight-sets $\wt V$ of h.w.m.s $V$}\label{Section 4}
    \subsection{Notations}\label{Subsection 2.1}
    1) For $S\subseteq\mathbb{C}$ and subsets $B, C$ of a $\mathbb{C}$-vector space, the sets of $S$-linear combinations in  $B$, and of Minkowski sum (resp. differences) of $B \text{ and } C$, used in weight-formulas: 
$SB:=\Big\{\sum_{j=1}^{n}r_jb_j\text{ }\Big|\text{ }n\in \mathbb{N},\text{ } b_j\in B,\ r_j\in S\text{ }\ \forall  \text{ }1\leq j\leq n \Big\}$, $B\pm C:=\{b\pm c\ |\ b \in B,\ c\in C\}$. $x\pm C \text{ denotes } \{x\}\pm C$.
Throughout, we treat the sums in vector spaces over empty sets as 0. 
\\ 2) $[.,.]: \mathfrak{g} \rightarrow \mathfrak{g}$ denotes the Lie bracket in $\mathfrak{g}$ as well as its extended commutator bracket in $U(\mathfrak{g})$. 
The adjoint ``Lie-multiplication'' operator for $x\in \mathfrak{g}$ is denoted $\ad_x(.) = [x,.]:\mathfrak{g}$  $\rightarrow \mathfrak{g}$ \big(or $U(\mathfrak{g})\rightarrow U(\mathfrak{g})$\big). 
We work with the usual $\mathbb{Z}_{\leq 0}\Pi$-gradation (for $\mathfrak{h}$-action) on $U(\mathfrak{n}^-)$.
\\ 
3) For simplicity, we assume $\mathcal{I}$ to be finite and ordered.  
The subsets of real, imaginary, domestic and alien-imaginary roots in $\Delta$ are $\Delta^{re}=W(\pm \Pi_{\mathcal{I}^+})$, $\Delta^{im}$, $W(\pm\Pi_{\mathcal{I}^0\sqcup \mathcal{I}^-})$ and $\Delta^{im}\setminus W(\pm\Pi_{\mathcal{I}^0\sqcup \mathcal{I}^-})$. 
   \\ 4) Root space decomposition : $\mathfrak{g} = \mathfrak{h} \oplus \Big(\bigoplus\limits_{\alpha\in \Delta} \mathfrak{g}_{\alpha}\Big)$ for $\mathfrak{g}_{\alpha}:=\big\{ 
x\in \mathfrak{g}\ \big|\  [h, x]=\alpha(h)x\ \forall\ h\in \mathfrak{h}\big\}$.
 \\ 5) Parabolic analogues : In several proofs in the paper, we deal with the restrictions of $\mathfrak{g}$-h.w.m.s $V$ to BKM Lie subalgebras $\mathfrak{g}_J:=\mathfrak{g}(A_{J\times J})$ for principle submatrices $A_{J\times J}$'s of $A$ for $J\subseteq \mathcal{I}$; or more precisely we work over the parabolic subalgebra $\mathfrak{p}_J\supseteq \mathfrak{g}_J$.
 We fix for $\mathfrak{g}_J$ : Weyl group $W_J:=W_{\mathcal{I}^+\cap J}$; root system $\Delta_J:= \Delta_J^+ \sqcup \Delta_J^-$; simple systems $\Pi_J, \Pi_J^{\vee}$.
   So $\Delta_J^{\pm}=\Delta^{\pm}\cap (\mathbb{Z}\Pi_J)$; and conventions to be noted are $\Pi_{\emptyset}=\Delta_{\emptyset}:=\emptyset$, $\mathfrak{g}_{\emptyset}=\{0\}$, and $W_{\emptyset}=\{1\}$ the trivial subgroup of $W$.  
   \\ 6) Height functions for $J\subseteq\mathcal{I}$ : The $J$-projected height of root sums $x=\sum_{i\in \mathcal{I}}c_i \alpha_i$ for $c_i \in \mathbb{C}$ is $\height_{J}(x):=\sum_{j\in J}c_j$, and $\height_{\mathcal{I}}(.)$ is the usual height. \  7) Weight intervals :	For $ \mu\in\mathfrak{h}^*,\ \alpha\in \Delta, k\in \mathbb{Z}_{\geq0}$,
	\begin{equation}\label{Defn weight intervals}
  \hspace*{0.5cm} \big[\mu-k\alpha,\text{ }\mu\big]:=\{\mu-j\alpha\text{ }\big|\text{ }0\leq j\leq k\}.\qquad (\text{Not to be confused with the Lie bracket in }\mathfrak{g}.)
  \end{equation}
  8) Throughout, maximal vectors are assumed to be non-zero.  
 For $\lambda\in \mathfrak{h}^*\ \  \text{and}\ \  (H, m_H)\in \mathrm{Indep}(J_{\lambda})$, in spirit of involutions $s_H$, we set for convenience $f_{H}^{m_H}:=\ \prod_{h\in H}f_h^{m_H(h)}$.
  \\ 9) To work with holes, we use the standard notion of upper closures $\bar{C}$ for subsets $C$ of the {\it poset} $\big(\mathrm{Indep}(J_{\lambda})\ , \ \preceq \big)$; see Definition \ref{Defn partial ordering on holes}.
  Namely, $\bar{C}:=\{x\in \mathrm{Indep}(J_{\lambda})\ |\ \ c\preceq x \text{ for some }c\in C \}$.
\subsection{Root strings through weights of $V$}    \label{Subsection 2.2}
In several proofs of the paper, we will need following Lemma \ref{Lemma real root strings}, and notably its stronger version Lemma \ref{Lemma imaginary root strings}, for $e_i$-raising actions on weight spaces, and for (simple) root strings through weights in $\mathfrak{g}$-h.w.m.s.
    \begin{lemma}\label{Lemma reaching lambda from mu}
	 Fix a BKM LA $\mathfrak{g}$, $\lambda\in\mathfrak{h}^*$, and $\mathfrak{g}$-h.w.m. $M(\lambda)\twoheadrightarrow V$ and $\mu\in \wt V$ :
  \begin{itemize}
  \item[(a)] $V_{\mu}$ is the $\mathbb{C}$-span of weight vectors $f_{i_1}\ \cdots\  f_{i_n} \cdot V_{\lambda}$\  for $i_1,\ldots, i_n\in\mathcal{I}$ with $\lambda-\mu= \sum_{j=1}^{n}\alpha_{i_j}$.
  \item[(b)] In (a), the partial sums $\lambda-\sum_{t=k}^n\alpha_{i_t} \ \in\  \wt V $ \ $\forall$ $k\leq n$.
  \end{itemize}
	\end{lemma}
 	\begin{lemma}\label{Lemma real root strings}
		Fix the notations of Lemma \ref{Lemma reaching lambda from mu}.
	Let $\lceil \cdot\rceil$ be the ceiling function on $\mathbb{R}$.
		\begin{itemize}
			\item[1)]	For $\alpha\in \Delta^{re}\cap \Delta^+$, \   $\langle\mu,\alpha^{\vee}\rangle\geq  0\ \ \implies\ \ \big[\mu-\lceil\langle\mu,\alpha^{\vee}\rangle\rceil\alpha,\ \mu\big]\subseteq\wt V$.\ \      Particularly  $\langle\mu,\alpha^{\vee}\rangle\in\mathbb{Z}_{\geq 0}$ $\implies s_{\alpha}\mu=\mu-\langle\mu,\alpha^{\vee}\rangle\alpha\in \wt V$.
            \big(However $\mu(\alpha^{\vee})\in \mathbb{Z}_{<0}\ \not\!\!\!\implies s_{\alpha}\mu\in \wt V$.\big)
			\item[2)] Analogously for $i\in \mathcal{I}^-$ ($A_{ii}< 0$)\quad  $\mu(\alpha_i^{\vee})\leq 0\  \implies  \  \displaystyle \Big[\underbrace{\mu-\left\lceil\frac{2}{A_{ii}}\langle\mu,\alpha_i^{\vee}\rangle\right\rceil\alpha_i}_{=\ s_{\alpha_i}(\mu) \text{, as in Note }\ref{Note imaginary reflections in hole-weihts}},\ \mu\Big]\subseteq\wt V$.
			\item[3)] Finally for $i\in \mathcal{I}^0$ ($A_{ii}=0$)\qquad  $\mu(\alpha_i)\neq 0 \ \ \implies\ \ \mu -\mathbb{Z}_{\geq 0}\alpha_i\ \subseteq\ \wt V$.
            \end{itemize}
			\end{lemma}
\begin{proof}[{\bf \textnormal Proof}]
We begin by fixing $n\in \mathbb{Z}_{\geq 0}$ with $\mathfrak{g}_{\alpha}^nV_{\mu}\neq 0$ but $\mathfrak{g}_{\alpha}^{n+1}V_{\mu}=0$ for $\alpha$ (or $\alpha_i$) in the statement.
We direct the reader to \cite[Proof of Lemma 3.1]{WFHWMRS} for a proof of part 1).
For showing parts 2) and 3), one can replacing $\alpha$ in that proof by $\alpha_i\in \Pi_{\mathcal{I}^-\sqcup \mathcal{I}^0}$, with noting the following steps. 
	Suppose $i\in \mathcal{I}^-$.
    Note $\langle\alpha_i,\alpha_i^{\vee}\rangle= A_{ii}<0$, and so  $\langle\mu+n\alpha_i,\alpha_i^{\vee}\rangle= \langle\mu,\alpha_i^{\vee}\rangle\ +\ nA_{ii}\leq 0$.
Proceeding as in cases (i) or (ii) of \cite[Proof of Lemma 3.1]{WFHWMRS}, when $\langle\mu,\alpha_i^{\vee}\rangle +n A_{ii}$ is in or resp. not in $\frac{A_{ii}}{2}\mathbb{Z}_{\geq 0}$, shows the desired result for $i\in \mathcal{I}^-$ as well. 
Finally, suppose $i\in \mathcal{I}^0$.
Note $\langle\mu+n\alpha_i,\alpha_i^{\vee}\rangle= \langle\mu,\alpha_i^{\vee}\rangle\ +\ nA_{ii}= \langle\mu,\alpha_i^{\vee}\rangle\neq 0$.
We see by Lemma \ref{Lemma rank-1 BKM rep. theory} that $\mathbb{C}[f_i]\cdot V_{\mu+n\alpha_i}$ is a direct sum of Verma lines over $\mathfrak{g}_{\{i\}}$.
		\end{proof}
	\begin{lemma}\label{Lemma imaginary root strings}
		Fix the notations of Lemma \ref{Lemma real root strings} (top $\lambda$ any), and imaginary $\alpha_i$ for $i\in \mathcal{I}^-\sqcup\mathcal{I}^0$.
        If $i$ has an edge (in the Dynkin graph of $\mathfrak{g}$) to a node in $\supp(\lambda-\mu)$, then $\mu-\mathbb{Z}_{\geq 0}\alpha_i\subseteq\wt V$.
	\end{lemma}
	\begin{proof}[{\bf \textnormal Proof.}]
		Fix a h.w. vector $v_{\lambda}\neq 0 \in V_{\lambda}$, and $\mu$, $i$ as in the lemma.
		Note $\beta(\alpha_i^{\vee})\leq 0$ $\forall$ $\beta\in \mathbb{Z}_{\geq 0}\Pi$. 
		\begin{equation}\label{Fact e's local nilpotency on U(g)}
        \textbf{Fact \ : }\ \  
	\text{For each }\beta\in \mathbb{Z}_{\leq 0}\Pi,\quad  	\ad_{e_i}^{\circ n} \big(U(\mathfrak{n}^-)_{\beta}\big)\ =\ 0\ \ \ \forall\ n \geq |\height(\beta)|.\hspace*{2.5cm}
		\end{equation}
        Above $U(\mathfrak{n}^-)_{\beta}$ is the $\beta^{\text{th}}$ graded piece of $U(\mathfrak{n}^-)$.
		Assuming the lemma to be false -- i.e., $\mu-\mathbb{Z}_{\geq 0}\alpha_i\nsubseteq \wt V$ -- we show contradiction to \eqref{Fact e's local nilpotency on U(g)}.
		So we fix $\mu'\in \wt V$ with $\mu'\in \mu-\mathbb{Z}_{\geq 0}\alpha_i$ and $\mu'-\alpha_i\notin \wt V$, and vector $Fv_{\lambda}\neq 0\in V_{\mu'}$ for some $F\in (U(\mathfrak{n^-}))_{\mu'-\lambda}$.
		We show to contradict the existence of $\mu'$
		\begin{equation}\label{e1}
		\Big(\ad_{f_i}^{\circ m} \ \circ\  \ad_{e_i}^{\circ m} \Big)(F)\ \cdot v_{\lambda}\ \neq\  0\  \in \ V_{\mu'} \quad \forall \  m\in\mathbb{Z}_{\geq 0},\qquad \text{by mathematical induction}.
		\end{equation}
		The base step $\ad_{f_i}^{\circ 0}(F)v_{\lambda}=Fv_{\lambda}\neq 0$ is clear.
		We prove \eqref{e1} at $m=1$ in the interest of clarity and  as a warm-up to the induction step.
		We observe for $[f_i,F]v_{\lambda}\in V_{\mu'-\alpha_i}=0$ (by the choice of $\mu'$) :
		\begin{align*}
		0=e_i[f_i,F]v_{\lambda}= [e_i,[f_i,F]\!]v_{\lambda}+ \cancel{[f_i,F]e_iv_{\lambda}}^{=0}=  \big([\alpha_i^{\vee} & ,F]+[f_i,[e_i,F]\!]\big)v_{\lambda}.\\
		\text{So }\ V_{\mu'} \ni \ [f_i,[e_i,F]\!]v_{\lambda}\ =  0- [\mu'-\lambda](\alpha_i^{\vee})F v_{\lambda}\ \in \mathbb{R}_{<0} Fv_{\lambda}  \neq 0, &\ \ \text{as }i \text{ has an edge to }\supp(\lambda-\mu').
		\end{align*}
		Induction step: 
		We assume Claim \eqref{e1} to be true at $m\geq 1$.
        We show for the $(m+1)$-th stage - 
		\begin{equation}\label{e2}
		\ad_{f_i}^{\circ (m+1)}\big([e_i,F']\big)v_{\lambda} \neq 0 \in V_{\mu'},\ \ \  \text{for }  F':=\ad_{e_i}^{\circ m}(F)\ \in \big(U(\mathfrak{n}^-)\big)_{\mu'-\lambda+m\alpha_i}.
		\end{equation}
		By the induction hypothesis $\big(\ad_{f_i}^{\circ m}(F')\big)v_{\lambda}\neq 0\in V_{\mu'}$ (so $F'\neq 0$), but $\big(\ad_{f_i}^{\circ (m+1)}(F')\big)v_{\lambda}\in V_{\mu'-\alpha_i} =0$ by the choice of $\mu'$.
		The following calculations similar to those for $m=1$, imply \eqref{e2} :
		\[ 
		0= e_i\big(\ad_{f_i}^{\circ (m+1)}(F')\big)v_{\lambda}=\big[e_i,\big[f_i,\ \ad_{f_i}^{\circ m}(F')\big]\!\big]v_{\lambda}=
		[\alpha_i^{\vee}, \ad_{f_i}^{\circ m}(F')]v_{\lambda}+\big[f_i,[e_i,\ad_{f_i}^{\circ m}(F')\big]\!\big]v_{\lambda}\]
        \begin{align*}
        \begin{aligned}
		&\implies&  \big[f_i,\big[e_i, \ad_{f_i}^{\circ m}(F')\big]\big]v_{\lambda}\ = &\  -[\mu'-\lambda](\alpha_i^{\vee})\ad_{f_i}^{\circ m}(F')v_{\lambda} \ \ \neq 0  \\
		&\implies&  \hspace*{-0.5cm} \big[f_i\big[f_i,\big[e_i, \ad_{f_i}^{\circ (m-1)}(F')\big]\big]\big]v_{\lambda} \ =& \ - [2(\mu'-\lambda)+\alpha_i](\alpha_i^{\vee})\ad_{f_i}^{\circ m}(F')v_{\lambda}  \ \ \neq 0 \ \ & \\
        &\ & \vdots \ & \\
		&\implies &  \hspace*{-0.5cm}\ad_{f_i}^{\circ (m+1)}\big([e_i,F']\big) v_{\lambda}\ = &\  -(m+1) \Big[\mu'-\lambda+\frac{m}{2}\alpha_i\Big](\alpha_i^{\vee})\ad_{f_i}^{\circ m}(F')v_{\lambda}\ \ \neq 0
        \end{aligned} 
       		\end{align*}
	The above non-vanishings are because $i$ has at least one edge to $\supp(\lambda-\mu')$.  \end{proof}
\begin{cor}\label{Corollary no integtrable Slice-decomp.}
By applying Lemma \ref{Lemma imaginary root strings}, we see  below that the branching-typed slice-decompositions as in \eqref{Eqn wt-formula for all simple in KM setting} \cite{Khare_JA, Dhillon_arXiv} fail to hold in BKM case, w.r.t. basic simples $L(\lambda\in P^+)$ or $L(\lambda\in P^{\pm})$  :
    \begin{example}\label{Counter examples for slice-decomp.}
    Fix $\mathcal{I}=\{1,2\}$, $A=A(2)$, the fundamental weights $\{\varpi_1, \varpi_2\}$ dual to $\Pi^{\vee}$.\\
    (a) Let $\lambda = -\varpi_2\in P^{\pm}\setminus P^+$ \big(so $\lambda(\alpha_2^{\vee})=-1$) with the non-$P^+$-dominant direction $\{2\}$.
  One begins to write the skeleton/top-weights $\lambda-\mathbb{Z}_{\geq 0}\alpha_2$ along $\{2\}$-direction for the $\{1\}$-integrable slices -- $\wt L_{\{1\}}(\lambda - k \alpha_2)$ for $k\in \mathbb{Z}_{\geq 0}$ -- but quickly notices $[\lambda-\mathbb{Z}_{\geq 0}\alpha_2\big]\cap \wt L(\lambda)=\{\lambda,\ \lambda-\alpha_2\}$ (by rank-1 theory Lemma \ref{Lemma rank-1 BKM rep. theory}) is not an infinite-ray!
   \big(Note for $L\big(\lambda'=-\frac{1}{2}\varpi_2\big)$ we see the infinite skeleton-ray.)
    So we cannot write slice-formulas for $\wt L(\lambda\in P^{\pm})$ using usual integrable simples.\\
  (b) Let $\lambda= \varpi_1-\varpi_2\in \mathfrak{h}^*\setminus P^{\pm}$; so $\lambda$ is $P^{\pm}$-dominant w.r.t. only $\{2\}$-direction.
  Then the second expected slice $\wt L_{\{2\}}(\lambda-\alpha_1)= \{\lambda-\alpha_1\}$ is finite, but $[\lambda-\alpha_1-\mathbb{Z}_{\geq 0}\alpha_2]\cap \wt L(\lambda) $ is the whole of this ray $\lambda-\alpha_1-\mathbb{Z}_{\geq 0}\alpha_2$; as nodes $1$ and $2$ have edges in-between. \\
(c) \underline{Positive example}: Let $\lambda= \varpi_1\in P^+\setminus P^{\pm}$. 
So $\lambda$ is $P^{\pm}$-dominant w.r.t. only $\{2\}$-direction, and further (for the desired skeleton) $\lambda-\mathbb{Z}_{\geq 0}\alpha_1\ \in \wt L(\lambda)$ by Lemma \ref{Lemma rank-1 BKM rep. theory}.
\[
\wt L_{\{2\}}(\lambda-n\alpha_1)\ \ = \ \ \lambda-n\alpha_1-\mathbb{Z}_{\geq 0}\alpha_2 \ \ = \ \ \big[\lambda-n\alpha_1-\mathbb{Z}_{\geq 0}\alpha_2\big]\cap \wt L(\lambda) \qquad \forall\ n\in \mathbb{Z}_{> 0}.
\]
    \end{example}
\end{cor}	
   For $\lambda\in P^+$, Example \ref{Counter examples for slice-decomp.}(c) showing $\wt L(\lambda)$ to admit slice-decompositions w.r.t. $\wt L(\mu\in P^{\pm})$s for $\mu\in \big[\lambda-\mathbb{Z}_{\geq 0}\Pi_{J_{\lambda}^c}\big]$ \big(where $J_{\lambda}^c=\{i\in \mathcal{I}^0\sqcup \mathcal{I}^-\ \big| \  \lambda(\alpha_i^{\vee})>0 \big\}$\big), supports Question \ref{Question on wt-formulas for simples}(c).

\subsection{Freeness-directions for weights of $V$} \label{Subsection 2.3}
	The largest parabolic Weyl-subgroup of $W$ containing (resp. equaling) the symmetries of any $\mathfrak{g}$-h.w.m. $V \twoheadleftarrow M(\lambda)$ (resp. of $V=L(\lambda)$) is $W_{J_{\lambda}\cap \mathcal{I}^+}$; it leaves $V, \wt V$ and $\mathrm{char}(V)$ all invariant.
I.e., $J_{\lambda}$ is the maximum set of (real and imaginary) integrable directions for all $V\twoheadleftarrow M(\lambda)$.
In the complementary directions $V$ has freeness for its weights, well-studied here for showing our main theorems (and in spirit of \cite{MDWF}) :
	\begin{lemma}\label{Freeness-directions for weights lemma}
		Fix a BKM LA $\mathfrak{g}$, $\lambda\in\mathfrak{h}^*$, and $M(\lambda)\twoheadrightarrow{ }V$. Then-
        \begin{itemize}
        \item $\mu-\sum_{i\in\mathcal{I}\setminus   J_{\lambda}}c_i\alpha_i\in\wt V\text{ for any } \mu\in \wt_{J_{\lambda}}V \text{ and for any numbers }c_i\in \mathbb{Z}_{\geq 0}\ \forall\ i \in J_{\lambda}^c$.
        \item More generally,  for any $i\in J_{\lambda}^c$, we have $\wt_{\mathcal{I}\setminus \{i\}}V\ - \mathbb{Z}_{\geq 0}\alpha_i \  \subseteq \ \wt V$.
        \end{itemize}
	\end{lemma}
    Proof of it is very similar to that of \cite[Theorem 5.1]{Dhillon_arXiv}.
					As an immediate application of Lemma \ref{Freeness-directions for weights lemma} (also extending $J_{\lambda}^c$-freeness) we have the {\bf no holes implies full weights} result for all $V$, extending \cite[Theorem 3.14]{Teja--Khare} to BKM case.
First, a warm-up to Definition \ref{Defn holes in BKM setting} of holes in BKM case :
\begin{observation}
  Let $A=\begin{bmatrix}
                2 & 0 & 0 \\
                0 & 0 & 0 \\
                0 & 0 & -3
            \end{bmatrix}$, $\mathcal{I}^+=\{1\},\ \mathcal{I}^0 =\{2\}, \ \mathcal{I}^-= \{3\}$, and (any) fundamental weights $\varpi_1, \varpi_2, \varpi_3\in \mathfrak{h}^*$ \big($\dim \mathfrak{h}^* = 4$\big).
            Note $\varpi_1+\sqrt{2}\varpi_2+\varpi_3\in P^+\setminus P^{\pm}$ and 
$\lambda=\varpi_1- \frac{3}{2}\varpi_3\in P^{\pm}\setminus  P^+$. 
Since $  \Big(\mathcal{I},\  \underset{{\bf 1}\mapsto 2 ,\ {\bf 2}\ \mapsto \ n,\ {\bf 3}\mapsto 2 }{m_{\mathcal{I}}^{(n)} : \mathcal{I} \rightarrow \mathbb{Z}_{\geq 0}} \Big)\in \mathrm{Indep}(J_{\lambda})\   \forall \ n \in \mathbb{Z}_{\geq 0}$, $\mathrm{Indep}(J_{\lambda})$ is infinite; 
each hole-space $M(\lambda)_{\lambda-n\alpha_2}$ consists of maximal vectors. 
In order to distinguish infinitely many holes along $\mathcal{I}^0$-directions, we included arbitrary power maps on fixed support of each hole; in contrast to KM setting where-in powers in holes are canonical (which is also true in BKM setting but along $\mathcal{I}^{\pm}$-directions) and are omitted.
If $\mathcal{H}_V\ni\Big(H, \underset{H\mapsto \{0\}}{\textbf{0-map}}\Big)$ for $V\twoheadleftarrow M(\lambda)$, then $V=0$.
We denote holes $\Big(\{h\}, \underset{h\mapsto k}{\textbf{k-map}} \Big)$ with singleton supports and holes $\big(H, \textbf{0-map}\big)$, by simply $(\{h\}, k)$ and resp. $(H, 0)$.
\end{observation}
\begin{prop}\label{Proposition no holes full weights}
Fix the notations of Lemma \ref{Freeness-directions for weights lemma}.
Suppose $\mathcal{H}_V=\emptyset$, i.e, $V$ has no holes (note $V$ need not be $M(\lambda)$).
Then $\wt V=\wt M(\lambda)=\lambda-\mathbb{Z}_{\geq 0}\Pi$. More generally,
		\[
		\wt V=\wt M(\lambda)\ \iff \ \mathcal{H}_{V}=\emptyset.
		\]
	\end{prop}
	\begin{proof}[{\bf \textnormal Proof.}]
		 $\wt V=\wt M(\lambda)\implies \mathcal{H}_V=\emptyset$ follows by definitions.
		For the converse, assuming $\mathcal{H}_V=\emptyset$ we show every $\mu=\lambda-\sum_{t\in \mathcal{I}}c_t\alpha_t\in \wt M(\lambda)$ lies in $\wt V$, by inducting on $\mathrm{ht}(\lambda-\mu)=\sum_{t\in \mathcal{I}}c_t\alpha_t\geq 0$.
		Induction step: We assume $\mathrm{ht}(\lambda-\mu)>0$, i.e. $I:=\supp(\lambda-\mu)\neq\emptyset $.
		If $I$ is independent,  (i.e., $\mathrm{dim}M(\lambda)_{\mu}=1$), our assumption of all 1-dim. weight spaces of $M(\lambda)$ being intact in $V$, implies $\mu\in \wt V$.
		So we assume two nodes $a,b\in I$ to be connected by an edge(s) in the Dynkin graph.

We find in each case we proceed in below, a suitable node $i\in I$ with $\mathfrak{g}_{\{i\}}$-action on $V_{\mu+c_i\alpha_i}$ ($\neq 0$ by the induction hypothesis) yields $\mu\in \wt V$. 
Recall, $\mathfrak{g}_{\{i\}}$ is either $\mathfrak{sl}_2(\mathbb{C})$  ($\text{if }i\in \mathcal{I}^+\  \sqcup \mathcal{I}^-$), or 
$\text{3-dim. Heisenberg LA}$ ($\text{if }i\in \mathcal{I}^0$).
		Suppose $I\cap J_{\lambda}^c\neq 0$. 
	Pick any $i\in I\cap J_{\lambda}^c$, and a weight vector $Fv_{\lambda}\neq 0 \in V_{\mu+c_i\alpha_i}$ for some $F\in U(\mathfrak{n}^-)_{\mu-\lambda}$.
		Now $F\cdot f_i^{c_i}\cdot v_{\lambda}\in V_{\mu}\neq 0$ by \eqref{sl-2 theory non vanishing saclar for e-i action}, implies $\mu\in \wt V$:
		\begin{equation}\label{sl-2 theory non vanishing saclar for e-i action}    Ff_i^{c_i}v_{\lambda}=0=e_i^{c_i}Ff_i^{c_i}v_{\lambda}\implies\ 
		0= Fe_i^{c_i}f_i^{c_i}v_{\lambda} = c_i\Big[\lambda(\alpha_i^{\vee})- (c_i-1)\frac{A_{ii}}{2}\Big]Fv_{\lambda}\neq 0\in V_{\mu} \Rightarrow\!\Leftarrow.     \end{equation}
		So, we proceed assuming $J_{\lambda}\supseteq I$. 
When $\{a,b\}\cap \mathcal{I}^0\sqcup\mathcal{I}^-\neq \emptyset$, we choose any $i$ form it and apply Lemma \ref{Lemma imaginary root strings} to $\mu+c_i\alpha_i$ for $\mu\in \left(\mu+c_i\alpha_i-\mathbb{Z}_{\geq 0}\alpha_i\right)\subseteq \wt  V$.
		Finally suppose $\{a,b\}\subseteq I\cap\mathcal{I}^+\subseteq J_{\lambda}$; so $\lambda(\Pi_{\{a,b\}})\subset \mathbb{Z}_{\geq 0}$.
        W.l.o.g. let $c_b\geq c_a$.
        The proof can be completed following the standard calculations in \cite[Proof of Theorem 3.14 in Appendix-B]{Teja--Khare}, by replacing $(i',i)$ therein by $(b,a)$.	
	\end{proof}
	\section{Theorems \ref{Theorem weight formula for nice modules} \& \ref{Theorem wt-formula by composition series simples for nice V}: Weyl-orbit type and composition series based weight-formulas}\label{Section 3}
We show our first two main results Theorem \ref{Theorem weight formula for nice modules}, \ref{Theorem wt-formula by composition series simples for nice V} that yields at once uniform weight-formulas for all : i) simple h.w.m.s $L(\lambda)$ $\forall$ $\lambda\in \mathfrak{h}^*$; ii) parabolic Vermas in BKM setting, and for their generalizations $\mathbb{M}(\lambda,\text{ nice }  \mathcal{H})$ (Definition \ref{Defn holes in BKM setting}[d]); iii) moreover for all h.w.m.s $V$ with $\mathcal{H}_V$ nice.
	\begin{definition}[\bf Partial order set $\big(\mathrm{Indep}(J_{\lambda}),\ \preceq \big)$  and minimal holes]\label{Defn partial ordering on holes}
		Fix two independent subsets $H, H' \subseteq \mathcal{I}$, and power sequences $m_H:H \rightarrow \mathbb{Z}_{\geq 0}$ and $m_{H'}: H' \rightarrow \mathbb{Z}_{\geq 0}$.
        We define
		\[
		(H',m_{H'})\ \preceq \ (H,m_H) \quad\ \  \text{iff}\ \ \quad  H'\subseteq H \ \text{ and }\  m_{H'}(h')\leq m_H(h')\ \forall \ h'\in H'. 
		\]
	 $(H, {\bf 0}\text{-map})\preceq  \text{ any }(H, m_H:H\rightarrow\mathbb{Z}_{\geq 0})$.	
        The subset consisting of the minimal holes in hole-set $\mathcal{H}$:
		\[
		\mathcal{H}^{\min}:=\big\{(H,m_H)\in \mathcal{H}\ |\ (H,m_H) \succeq (H',m_{H'})\in \mathcal{H} \ \implies\  (H',m_{H'})=(H,m_H)\big\}.
		\]
	\end{definition}
	A useful note on the freeness of higher order Vermas $\mathbb{M}(\lambda,\mathcal{H})$ (towards parabolic induction):
	\begin{observation} 
		Fix $\mathcal{H}\subseteq \text{Indep}(J_{\lambda})$. 
        Suppose for some $I\subseteq \mathcal{I}$, all the minimal holes of $\mathcal{H}$ are contained in $I^c$, i.e. $(H,m_{H})\in \mathcal{H}_V^{\min} \ \implies \ H\subseteq I^c$. 
		So there are no holes in $I$;
		e.g. no hole of $V \twoheadleftarrow M(\lambda)$ is contained in $I=J_{\lambda}^c$.
		Then $\mathbb{M}(\lambda, \mathcal{H})$ is the free $U\left(\oplus_{\alpha\in \Delta^-\setminus \Delta_{I^c}^-}\mathfrak{g}_{\alpha}\right)$-module -- particularly free $U(\mathfrak{n}^-_I)$-module -- freely generated by any (PBW) vector space basis of the $\mathfrak{g}_{I^c}$-h.w.m. $U(\mathfrak{n}^-_{I^c})\cdot \mathbb{M}(\lambda,\mathcal{H})_{\lambda}$.
		This follows by the existence of the suitable PBW-monomial basis for $U(\mathfrak{n}^-)$  where-in  $I^c$-supported elements always appear first, via $U(\mathfrak{n^-})=U\left(\oplus_{\alpha\in  \Delta^-\setminus \Delta_{I^c}^-}\mathfrak{g}_{\alpha}\right)\otimes U(\mathfrak{n}^-_{I^c})$.
	\end{observation}
The necessity of (C1)--(C2) in the weight-formula in Theorem \ref{Theorem weight formula for nice modules} is seen below; it is a consequence of, as well as isolates, the symmetry due to singleton $\mathcal{I}^+$-holes of $V$ vs of the other holes of $V$.   
	\begin{lemma}\label{weight formula 1 forward inclusion lemma}
		Fix any BKM LA $\mathfrak{g}$,  $\lambda\in\mathfrak{h}^*$, and   $M(\lambda)\twoheadrightarrow V\neq 0$ with integrability $I_V$ (and $\mathcal{H}_V$ unconstrained). 
        The forward inclusion in formula \eqref{Weight formula for nice modules} in Theorem \ref{Theorem weight formula for nice modules} is true :
		\begin{equation}\label{Weight formula 1 - forward inclusion}
		\wt V \ \subseteq  \ W_{I_V} \{\mu \preceq \lambda\ |\ 
		\mu \text{ satisfies conditions } (C1) \text{ and }(C2) \text{ in Theorem }\ref{Theorem weight formula for nice modules} \}.
		\end{equation}
			\end{lemma}
	\begin{proof}[{\bf \textnormal Proof.}]
		Fix $\mu\in \wt V$, with property (C1) (by the $W_{I_V}$-invariance of 
 $V$).
 		For (C2), we assume $\mu\precneqq \lambda$, and fix  $c_j$s, support $J$ and components $J_k\ni j_k$, $\mu'$, all as therein.
		By Lemma \ref{Lemma reaching lambda from mu}(a), we fix 
		\begin{center}
		$x=(f_{i_1}\cdots f_{i_r})\cdot v_{\lambda}\neq 0 \in V_{\mu}\quad \text{ for nodes }\ 
		i_1,\ldots, i_r\in J \text{ s.t. } \mu=\lambda-\sum_{j\in J}c_j\alpha_j=\lambda-\sum_{t=1}^{r}\alpha_{i_t}$.
		\end{center}
        Above, $i_1,\ldots, i_r$ need not be distinct.
		If $n=1$ (i.e. $J$ is connected) and $|J_1|=1$, the result is manifest as $\mu'=\mu$.
		If $n=1$ and $|J_1|>1$, we set $\mu'=\lambda-\alpha_{i_r}$ lying in $\wt V$ by $f_{i_r}v_{\lambda}\neq 0$.
		For simplicity, we show the result for $n=2$; the proof in the general case $2\leq n\in \mathbb{N}$ can be completed analogously.
		We set $d:=\sum_{j\in J_1}c_j\ \leq \ r=\sum_{j\in J_1\sqcup J_2}c_j$.
		By permuting $\{1,\ldots,r\}$ -- since $f_p$ and $f_q$ commute $\forall$ $p\in J_1,\ q\in J_2$ -- we assume  $i_{1},\ldots, i_{d}\in J_1$ \ and \ $i_{d+1},\ldots, i_{r}\in J_2$, and thereby we set 
		\[
       \ F_1:= (f_{i_{1}}\cdots f_{i_{d}})\in U\big(\mathfrak{n}^-_{J_1}\big) \ \text{ and }\ F_2:=(f_{i_{d+1}}\cdots f_{i_r}) \in U\big(\mathfrak{n}^-_{J_2}\big); \quad \text{so that }\   
		x=F_1F_2 \cdot v_{\lambda}\ \neq 0.
		\]
		For $t\in \{1,2\}$, if $|J_t|=1$, the $t^{\text{th}}$-monomial  above is  $f_{j_t}^{c_{j_t}}$.
		Set 
		$b_t=\begin{cases}
		1 \  & \text{if }|J_t|>1,\\
		c_{j_t}\  & \text{if }|J_t|=1.
		\end{cases}
		$  $x=(F_1\cdot f_{i_d}^{-b_1})\\  (F_2\cdot f_{i_r}^{-b_2})\cdot f_{i_d}^{b_1}f_{i_r}^{b_2}\cdot v_{\lambda}\neq 0$ (in spirit of \cite{Malikov}) shows $\mu'=\lambda-b_1\alpha_{i_d}-b_2\alpha_{i_r}\in \wt V$ as desired.
	\end{proof}
    \subsection{All h.w.m.s $V$ with Minkowski difference weight-formulas}\label{Subsection Minkowski difference wt.-form.}
Lemma \ref{weight formula 1 forward inclusion lemma} naturally raises the question of finding all $V$ for which the reverse inclusion (i.e., $\wt V $ equaling the $W$-orbits) in \eqref{Weight formula 1 - forward inclusion} is true.
	Using nice-holes (Definition \ref{Defn holes in BKM setting}[d], extending integrability) solves it, and proves Theorem \ref{Theorem weight formula for nice modules}. 
   For our study of the $I_V$-symmetry side of $\wt V$ (for Theorem \ref{Theorem weight formula for nice modules}), crucial is the understanding of the complementary freeness aspect : Minkowski difference formulas, Proposition \ref{Proposition Minkowski difference formula} below.
	We now extend the results from KM case \cite{MDWF} to BKM setting; for freeness $\mathcal{I}\setminus J_V$ of $\wt V$ below.
	\begin{lemma}\label{Lemma freeness outside holes support}
		Let $\mathfrak{g}, V$ be as in Theorem \ref{Theorem weight formula for nice modules} ($\mathcal{H}_V$ is nice). 
		Set $J_V:=\bigcup\limits_{(H,m_H)\in \mathcal{H}_V^{\min}}H$; 
		so $J_V\cap \mathcal{I}^+\subseteq I_V$.
		We have refining the result 
for $J=J_{\lambda}$ in Lemma \ref{Freeness-directions for weights lemma}, now to $J=J_V$ (and moreover for all $J \supseteq J_V$):
		\[
		\big[\wt V\cap (\lambda-\mathbb{Z}_{\geq 0}\Pi_{J_V})\big] \ -\  \mathbb{Z}_{\geq 0}\Pi_{J_V^c}\ \ \subseteq\  \wt V.
		\]
	\end{lemma}
	\begin{proof}[{\bf \textnormal Proof.}]
   We prove the lemma for $J_V$; this result for any $J\supseteq J_V$ is seen by Proposition \ref{Proposition Minkowski difference formula}\eqref{Eqn Minkowski difference formula all V}.
		We prove for an arbitrary weight $\mu\in \wt V\cap [\lambda-\mathbb{Z}_{\geq 0}\Pi_{J_V}]$ and sum $x\in \mathbb{Z}_{\geq 0}\Pi_{J_V^c}$ that $\mu-x\in \wt V$ by inducting on 
		$\mathrm{ht}(\lambda-\mu+x)\geq 0$.
		In the base step $\mu-x=\lambda$ (so $x=0$) the result is manifest.
		\smallskip\\
		Induction step: $\mu-x\in \wt V$ trivially when $x=0$, and by  Proposition~\ref{Proposition no holes full weights} when $\mu=\lambda$.
        We assume henceforth $x\neq 0$ and $\mu\precneqq \lambda$, and write $x=\sum_{t\in \supp(x)}x_t\alpha_t$ for $x_t\in \mathbb{Z}_{>0}$ $\forall$ $t\in \supp(x)$ 
		and $\mu=\lambda-\sum_{t\in \supp(\lambda-\mu)}c_t\alpha_t$ for $c_t\in \mathbb{Z}_{>0}$ $\forall$ $t\in \supp(\lambda-\mu)$. 
		Recall $J_{\lambda}=\left\{j\in \mathcal{I}\ \Big|\ 
		\lambda(\alpha_j^{\vee})\in \frac{A_{jj}}{2}\mathbb{Z}_{\geq 0} \right\}.$
		
        When $i\in \supp(x)\cap J_{\lambda}^c\neq  \emptyset$, by the induction hypothesis $\mu'=\mu- \sum_{t\in \supp(x)\setminus \{i\}}x_t\alpha_t\in\wt V$, and so $\mu'-x_i\alpha_i\in \wt V$ by Lemma \ref{Freeness-directions for weights lemma}.
So assume below that $\supp(x)\subseteq J_{\lambda}$; indeed $\supp(x)\subseteq J_{\lambda}
		\setminus J_V$.  
		\smallskip\\
		\underline{Step 1}: Suppose the Dynkin subdiagram on $\supp(\lambda-\mu + x)= \supp(\lambda-\mu)\sqcup \supp(x)$ has no edges, and $\mu- x\notin \wt V$.
		There exists a hole $(H,m_H)\in \mathcal{H}_V^{\min}$ s.t. $\mu-x\preceq \lambda-\sum_{h\in H}m_H(h)\alpha_h$.
		Now $H\subseteq J_V$ by definitions and $\supp(x)\subseteq J_V^c$ together force $H\subseteq \supp(\lambda-\mu)$, and moreover $\mu\preceq \lambda-\sum_{h\in H}m_H(h)\alpha_h$.
		Finally $V_{\mu}=0$ (as $\dim M(\lambda)_{\mu}=1$) contradicts $\mu\in \wt V$ we began with.
		\smallskip\\
		\underline{Step 2}: 
		Suppose there are edges between say $a,b\in \supp(x)\subseteq J_V^c$; and let $x_b\geq x_a>0$.
		Firstly, $\mu'=\mu-\sum_{t\in \supp(x)\setminus \{a\}}x_t\alpha_t\in \wt V$ by the induction hypothesis.
If $a\in \mathcal{I}^0\sqcup\mathcal{I}^-$, Lemma \ref{Lemma imaginary root strings} applied to $\mu'$ and $a$ yields $\mu\in \mu'- \mathbb{Z}_{\geq 0}\alpha_a\subseteq \wt  V$.
		So we let, $a\in J_{\lambda}\cap \mathcal{I}^+$ (by the assumption above Step 1); thus $\lambda(\alpha_a^{\vee})\in \mathbb{Z}_{\geq 0}$.
        We will be done by following the conclusive lines of proof in Proposition \ref{Proposition no holes full weights}.		\smallskip\\
		\underline{Step 3}: We assume w.l.o.g. in this concluding step that node $a\in\supp(\lambda-\mu)\neq \emptyset$; and $b\in \supp(\lambda-\mu)\sqcup \supp(x)$.
	Now, either $\{a,b\}\subseteq \supp(\lambda-\mu)$, or $\{a,b\}\cap \supp(\lambda-\mu)=\{a\}$ and $b\in \supp(x)$.
		Let $T$ be the Dynkin graph (connected) component of $\supp(\lambda-\mu)$ containing $a$. 
		As $U\left(\mathfrak{n}^-_{\supp(\lambda-\mu)}\right)=U\left(\mathfrak{n}^-_T\right)\otimes U\left(\mathfrak{n}^-_{\supp(\lambda-\mu)\setminus T}\right)= U\left(\mathfrak{n}^-_{\supp(\lambda-\mu)\setminus T}\right)\otimes U\left(\mathfrak{n}^-_T\right)$, we pick a node $j_1\in T$ with $\mu+\alpha_{j_1}\in \wt V\cap [\lambda-\mathbb{Z}_{\geq 0}\Pi_{J_V}]$.
		Note, $(\mu+\alpha_{j_1})-x\in \wt V$ by the induction hypothesis.
		Next, the Dynkin subgraph component in $\supp(\lambda-\mu+x)$ containing $T\cup \{b\}$ is non-singleton, and so $j_1$ has edge(s) to some node in $\supp(\lambda-\mu+x)$.
		When $j_1\in \mathcal{I}^-\sqcup \mathcal{I}^0$, Lemma~\ref{Lemma imaginary root strings} applied to $\mu+\alpha_{j_1}-x$ and $\alpha_{j_1}$ yields $\mu-x\in \mu+\alpha_{j_1}-x-\mathbb{Z}_{\geq 0}\alpha_{j_1}\subseteq  \wt V$ as required.
The leftover case now is when $j_1\in \mathcal{I}^+$; our assumption that the minimal holes along $J_V\cap \mathcal{I}^+$-directions are singletons is needed here.
		Since $j_1\in J_V\cap \mathcal{I}^+\subseteq I_V$,  $\wt V$ is $s_{j_1}$-invariant; and recall both $\mu ,\mu+\alpha_{j_1}\in \wt V$.
        So there exists maximum weight $\widetilde{\mu}$ in the $\alpha_{j_1}$-root string $(\mu+\mathbb{Z}_{>0}\alpha_{j_1})\cap \wt V$ through $\mu$ s.t. $\widetilde{\mu}(\alpha_{j_1}^{\vee})>0$ and $\mu\in [s_{j_1}\widetilde{\mu},\ \widetilde{\mu}]\subseteq \wt V$.
		Note  $\widetilde{\mu}\succeq \mu+\alpha_{j_1} \succneqq \mu$, and so $\widetilde{\mu}-x\in \wt V$ by the induction hypothesis. 
		Next, $x(\alpha_{j_1}^{\vee})\leq 0$ since $j_1\notin \supp(x)\subseteq J_V^c$, and so $[\widetilde{\mu}-x](\alpha_{j_1}^{\vee})\geq \widetilde{\mu}(\alpha_{j_1}^{\vee})>0$.
		We apply Lemma~\ref{Lemma real root strings} to $\widetilde{\mu}-x$ and $\alpha_{j_1}$, which implies $\mu-x\in [s_{j_1}\widetilde{\mu}, \widetilde{\mu}]\subseteq \wt V$, completing the proof.
	\end{proof}
We immediately have all possible uniform Minkowski difference formulas for every $\wt V$. \begin{prop}\label{Proposition Minkowski difference formula}
		Let $\mathfrak{g}, V$ and $J_V$ be as in the above lemma ($\mathcal{H}_V$ is nice).
		Then
		\begin{equation}\label{Eqn Minkowski difference formula all V}
		\wt V \ \  =\ \   \big(\wt V\cap [\lambda-\mathbb{Z}_{\geq 0}\Pi_{J_V}]\big) \ -\  \mathbb{Z}_{\geq 0}\big(\Delta^+\setminus \Delta_{J_V}^+\big).
		\end{equation}
		\begin{equation}\label{Eqn freeness inclusion}
		\text{In particular,}\quad \wt V \ -\ \mathbb{Z}_{\geq 0}\Pi_{J_{\lambda}^c}\ \  \subseteq\ \  \wt V \ -\ \mathbb{Z}_{\geq 0}\Pi_{J_V^c}\ \  \subseteq \ \ \wt V.
		\end{equation}
		The analogous results hold true for any $J\supseteq J_V$ in place of $J_V$ above.
	\end{prop}
	\begin{proof}[{\bf \textnormal Proof.}]
    We completely adopt the proof lines of \cite[Theorem A]{MDWF}, with writing necessary additional steps, and show Proposition \ref{Proposition Minkowski difference formula}.
    The strategy here is the same as the one in the proof of \cite[Equation (3.3)]{MDWF} presently for $J=J_V$,  $\gamma_1,\ldots, \gamma_n\in \Delta^+\setminus \Delta_{J_V}^+$ and $\mu\in \wt V$ with $\supp(\lambda-\mu)\subseteq J_V$; in the place of $J=J_{\lambda}$, $\gamma_t$s in $\Delta^+\setminus \Delta_{J_{\lambda}}^+$ (or more precisely, $\Delta_{J_{\lambda}^c,1}$) all as therein.
    In this, we induct on $\height_{J_V}(\gamma_1+\cdots + \gamma_n)\geq 0$; the proof for the base case follows by the previous lemma. 
   For the induction step, assuming w.l.o.g. $\height_{J_V}(\gamma_1)>0$, we pick by \cite[Lemma 3.4]{Teja-ArXiv} or \cite[Lemma 4.6]{Venkatesh}, a node $j_1$ with $\gamma_1-\alpha_{j_1}\in \Delta^+\setminus \Delta_{J_V}^+$. 
 When $j_1\notin J_V$, we fix in turn by those lemmas
		\begin{itemize}
			\item nodes $j_1,\ldots,j_{m-1}\in J_V^c$ and $j=j_m\in J_V$ -- 
			 stage `$m$' exists since $\mathrm{ht}_{J_V}(\gamma_1)>0$ -- s.t.
			\item $\gamma_1- (\alpha_{j_1}+\cdots +\alpha_{j_k})\in \Delta^+\setminus \Delta_{J_V}^+$ for every $1\leq k\leq m$.
		\end{itemize}
	\[ \text{We set }\ \ \   
 		\gamma_1':=\gamma_1-(\alpha_{j_1}+\cdots+ \alpha_{j_m})\ ;\quad \gamma_t':=\gamma_t\ \ \forall\ 2\leq t\leq n \ ; \quad    \gamma_{n+t}':=\alpha_{j_t} \ \in \Pi_{J_V^c}\ \ \forall\ 1\leq t\leq m-1 .
		\]
		Note $\mu-\sum_{t=1}^n\gamma_t=\left(\mu-\sum_{t=1}^{m+n-1}\gamma_t'\right)-\alpha_{j_m}$.
		Now $\mathrm{ht}_{J_V}(\sum_{t=1}^{n+m-1}\gamma_t')=\mathrm{ht}_{J_V}(\sum_{t=1}^{n}\gamma_t)-1$ implies by the induction hypothesis $\mu'=\mu-\sum_{t=1}^{n+m-1}\gamma_t'\in \wt V$.
		Showing $\mu'-\alpha_{j_m}\in \wt V$ completes the proof, for which we analyze the $\alpha_{j_m}$-root string through $\mu'$ below; in two cases  $j_m\in (\mathcal{I}^-\sqcup \mathcal{I}^0)\cap J_V$ or $j_m\in \mathcal{I}^+\cap J_V$.
		Suppose $j_m\in (\mathcal{I}^-\sqcup \mathcal{I}^0)\cap J_V$.
        Recall $\supp(\gamma_1)$ is non-singleton and is connected. 
        Then $j_m$ has edge(s) to some node in $\supp(\gamma_1)\setminus \{j_m\}\subseteq \supp(\lambda-\mu')=\supp(\lambda-\mu)\cup \supp\left(\sum_{t=1}^{n+m-1}\gamma'_t\right)$, and we are done by Lemma \ref{Lemma imaginary root strings} applied to $\mu',\alpha_{j_m}$.
        In the case $j_m\in \mathcal{I}^+\cap J_V$, we will be done by imitating the proof lines in the case $j\in J_{\lambda}$ in \cite[ Theorem A]{MDWF}.
        And this requires as in that proof: the top elements for   $\mathfrak{g}_{\{j_m\}}$-action on $V_{\mu}$ and $\mathfrak{g}_{\gamma_1'-\alpha_{j_m}}$ and on other $\mathfrak{g}_{\gamma_t}$s, namely, setting up 
$\widetilde{\mu}$ and $\widetilde{\gamma}_1,\ldots, \widetilde{\gamma}_{n+m-1}$, with four properties (a)--(d) there-in.
The first inclusion in \eqref{Eqn freeness inclusion} is true as $J_{\lambda}^c\subseteq J_V^c$.
The last assertion for any $J\supseteq J_V$, can be showed by repeating the proof above for $J$. 
	\end{proof}
    \subsection{All h.w.m.s $V$ with Weyl orbit weight-formulas, including $V=L(\lambda)\ \forall\ \lambda\in \mathfrak{h}^*$}\label{Subsection Theorem A proof}
	Equipped with Proposition \ref{Proposition Minkowski difference formula}, importantly freeness \eqref{Eqn freeness inclusion} of $\wt V$, we prove our first main result.
   \begin{proof}[{\bf Proof of Theorem \ref{Theorem weight formula for nice modules}:}]
   Weight-formula in Corollary \ref{Corollary to Theorem A} for all $L(\lambda)$ follows by Theorem \ref{Theorem weight formula for nice modules}, via  noting : i)~In the notations as there-in, for any component $J_t\subseteq J_{\lambda}$ of $\supp(\lambda-\mu)$, $f_{j_t}L(\lambda)_{\lambda}\neq 0\iff \lambda(\alpha_{j_t}^{\vee})\neq 0$. 
   ii)~Moreover while $J_t=\{j_t\}$, observe $\lambda-\sum_{j\in \supp(\lambda
   -\mu)\setminus J_t}c_j\alpha_j\in \wt L(\lambda)\iff \mu\in \wt L(\lambda)\iff c_{j_t}\leq \frac{2}{A_{j_t j_t}}\lambda(\alpha_{j_t}^{\vee})$ by Lemmas \ref{Lemma rank-1 BKM rep. theory} and \ref{Lemma real root strings}, and we are done. 
   Now we prove the result for all nice $V$.
		
		Our strategy in the proof is to work uniformly with {\it all triples} $(\mathfrak{g},V,\nu)$ of : 
BKM LA $\mathfrak{g}$, any $\mathfrak{g}$-h.w.m. $V\twoheadleftarrow M(\lambda)$ with nice $\mathcal{H}_V$, and $\nu\preceq \lambda\in \wt V$ with properties (C1)--(C2) in Theorem \ref{Theorem weight formula for nice modules}.
	
        The theorem is clearly true for all  $(\mathfrak{g},V,\lambda)$ in that triples' space.
We assume that the theorem is false for a triple $(\mathfrak{g},V,\mu)$ -- i.e., $M(\lambda)\twoheadrightarrow V$ and $\mu\preceq \lambda$ satisfies (C1) and (C2), but $\mu\notin \wt V$ -- and exhibit contradictions in several steps we proceed-in below; and w.l.o.g. let $(\mathfrak{g},V,\mu)$ have the least $\mathrm{ht}(\lambda-\mu)>0$ among such counter examples.
		Write $\mu=\lambda-\quad \  \sum\limits_{\mathclap{j\in J= J_1\sqcup \cdots \sqcup J_n}}\quad c_j\alpha_j$ for $c_j\in \mathbb{Z}_{>0}$ and components $J_t$s as in (C2).
		By $\mu\notin \wt V$ we fix a weight $\eta\in \wt V$ with the least $\mathrm{ht}(\eta-\mu)$ s.t. $\mu\precneqq \eta \preceq  \lambda$. 
		So $\mathrm{ht}(\eta-\mu)\leq \mathrm{ht}(u-\mu)$ for all $u\in \wt V$ that are above $\mu$.
		Note as $J=\supp(\lambda-\mu)\neq \emptyset$, $\lambda\succneqq \mu'$ which is
		the element associated to $\mu$ \big(i.e. to nodes $J$ and sequence $(c_j)_{j\in J}$\big) in (C2).
	Thus $\mathrm{ht}(\lambda-\eta)\geq \mathrm{ht}(\lambda-\mu')>0$ says $J\supseteq I:=\supp(\lambda-\eta)\neq \emptyset$.
    Next let $K:=\supp(\eta-\mu)\subseteq J$, and $\eta-\mu=\sum_{k\in K}d_k\alpha_k$ for $0<d_k\leq c_k$ $\forall$ $k\in K$.
    Note $J\setminus I\subseteq K$, and  $K\cap I$ is $=\emptyset$ or $\neq\emptyset$.
       The weights and supports extensively used in the proof below are pictured for convenience-
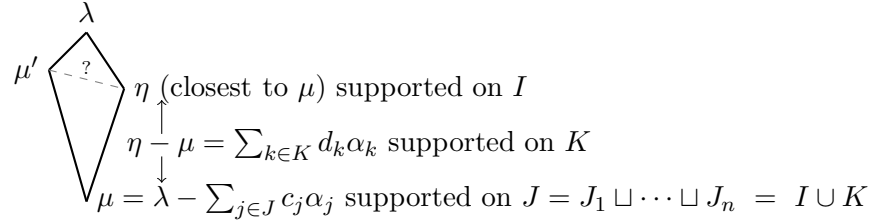
       \begin{center}
       \hspace*{2cm}
\begin{tikzpicture}[scale=0.4]
\draw (0,3) node[anchor=south]{$\lambda$};
\draw[thick, black] (0, 3) -- (1, 1.5) ;
\draw[thick, black] (1, 1.5) -- (0, -1.5) ;
\draw[thick, black] (-1,2) -- (0, -1.5) ;
\draw (1,1.5) node[anchor=west]{$\eta$ (closest to $\mu$) supported on $I$};
\draw[thick, black] (0, 3) -- (-1, 2) ;
\draw (-1,2) node[anchor=east]{$\mu'$} ;
\draw (-0.8,1.1) node[anchor=east]{\text{in (C2)}} ;
\draw[dashed, gray] (-1,2) -- (1, 1.5) ;
\draw (0,2.5) node[anchor=north]{\tiny ?} ;
\draw (0,-1.5) node[anchor=west]{$\mu=\lambda-\sum_{j\in J}c_j\alpha_j$\  supported on $J=J_1\sqcup \cdots \sqcup J_n\ = \ 
I\cup K $};
\draw[->] (2,0.3) --  (2,1.2) ;
\draw[->] (2,-0.3) --  (2,-1) ;
\draw (0.8,0) node[anchor=west]{$\eta-\mu=\sum_{k\in K}d_k\alpha_k$\  supported on $K$};
\end{tikzpicture}
\end{center}
If $J$ is independent, then $\mu=\mu'\in \wt V$ contradicts our assumption that the theorem is not true  for $(\mathfrak{g},V, \mu)$.
		Indeed, we will eventually see proceeding in the below five steps that $J$ cannot have any edge in its Dynkin subdiagram, and then the proof concludes by the previous line.

        We note on Steps 1--5 before we start: i) The plan for  each step is quoted at its beginning, and all the cases occurring in it are completely solved. 
		ii) The notations in a step are independent of the other steps' \big(ex: $\delta, \gamma, J_1,V',V''$ in Steps 5, 6\big).
		iii) Our goal(s) is to find a node $k\in K$ with $\eta-\alpha_k\in \wt V$.
        Then $\mu\preceq \eta-\alpha_k \precneqq \eta$ contradicts the minimality of $\eta\in \wt V$ in covering $\mu$. 
		\smallskip\\ 
		\underline{Step 1}. {\it Necessarily $K\subseteq J_V=\ \ \ \ \bigcup\limits_{\mathclap{(H,m_H)\in \mathcal{H}_V^{\min}}}\ H\ \    \subseteq J_{\lambda}$, and moreover $V$ is $\mathfrak{g}_{K\cap \mathcal{I}^+}$-integrable}: 
	On the contrary suppose there is a node $k\in K\setminus J_V$.
		Observe by the freeness Proposition \ref{Proposition Minkowski difference formula} \eqref{Eqn freeness inclusion}, that $\eta-(c_k-d_k)\alpha_k\in \wt V -\mathbb{Z}_{\geq 0}\Pi_{J_{V}^c}\subseteq \wt V$; in particular $\eta-\alpha_k\in \wt V$ $\Rightarrow\!\Leftarrow$. 
	The niceness of $\mathcal{H}_V$ implies $\big(\{k\}, \lambda(\alpha_k^{\vee})+1\big)\in \mathcal{H}_V^{\min}$ $\forall$ $k\in K\cap\mathcal{I}^+$; proving the second claim in this step.
		\smallskip\\
		\underline{Step 2}. {\it  $I$ (as well as $J\setminus I=K\setminus I$) is the union(s) of some of the sets among $J_1,\ldots ,J_n$}:
		This claim in particular says that no two nodes (one each) from $I$ and $J\setminus I$ have any edges in-between in the Dynkin diagram.
		On the contrary suppose nodes $j\in J\setminus I$ and $i\in I$ have at least one edge in-between.
		Note that $j\in K=\supp(\eta-\mu)\subseteq J_{\lambda}$, and so $\lambda(\alpha_j^{\vee})\in \frac{A_{ii}}{2}\mathbb{Z}_{\geq 0}$.
		When $j\in \mathcal{I}^+$, observe
        $\eta(\alpha_j^{\vee})\geq \lambda(\alpha_j^{\vee})-\sum_{t\in I}\alpha_t(\alpha_j^{\vee})\geq \lambda(\alpha_j^{\vee})-\alpha_i(\alpha_j^{\vee})\geq -\alpha_i(\alpha_j^{\vee})>0$.
        So Lemma \ref{Lemma real root strings} when $j\in \mathcal{I}^+$, or Lemma \ref{Lemma imaginary root strings} when $j\in \mathcal{I}^-\sqcup\mathcal{I}^0$, applied for $\eta$ and $\alpha_j$ shows the contradiction $\eta -\alpha_j\in \wt V$.
		So there are no edges between $I$ and $J\setminus I$.
		Hence, by the connectedness of each $J_t$, we have that either $J_t\cap I=\emptyset$ or $J_t\subseteq I$ $\forall$ $1\leq t\leq n$; i.e., $J_t\cap I \neq \emptyset \implies J_t\subseteq I$.
		Now we study $K\cap I$.
		\smallskip\\ 
		\underline{Step 3}. {\it No two nodes each in $K$ and $I$ have any edges in-between, and so $K\cap I$ (when non-empty) is the union of a sub-collection of the singleton $J_t$s among $J_1,\ldots, J_n$}:
		We begin with an arbitrary component $K_1$ of $K$.  
		For showing the disconnectedness of $K_1$ (thereby of $K$) with $I$, we assume on the contrary a node $k\in K_1$ to have an edge to some $i\in I$; possibly $i\in K_1$.
		$K_1\cup\{i\}$ must be contained in a connected component of $J$, say $J_1$.
		Thus, $|J_1|>1$ and moreover $i\in I$ forces $J_1\subseteq I$ by Step 2; in particular $K_1\subseteq I$.
		If $k'\in K_1\cap (\mathcal{I}^-\sqcup\mathcal{I}^0)\neq \emptyset$, Lemma \ref{Lemma imaginary root strings} for $\eta$ and $k'$ leads to the contradiction $\eta-\alpha_{k'}\in \wt V$.
		So $K_1\subseteq \mathcal{I}^+$; and $\mathfrak{g}_{K_1}$ is an indecomposable (as $K_1$ is connected) KM subalgebra of $\mathfrak{g}$.
		We proceed by types of $\mathfrak{g}_{K_1}$- finite/affine/indefinite.\\
		(i) $\mathfrak{g}_{K_1}$ is of affine or indefinite type:
		We have $x\in \mathbb{Z}_{>0}\Pi_{K_1}^{\vee}$ with $\alpha_t(x)\leq 0$ $\forall$ $t\in K_1$, also $\forall\ t\in I$, by \cite[Theorem 4.3]{Kac book}.
		Either i) $K_1\subsetneqq J_1$, or ii) $K_1=J_1\subseteq J=\supp(\lambda-\mu)$.
		In case ii), there exists $k'\in K_1$ s.t. $\lambda-\alpha_{k'}\in \wt V$ (by Lemma \ref{weight formula 1 forward inclusion lemma} or (C2) for $\mu\in \wt V$),  equivalently $\lambda\big(\alpha_{k'}^{\vee}\big)>0$ by the $\mathfrak{g}_{K_1}$-integrability of $V$.
		So in cases i)--ii), the first term of $\eta(x)$ on the r.h.s. below is $>0$.
		\begin{center}
		$\eta(x) \ \ =\  \ \underbrace{\Big[ \lambda-\sum\nolimits_{j\in J_1\setminus K_1} \big(c_j-d_j\big) \alpha_j\Big](x)}_{>0} \ - \   
		\sum_{j\in K_1\sqcup \big(I\setminus J_1  \big)}\underbrace{(d_t-c_t)\alpha_t(x)}_{\leq 0}\ \ >0$.
		\end{center}
So $\eta(\alpha_{k'}^{\vee})>0$ for some $k'\in K_1$, and then Lemma \ref{Lemma real root strings} shows $\eta-\alpha_{k'}\in \wt V$ $\Rightarrow\!\Leftarrow$.\\
		(ii) $\mathfrak{g}_{K_1}$ is of finite type:
		We have $x\in \mathbb{Z}_{> 0}\Pi_{K_1}^{\vee}$ s.t. $\alpha_t(x)>0$ $\forall$  $t\in K_1$, and thus $[\eta-\mu] (x)=\sum_{k\in K_1}d_k\alpha_k(x)+\sum_{k\in K\setminus K_1}d_k\cancel{\alpha_k(x)}^{=0}>0$.
		So we fix $k'\in K_1$ s.t. $[\eta-\mu](\alpha_{k'}^{\vee})=\eta(\alpha_{k'}^{\vee})-\mu(\alpha_{k'}^{\vee})>0$.
		As $K\subseteq I_V$ by Step 1 and $\mu$ satisfies (C1), we have $\mu(\alpha_{k'}^{\vee})\in \mathbb{Z}_{\geq 0}$, and so $\eta(\alpha_{k'}^{\vee})>0$.  
		The desired contradiction is again seen.
		Hence, there are no edges between the nodes of $K$ and of $I$.
		Step 4 runs through several arguments, to prove the independence of $K$.
		\smallskip\\
		\underline{Step 4}. {\it $K$ is the union of some singleton $J_t$s among $J_1,\ldots, J_n$}:
		We assume on the contrary for the entirety of this step, that $K$ has a Dynkin subgraph component $K_1$ with $|K_1|>1$.
		Now observe by Step 3 that $K_1\cap I=\emptyset$, and importantly $K_1$ equals some Dynkin subdiagram component of $J=\supp(\lambda-\mu)=I\cup K$, say $J_1$.
		So, $\alpha_i(\alpha_k^{\vee})=0$ for all $i\in I$ and $k\in K_1$.
		Thus, $\eta(\alpha_k^{\vee})=\lambda(\alpha_k^{\vee})-0=\lambda(\alpha_k^{\vee})$ $\forall$ $k\in K_1$.
		Indeed crucially $\lambda(\alpha_k^{\vee})=0=\eta(\alpha_k^{\vee})$ $\forall$ $k\in K_1$ explained as follows:
		For $k\in K_1\cap \mathcal{I}^0\neq \emptyset$, $\lambda(\alpha_k^{\vee})\neq 0\implies \eta(\alpha_k^{\vee})\neq 0$ yields the contradiction $\eta-\alpha_k\in \wt V$ by Lemma~\ref{Lemma real root strings}(3).
		Similarly, for $k\in K_1\cap \big(\mathcal{I}^-\sqcup \mathcal{I}^+\big)\subseteq J_{\lambda}$ (by Step 1),  $\lambda(\alpha_k^{\vee})=\eta(\alpha_k^{\vee})\in \frac{A_{ii}}{2}\mathbb{Z}_{> 0}$ implies by Lemma~\ref{Lemma real root strings}(1)--(2) the contradiction $\eta-\alpha_k\in \wt V$.
			
		Hereafter, $|K_1=J_1|>1$ and moreover $\lambda(\alpha_k^{\vee})=0$ $\forall$ $k\in K_1$. 
		We begin studying the $\mathfrak{g}_{K_1}$-h.w.m. $V'=U(\mathfrak{g}_{K_1})v_{\lambda}$;  $v_{\lambda}$ also treated as a h.w. vector in $V'_{\lambda}$. 
		Note by definitions that $\mathcal{H}_{V'}^{\min}\subseteq \mathcal{H}_V^{\min}$.
		The independence of $\supp(\lambda-\mu')$ of $\mu'\in \wt V$ \big(defined for $\mu$, or for $J$ and $c_j$s\big) in (C2), yields for $V'$ a node $j_1\in K_1=J_1$ with $\lambda-\alpha_{j_1}\in \wt V\cap [\lambda-\mathbb{Z}_{\geq 0}\Pi_{K_1}]=\wt V'$.
		Now $f_k V_{\lambda}'=0$ $\forall$ $k\in K_1\cap \mathcal{I}^+$ (by the $\mathfrak{g}_{K_1\cap \mathcal{I}^+}$-integrability in Step 1 of 
		$V$, thereby of $V'$), forces $j_1\in K_1\cap(\mathcal{I}^0\sqcup \mathcal{I}^-)$.
		
		For each $k\in K_1$,  $f_k v_{\lambda}$ is maximal as $\lambda(\alpha_k^{\vee})=0$; and $f_{j_1}^rv_{\lambda}$ is maximal $\forall$ $r$ if $j_1\in \mathcal{I}^0$.
	\begin{center}
	$\text{We consider the }\mathfrak{g}_{K_1}\text{-h.w.m.}\ 	  \ \displaystyle V''\ \ :=\ \ \frac{V'}{\sum_{k\in K_1\setminus \{j_1\}}U(\mathfrak{n}^-)f_k v_{\lambda}\ +\ \sum_{k\in \{j_1\}\cap \mathcal{I}^0}U(\mathfrak{n}^-)f_{k}^2v_{\lambda}}\ \ \ \twoheadleftarrow\ \ V'$.
		\end{center}
		So $f_{j_1}^2V_{\lambda}''=0$ if $j_1\in K_1\cap \mathcal{I}^0$, and $f_{j_1}^rV_{\lambda}''\neq 0$ $\forall\ r$ if $j_1\in K_1\cap \mathcal{I}^-$ as $\big( \{j_1\}, 1\big)\notin \mathcal{H}_{V'}^{\min}$. Hence 
			\begin{center}
		$\mathcal{H}_{V''}^{\min}\  =  \     \begin{cases} 
		\ \big\{ \big(\{k\}, 1\big)\ \big|\ k\in K_1\setminus \{j_1\} \big\}  & \text{if }j_1\in K_1\cap \mathcal{I}^-,\\
		\ \big\{ \big(\{k\}, 1\big)\ \big|\ k\in K_1\setminus \{j_1\} \big\}\sqcup \big\{  \big(\{j_1\}, 2\big)   \big\}  & \text{if }j_1\in K_1\cap \mathcal{I}^0.
		\end{cases} \qquad\big(\{j_1\}, 1\big)\ \notin\ \mathcal{H}_{V''}^{\min}$.
		\end{center}
		So $\lambda-\alpha_{j_1}\in \wt V''$.
		We consider the ($\lambda$-shifted) $K_1$-projection $\gamma$ of $\mu$, which we show to lie in $\wt V''$:
		\begin{center}
		$\gamma\ \ :=\ \ \lambda\ -\ \sum_{k\in K_1}(c_k=d_k)\alpha_k\ \ \precneqq\ \  \lambda$.   
		\end{center}
		Note, $\gamma(\alpha_k)=\mu(\alpha_k)\in \mathbb{Z}_{\geq 0}$ $\forall$ $k\in K_1\cap \mathcal{I}^+$. 
		Next, similar to $\mu'$ associated to $\mu$ in  (C2), we set $\gamma'=\lambda-\alpha_{j_1}\in \wt V''$.
		Hence, $\gamma$ satisfies both conditions (C1) and (C2) (with $\gamma'$) in $V''$; $\mathcal{H}_{V''}^{\min}$ being nice by $I_{V''}=K_1\cap \mathcal{I}^+$.
		Therefore- i) $\mathrm{ht}(\lambda-\gamma)<\mathrm{ht}(\lambda-\mu)$ (since $I\neq\emptyset$), ii)
		the minimality of the triple $(\mathfrak{g}, V, \mu)$ we began with forces  the theorem to be true for $(\mathfrak{g}_{K_1}, V'', \gamma)$.
		So $\gamma\in \wt V''$; and $\wt V''\subseteq  \wt V'\subseteq \wt V\implies \gamma\in \wt V$.
		Now fix a node $p\in K_1=\supp(\lambda-\gamma)$ s.t. $\gamma+\alpha_p\in\wt V''$ by Lemma \ref{Lemma reaching lambda from mu}.
		Observe  $0<\mathrm{ht}_{\{j_1\}}(\lambda-\gamma-\alpha_p)$ as $f_iV_{\lambda}''=0$ $\forall$ $i\in K_1\setminus \{j_1\}$.
        For $\gamma+\alpha_{j_1}\in \wt V''$, either $p\neq j_1$, or $p=j_1$ with $d_{j_1}>1$ .
		While $\gamma+\alpha_p$ not dominant integral along $K_1\cap \mathcal{I}^+$-directions,
		pick a suitable $w\in W_{K_1\cap \mathcal{I}^+}$, s.t. $\wt V''\ni w(\gamma+\alpha_p)$  is dominant integral along $K_1\cap \mathcal{I}^+$-directions. 
		Note $w(\gamma+\alpha_p)\precneqq \lambda$ (because $j_1\in \mathcal{I}^0\sqcup \mathcal{I}^-$) and satisfies (C1).
		To study $w(\gamma+\alpha_p)$, we set $T:=\supp\big(\lambda-w(\gamma+\alpha_p)\big)\supseteq \{j_1\}$.
		Observe $T$ is connected, as for Dynkin subgraph components $T'$ of $T$  with  $j_1\notin T'$, $\big(\{i\}, 1 \big)\in \mathcal{H}_{V''}^{\min}$ $\forall$ $j_1\neq i \in K_1\implies \mathfrak{n}^-_{T'}V''_{\lambda}=0$.
		We note before passing on-
		\begin{center}
		 $1\ \leq \ \  c_{j_1}'\  :=\  \mathrm{ht}_{\{j_1\}}\big(\lambda-w(\gamma+\alpha_p)\big)\ \ =\ \ \begin{cases}
		c_{j_1} & \text{if }p\neq j_1,\\
		c_{j_1}-1 & \text{if }p=j_1.
		\end{cases}$
		\end{center}
		With the above observations, we now return to the original $\mathfrak{g}$-h.w.m. $V$ and consider the element 
		\begin{center}
		$\delta\ \ :=\ \ \eta\ -\ \big(\lambda-w(\gamma+\alpha_p)\big)\ \  \precneqq\  \ \eta \ \  \precneqq\  \  \lambda.$
		\end{center}
		$\delta$ satisfies (C1), as $\eta$ and $w(\gamma+\alpha_p)$ with orthogonal supports individually satisfy (C1).
		We see below   importantly $\delta$ to satisfy (C2), so that $\delta\in \wt V$; after showing it, we will explain how it leads us to the desired contradiction.
		For this, we have to identify a candidate $\delta'\in \wt V$ associated to $\delta$; similar to $\mu'$ for $\mu$ and $\gamma'$ for $\gamma$.  
		We begin by recalling from Step 2 that $I$ is a union of some components among $J_1,\ldots, J_n$, precisely among $J_2,\ldots, J_n$ (since $J_1=K_1$ and $K_1\cap I=\emptyset$).
		W.l.o.g. let $I=J_2\sqcup\cdots\sqcup J_m$ for some $m\leq n$.
		Note by the previous step, $|J_k|=1$ when $J_k\cap K\neq \emptyset$ for $2\leq k\leq m$.
		We write by this and by the previous steps
		\[
		\eta\ \ =\ \ \lambda-\sum\limits_{\substack{j\text{ in the union of }J_k\text{s} \\ \text{for }k\in \{2,\ldots,m\}\ \&\  |J_k|>1}}c_j\alpha_j-\sum\limits_{\substack{k\in \{2,\ldots,m\}\\ \text{s.t. } |J_k|=1\ \&\  J_k\not\subseteq K}}c_{j_k}\alpha_{j_k}- \sum\limits_{\substack{k\in \{2,\ldots, m\}\\ \text{s.t. }|J_k|=1\ \&\  J_k\subseteq K }}(c_{j_k}-d_{j_k})\alpha_{j_k}.
		\]
		For the sought-for weight $\delta'$, we first consider the below defined $\widetilde{\delta}$ (differing from $\eta$ by $-\alpha_{j_1}$\big):
		\[
		\widetilde{\delta} :=\ \lambda-\underbrace{\alpha_{j_1}}_{\text{in }K_1=J_1}-\sum\limits_{\substack{k\in \{2,\ldots,m\}\\ \text{s.t. } |J_k|>1}}\alpha_{j_k}-\sum\limits_{\substack{k\in \{2,\ldots,m\}\\ \text{s.t. } |J_k|=1\ \&\  J_k\not\subseteq K }}c_{j_k}\alpha_{j_k}- \sum\limits_{\substack{k\in \{2,\ldots, m\}\\ \text{s.t. }|J_k|=1\ \&\  J_k\subseteq K}}(c_{j_k}-d_{j_k})\alpha_{j_k}.
		\]
		Recall from \eqref{Eqn canonical weight for mu},  $\mu'=  \lambda-\sum\limits_{\substack{k\in \{1,\ldots, n\}\\  \text{s.t. }|J_k|=1}}c_{j_k}\alpha_{j_k} - \sum\limits_{\substack{k\in \{1,\ldots, n\}\\ \text{s.t. }|J_k|>1}}\alpha_{j_k}\  \in\ \wt V$; i.e., $\prod\limits_{k\text{ s.t. }|J_k|=1}f_{j_k}^{c_k}\cdot \prod\limits_{k\text{ s.t. }|J_k|>1}f_{j_k}\cdot V_{\lambda}\neq 0$.
		Now by the independence of $\supp(\lambda-\mu')$, observe that $\widetilde{\delta}\succeq \mu'$ implies $\widetilde{\delta}\in \wt V$.
		So when $|T|>1$, it suffices to set $\delta'=\widetilde{\delta}$.
		Henceforth we assume $T=\{j_1\}$.
		Note, we require the coef. of $\alpha_{j_1}$ in the expansion of $\delta'$ to be $c_{j_1}'$, 
		but this coef. in $\widetilde{\delta}$ is 1.
		So, 
		\begin{center}
		$\text{we set}\ \ \  \delta'\ \  =\  \ \widetilde{\delta}\ -\ (c_{j_1}'-1)\alpha_{j_1}\qquad \text{ and }\ \text{ show below that }\ \ \delta'\ \in\  \wt V$.
		\end{center}
		By  assumption $|T|=1$, when $j_1\in K_1\cap \mathcal{I}^0$,  $w(\gamma+\alpha_p)\in \wt V''\cap [\lambda-\mathbb{Z}_{\geq 0}\alpha_{j_1}]$ and $f_{j_1}^2V_{\lambda}''=0$ together force $c'_{j_1}=1$, as required.
		So we assume $j_1\in K_1\cap \mathcal{I}^-$; and $c_{j_1}'$ is arbitrary.
		Note $[\widetilde{\delta}+\alpha_{j_1}](\alpha_{j_1}^{\vee})= \lambda(\alpha_{j_1}^{\vee})=0$, and moreover $V_{\widetilde{\delta}} = f_{j_1}V_{\widetilde{\delta}+\alpha_{j_1}} \  \neq 0$ is  maximal for $e_{j_1}$-action.
		By Lemma~\ref{Lemma rank-1 BKM rep. theory}, $\widetilde{\delta}(\alpha_{j_1}^{\vee})=-A_{j_1 j_1}\notin \frac{A_{j_1 j_1}}{2}\mathbb{Z}_{\geq 0}\implies \widetilde{\delta}-\mathbb{Z}_{\geq 0}\alpha_{j_1}\subset \wt V$.
		In particular, $\delta'=\delta- (c_{j_1}'-1)\alpha_{j_1}\in \wt V$ as required. 
		This shows that $\delta$ (with its associated weight $\delta'$ above) satisfies both (C1) and (C2). 
		Now $\mathrm{ht}\big(\lambda-w(\gamma+\alpha_p)\big)\leq \mathrm{ht}(\lambda-\gamma-\alpha_p)<\mathrm{ht}(\lambda-\gamma)$ implies that $\mathrm{ht}(\lambda-\delta)<\mathrm{ht}(\lambda-\mu)$.
		Since $(\mathfrak{g}, V, \mu)$ is the smallest counter example, the theorem must hold true for the triple $(\mathfrak{g}, V, \delta)$, i.e., $\delta\in \wt V$.
		But $j_1\in \supp(\eta-\delta)$ implies $\mu\prec\delta\precneqq \eta$, and now $\delta\in \wt V$ contradicts our assumption that $\eta\in \wt V$ is the closest to $\mu$ from top. 
		Hence, the Dynkin subdiagram on $K$ has no edges.
		\smallskip\\
		\underline{Step 5}. {\it $I$ (similar to $K$ in Step 4) is also the union of some singleton components among $J_1,\ldots, J_n$}:
		The below proof runs somewhat parallel to that of Step 4.
		Here as well we assume $I$ to have a Dynkin subgraph component say $J_1\subseteq I$ of $J$ (by Step 2) with $|J_1|>1$, and show  this leads to a contradiction to $(\mathfrak{g}, V, \mu)$ being a minimal counter example.
		Note from Step 3 $K\cap J_1=\emptyset$.
		Observe as $J_1$ is a component of $J$ :
		(i)~as in the beginning, the coef. in $\lambda-\eta$ of each $\alpha_j$ is $c_j$ $\forall$ $j\in J_1$; (ii)~moreover $\mu$ being dominant integral w.r.t. $J_1\cap I_V$-directions, element $\gamma$ defined below is so, \begin{center}
        $\gamma:=\lambda-\sum_{j\in J_1}c_j\alpha_j\ \in \wt \big(U(\mathfrak{n}^-_{J_1})v_{\lambda}\big)\ \subseteq \wt V \qquad \text{(not to be confused with }\gamma\text{ in Step 4})$.\end{center}
      			Note that $\eta\preceq\gamma\precneqq \lambda$.
		Recall by the independence of the support of $\mu'\in \wt V$ in (C2) of $\mu$, that for $j_1\in J_1$ we have $\lambda-\alpha_{j_1}\in \wt V$.
		Next by definitions and   by our niceness assumption on $\mathcal{H}_V^{\min}$, $J_V\cap \mathcal{I}^+=I_V$, and so  $U(\mathfrak{g}_{J_1})$-h.w.m. $U(\mathfrak{n}^-_{J_1})v_{\lambda}$ is $\mathfrak{g}_{J_1\cap J_V\cap \mathcal{I}^+}$-integrable.
		We use the $J_V^c$-freeness (Proposition \ref{Proposition Minkowski difference formula}) of $\wt V$ .
	We note either $j_1\in J_V$ or $j_1\notin J_V$.
		We consider the $U(\mathfrak{g}_{J_1})$-submodule $X=\sum\limits_{\substack{i\in J_V\cap J_1\cap(\mathcal{I}^+\sqcup \mathcal{I}^-)\\ 
				i\neq 
				j_1}}U(\mathfrak{n}^-_{J_1})f_i^{\frac{2}{A_{ii}}\lambda(\alpha_j^{\vee})+1}v_{\lambda}+ \qquad\qquad\quad \sum\limits_{\mathclap{j\in \{j_1\}\cap J_V\cap  (\mathcal{I}^+\sqcup \mathcal{I}^-) \text{ s.t. } \lambda(\alpha_j^{\vee})\neq 0}}\qquad\qquad U(\mathfrak{n}_{J_1}^-)f_j^{\frac{2}{A_{jj}}\lambda(\alpha_j^{\vee})+1}v_{\lambda}+ \sum\limits_{\substack{i\in J_V\cap J_1\cap \mathcal{I}^0\\ i\neq j_1 }}\\ U(\mathfrak{n}^-_{J_1})f_iv_{\lambda}+\sum_{j\in \{j_1\}\cap J_V\cap \mathcal{I}^0} U(\mathfrak{n}^-_{J_1}) f_j^2 v_{\lambda}$ of the $U(\mathfrak{g}_{J_1})$-h.w.m. $U(\mathfrak{n}^-_{J_1})v_{\lambda}$.
                As in Step 4 we work in the further quotient $V'=\frac{U(\mathfrak{n}^-_{J_1})v_{\lambda}}{X}$.
                Note when the 2-nd summation is non-zero in the penultimate line that $\lambda(\alpha_{j_1}^{\vee})\in \frac{A_{j_1j_1}}{2}\mathbb{Z}_{>0}$; when the 4-th sum is non-zero, $\lambda(\alpha_{j_1}^{\vee})=0$.
		For $i\in (J_1\cap J_V\cap \mathcal{I}^0)\setminus \{j_1\}$, observe $f_iV'_{\lambda}=0\implies (\{i\}, 1)\in \mathcal{H}_{V'}^{\min}$;	similarly, $\left(\{i\}, \frac{2}{A_{ii}}\lambda(\alpha_i^{\vee})+1\right)\in \mathcal{H}_{V'}^{\min}$ $\forall$ $i\in \big(J_1\cap J_V\cap (\mathcal{I}^+\sqcup\mathcal{I}^-)\big)\setminus \{j_1\}$. 
		Importantly $\lambda-\alpha_{j_1}\in \wt V'$, i.e. $(\{j_1\}, 1)\notin \mathcal{H}_{V'}$.
		This line and point (ii) in the above paragraph together imply that $\gamma=\lambda-\sum_{j\in J_1}c_j\alpha_j$ satisfies both (C1)--(C2), with $\gamma'=\lambda-\alpha_{j_1}$.
		Moreover, $\mathcal{H}_{V'}^{\min}$ is nice, and $\mathrm{ht}(\lambda-\gamma)<\mathrm{ht}(\lambda-\mu)=\mathrm{ht}(\lambda-\eta)+\sum_{k\in K}d_k$.
		So the minimality of $(\mathfrak{g},V,\mu)$ forces $\gamma\in \wt V'\subset \wt V$ by the theorem for the triple $(\mathfrak{g}_{J_1}, V', \gamma)$.
		Now we fix a pair $p\in J_1$ and  $w\in W_{J_1\cap I_V}\subseteq W_{J_1\cap I_{V'}}$ (possibly $w=1$), with $\gamma+\alpha_p\in \wt V'$ and $\wt V'\ni w(\gamma+\alpha_p)$ dominant integral along $J_1\cap I_V$-directions. 
		Let $T=\supp\big(\lambda-w(\gamma+\alpha_p)\big)$.
		We observe the two cases $j_1\in T$ or $j_1\notin T$; with two further cases in each of them that $|T|\leq 1$ or $|T|>1$.
        Let us first define (as in Step 4)  element $\delta$ supported over $T\sqcup J_2\sqcup \cdots\sqcup J_n$ below; by showing $\delta\in \wt V$ we see $\mu\in \wt V$. 
		\begin{center}
		$\delta\ \  := \ \  w (\gamma +\alpha_p)\ -\  \sum_{j\in J\setminus J_1}c_j\alpha_j$.
		\end{center}
	Note $\mu=\lambda-\sum_{j\in J}c_j\alpha_j=\gamma-\sum_{j\in J\setminus J_1}c_j\alpha_j$; and $\delta$ satisfies (C1). 
We show $\delta$ satisfies (C2) -- by finding a suitable weight $\delta'\in \wt V$ --
		so that 
		the minimality of $(\mathfrak{g}, V,\mu )$ forces $\delta\in \wt V$.
		Assuming $\delta\in \wt V$, we first see how to reach the contradiction $\mu \in \wt V$ :
		Note $w^{-1}\delta= \gamma +\alpha_p-\sum_{j\in J\setminus J_1}c_j\alpha_j=\mu+\alpha_p\in \wt V$, since $w\in W_{J_1\cap J_V\cap \mathcal{I}^+}$ leaves $\wt V$ invariant and fixes $\Pi_{J\setminus J_1}$.
	 Next $p$ has at least one edge to $\supp(\lambda-\gamma -\alpha_p)\subseteq J_1$.
		So by Lemma \ref{Lemma imaginary root strings}, $(\mu+\alpha_p)-\alpha_p=\mu\in \wt V$ when $p\in \mathcal{I}^0\sqcup \mathcal{I}^-$ as required. 
		When $p\in \mathcal{I}^+\cap J_V^c$, observe $(\mu+\alpha_p)-\alpha_p\in \wt V$ by freeness \eqref{Eqn freeness inclusion} of $\wt V$.
	When $p\in \mathcal{I}^+\cap J_V \subseteq I_V$, $[\mu+\alpha_p](\alpha_p^{\vee})\geq 2 >0$ \big(as $\mu(\alpha_p^{\vee})\geq  0$ by (C1)\big) implies $\mu\in \wt V$ by Lemma \ref{Lemma real root strings}.
		
		Now, it only remains to show that $\delta$ satisfies (C2) -- namely finding $\delta'$ -- and we begin by setting \begin{center}
		$\delta''\ \ :=\ \ \lambda\ -\alpha_{j_1} -\ \sum\limits_{\substack{k\in \{2,\ldots, n\} \  \text{s.t. }|J_k|>1 }}\alpha_{j_k}-\sum\limits_{\substack{k\in \{2,\ldots, n\} \  \text{s.t. }|J_k|=1 
			}}c_{j_k}\alpha_{j_k}\ \ \ \in\  \wt V$.
			\end{center}
			$\delta''\in \wt V$ since both of $\delta''\succneqq \mu'\in \wt V$ have independent supports.
			Recall  $T=\supp\big(\lambda-w(\gamma+\alpha_p)\big)$ (possibly independent), and we let $T=T_1,\ldots, T_m$ to be the Dynkin subgraph components of $T$ (when $T\neq\emptyset$); $m=1$ if $T$ is connected.
            When $T=\emptyset$, $\delta=\lambda-\sum_{j\in J\setminus J_1}c_j\alpha_j$ with (C1), and setting $\delta'=\delta''+\alpha_{j_1}$ suffices. Henceforth we assume $T\neq\emptyset$, and let
			\begin{center}
			 $w(\gamma+\alpha_p)\ =\ \lambda-\sum_{t\in T}c_t'\alpha_t
			\quad \text{for  }\ c_t'\in \mathbb{Z}_{> 0}\ \forall\ t$.
			\end{center}
			While $T_t\cap J_V^c\neq \emptyset$ $\forall$ $1\leq t\leq m$, picking any (set of) nodes $j'_t\in T_t\cap J_V^c$ $\forall$ $t$, it suffices to set
			\begin{center}
			$\delta'\ \ =\ \ (\delta''+\alpha_{j_1})\ -\ \sum\limits_{\substack{t\in \{1,\ldots, m\}\  \text{s.t. }|T_t|=1}}c_{j_t'}\alpha_{j_t'}\ -\ \sum\limits_{\substack{t\in \{1,\ldots, m\}\ \text{s.t. }|T_t|>1 }}\alpha_{j_t'} \qquad \underset{(\text{by Proposition } \ref{Proposition Minkowski difference formula})}{\in \wt V}$.
			\end{center}
            We henceforth assume w.l.o.g. for some $1< r\leq m$ that $T_t\subseteq J_V$ $\forall$ $t\leq r$ and $T_d\nsubseteq J_V$ $\forall$ $r<d\leq m$; $r=m$ if all $T_t\subseteq J_V$.
			Note $j_1$ $\in\supp(\lambda-\mu')\cap J_1$ can belong to at most one $T_t$; or no $T_t$.
			We fix by $w(\gamma+\alpha_p)\in \wt V'$,  nodes $\tau_t\in T_t$ $\forall$ $1\leq t\leq r$ s.t. $\lambda-\alpha_{\tau_t}\in \wt V'$ (by Lemmas \ref{weight formula 1 forward inclusion lemma} or \ref{Lemma reaching lambda from mu}).
            Recall, $w(\gamma+\alpha_p)$ is $J_1\cap I_V$-dominant, but nevertheless $\tau_t$s can be chosen from its $J_1\cap I_{V'}$-dominant conjugate.
			We extend this sequence $\tau_1,\ldots,\tau_r$ by choosing (any) $\tau_{r+1}\in T_{r+1}\cap J_V^c, \ldots, \tau_m\in T_m\cap J_V^c$. 
			Recall, $f_{j_1}V'_{\lambda}\neq 0$ in $V'$.
           i) For simplicity, we assume $t=1$ and $\tau_1=\{j_1\}$  whenever $j_1\in T_t$.
			For $1\leq t\leq r$, whenever $\tau_t\neq j_1$, necessarily $\tau_t\in J_V\cap (\mathcal{I}^+\sqcup \mathcal{I}^-)$ and $\lambda(\alpha_{\tau_t}^{\vee})\neq 0$.
            This is because $\tau_t\in J_V\cap \mathcal{I}^0\setminus \{j_1\} \text{ says } f_{\tau_t}V'_{\lambda}=0$, and $f_{\tau_t}^{\frac{2}{A_{\tau_t\tau_t}}\lambda(\alpha_{\tau_t}^{\vee})+1}V'_{\lambda}=0\neq f_{\tau_t}V'_{\lambda}$\  $\forall$ $\tau_t\in J_V\cap (\mathcal{I}^+\sqcup \mathcal{I}^-)\setminus \{j_1\}$.\\
			ii) Observe by the same reasoning $c'_{\tau_t}< \frac{2}{A_{\tau_t\tau_t}}\lambda(\alpha_{\tau_t}^{\vee})+1$, when $T_t=\{\tau_t\neq j_1\}\subseteq \mathcal{I}^+\sqcup \mathcal{I}^-$ for $ t\leq r$, or when $T_1= \{\tau_1=j_1\}\subseteq \mathcal{I}^+\sqcup \mathcal{I}^-
            $ with $\lambda(\alpha_{\tau_1}^{\vee})\neq 0$; recall $\lambda-\sum_{t\in T}c'_t\alpha_t=w(\gamma+\alpha_p)\in \wt V'$.
			All of the above observations help us identify the sought for weight $\delta'$ as follows (lastly if  $T\cap J_V =T_1\sqcup \cdots \sqcup T_r$):
			\begin{equation}\label{Eqn delta' in the final sub case}
			\delta'\ \ =\ \ \Bigg( \widetilde{\delta}=   \Bigg[\delta''-\sum_{\substack{t\in \{1,\ldots,r\}\\ \text{s.t. }|T_t|>1 }}\alpha_{\tau_t}-\sum_{\substack{t\in \{1,\ldots,r\}\\ \text{s.t. }|T_t|=1 }}c'_{\tau_t}\alpha_{\tau_t}\Bigg]\Bigg) -\ \sum_{\substack{t\in \{r+1,\ldots,m\}\\ \text{s.t. } |T_t|>1}}\alpha_{\tau_t}\ -\ \sum_{\substack{t\in \{r+1,\ldots,m\} \\ \text{s.t. }|T_t|=1 }}c'_{\tau_t}\alpha_{\tau_t}.
			\end{equation}
			We will be done up on showing that $\delta'$ defined in \eqref{Eqn delta' in the final sub case} is in $\wt V$, and in view of \eqref{Eqn freeness inclusion} it suffices to show the claim $\widetilde{\delta}\in \wt V$ (i.e., coming down from $\delta''$ in $\alpha_{\tau_1}, ..., \alpha_{\tau_r}$ directions).
            For this, we observe: 
		(i) $\delta''\in \wt V$, and when $j_1\notin T$ we work with $\delta''+\alpha_{j_1}\in \wt V$ instead of $\delta''$. 
        (ii) $\alpha_j(\alpha_{\tau_{t}}^{\vee})=0$ $\forall$ $j\in J_2\sqcup\cdots \sqcup J_n\sqcup \{\tau_1,\ldots,\cancel{\tau_t}, \ldots, \tau_r\}$, and so $\delta''(\alpha_{\tau_t}^{\vee})= [\delta''+\alpha_{j_1}](\alpha_{\tau_t}^{\vee})=\lambda(\alpha_{\tau_t}^{\vee})\neq 0$ $\forall$ $\tau_t\neq j_1$. 
        (iii) So by point ii) in the above paragraph and Lemma \ref{Lemma real root strings}, we can iteratively come down being in $\wt V$ from $\delta''$ or $\delta''+\alpha_{j_1}$ to $\widetilde{\delta}$, by each $c_{\tau_t}'\alpha_{\tau_t}$ (if $|T_t|=1$) or by $\alpha_{\tau_t}$ (if $|T_t|>1$) for $\tau_t\neq j_1$.  
        (iv) Now let $\tau_1=j_
        1\in T_1$.
        If $|T_1|>1$, then $\widetilde{\delta}$ with one $\alpha_{j_1}$ works.
        So we assume $T_1=\{j_1\}$.
        When $j_1\in \mathcal{I}^0$ we have $f_{j_1} V_{\lambda}'\neq 0 = f_{j_1}^2V_{\lambda}'$ and so $c_{j_1}'=1= \height_{\{j_1\}}(\lambda-\widetilde{\delta})$; and $\widetilde{\delta}$ works.
        Suppose $j_1\in \mathcal{I}^+\sqcup \mathcal{I}^-$ (recall $j_1\in T_1\subseteq  J_V\subseteq J_{\lambda}$);  we  go below $\widetilde{\delta}$ which is in $\wt V$ by previous lines.
        When $\lambda(\alpha_{j_1}^{\vee})=0$ by $\mathbb{C}[f_{j_1}]\cdot V_{\widetilde{\delta}}$ being $\mathfrak{g}_{\{j_1\}}$-simple-Verma, and when $[\widetilde{\delta}+\alpha_{j_1}](\alpha_{j_1}^{\vee})=\lambda(\alpha_{j_1}^{\vee})<0$ similar to in point (iii), we have $[\widetilde{\delta}+\alpha_{j_1}]- c'_{j_1}\alpha_{j_1}\in \wt V$.
		Hence, the proofs of Step 5 and theorem, are complete.
		\end{proof}
        \begin{proof}[{\textnormal \bf Proof of Theorem \ref{Theorem wt-formula by composition series simples for nice V}}]
    Fix an integrable h.w.m. $V\twoheadleftarrow M(\lambda)$.
 The reverse inclusion in \eqref{Eqn wt-fromula via comp. series for V in BKM case} is true, as maximal vectors with weights the surviving holes in $V$, generate non-zero h.w. submodules.
 
        For the forward inclusion, we show $\mu\in \wt V$ lies in the union on the r.h.s. of \eqref{Eqn wt-fromula via comp. series for V in BKM case}, by inducting on $\height(\lambda-\mu)\geq 0$; it is manifest if $\mu=\lambda$. 
        For the induction step, let $\mu=\lambda-\sum_{i\in I}c_i\alpha_i$  for $c_i\in \mathbb{Z}_{>0}$ $\forall$ $i\in I=\supp(\lambda-\mu)$.
       Suppose $I$ is independent.
       We set $I_{\mu}:=\big\{ i\in I\cap J_{\lambda}\cap (\mathcal{I}^+\sqcup \mathcal{I}^-)\ \big|\ c_i>\frac{2}{A_{ii}}\lambda(\alpha_i^{\vee}) \big\}$ and $\xi=\lambda-\sum_{i\in I_{\mu}} \big(\frac{2}{A_{ii}}\lambda(\alpha_i^{\vee})+1\big)\alpha_i-\sum_{i\in I\cap J_{\lambda}\cap \mathcal{I}^0}c_i\alpha_i$. 
       We see $\mu\in \wt L(\xi)$, via observing $\mu_1=\lambda - \sum_{j\in I_{\mu}\cup (I\cap J_{\lambda}\cap \mathcal{I}^0 ) }c_j\alpha_j \in \wt L(\xi)$, using Lemmas \ref{Lemma real root strings}, \ref{Freeness-directions for weights lemma} for pair \big($\mu_1$, $L(\xi)$\big) along directions $ I\cap J_{\lambda}\cap (\mathcal{I^+}\sqcup \mathcal{I}^- \setminus I_{\mu})\subseteq  J_{\xi}$ and resp. $(I\cap J_{\lambda}^c)\cup I_{\mu}\subseteq J_{\xi}^c$.
       
     We assume  $I$ to be non-independent.      If $\mu(\alpha_j^{\vee})<0$ for a $j\in I_V=\mathcal{I}^+$ (as $V$ is integrable), by the induction hypothesis  $s_j\mu$ lies in the r.h.s. of \eqref{Eqn wt-fromula via comp. series for V in BKM case}, and so does $\mu$ by Lemma \ref{Lemma real root strings}.
      So we assume $\mu\in \wt V$ to satisfy (C1) along $I\cap \mathcal{I}^+$, and let $I=J_1\sqcup \cdots\sqcup J_n$ be the Dynkin subgraph component decomposition as in (C2), with w.l.o.g. $|J_1|>1$.
We pick a node $i\in J_1$ (by Lemma \ref{Lemma reaching lambda from mu}) for which the maximal weight $\mu'$ in the $\alpha_i$-string $(\mu+\mathbb{Z}_{>0}\alpha_i)\cap \wt V$ though $\mu$ -- in particular, its ``$J_1$-projection'' $\big(\lambda-\sum_{t\in J_1}c_t\alpha_t\big)+ \height_{\{i\}}(\mu'-\mu)\alpha_i$ --  lies in $\wt V$.
By the induction hypothesis, we fix a hole $(H, m_H)\in \mathrm{Indep}(J_{\lambda})\setminus \mathcal{H}_V$ with $\xi= \lambda-\sum_{h\in H}m_H(h)\alpha_h\in \wt V$ and $\mu'\in \wt L(\xi)\subseteq \wt V$.
So $\mu'$ is also at the top of its $\alpha_i$-string in $\wt L(\xi)$.
If $i\in \mathcal{I}^+$, as $\mu, \mu'$ are $\{i\}$-dominant, and as $V$ is integrable, Lemma \ref{Lemma real root strings} for $\big(\mu', i, L(\xi)\big)$ shows $\mu\in \wt L(\xi)$ as required.
So we assume $i\in \mathcal{I}^0\sqcup \mathcal{I}^-$.
If $\lambda(\alpha_i^{\vee})\geq 0$, then $\mu'(\alpha_i^{\vee})>0$ as imaginary $i$ is adjacent to some node in $J_1$, and then $\mu'\succneqq \mu\in \wt L(\xi)$ as $\mathbb{C}[f_i]\cdot  L(\xi)_{\mu'}$ are $\mathfrak{g}_{\{i\}}$-Verma lines by Lemma \ref{Lemma rank-1 BKM rep. theory}.
We assume that $\lambda(\alpha_i^{\vee})<0$, and note $\lambda - \alpha_i\in \wt V$.
Moreover, the pairing of $\alpha_i^{\vee}$ with (sub)non-hole weight $\gamma=\lambda - \sum_{h\in H\setminus J_1}m_H(h)\alpha_h\in \wt V$ is equal to $\lambda(\alpha_i^{\vee})<0$, and so $\gamma - \alpha_i$ is also a non-hole weight in $V$.
We show $\mu\in \wt L(\gamma)$, by proving $\mu$ to satisfy (C2) w.r.t. the top $\gamma$; clearly $\mu\preceq \gamma$ and (C1) is satisfied along the directions of $J_{\gamma}\cap \mathcal{I}^+\subseteq \mathcal{I}^+=I_V$.
For this, we begin by noting by $\mu'$ satisfying (C2) w.r.t. $\xi$, that $\lambda- \sum_{t\in I\setminus J_1}c_t\alpha_t$ satisfies (C2) w.r.t. $\gamma$.
Next, $J_1$ is a component in $\supp(\mu-\gamma)$, and fix the desired node $j_1\in J_1$ to be $i$.
In the decomposition $\supp(\gamma-\mu)$ into components say as $J_1\sqcup K_2\sqcup \cdots \sqcup  K_m$ -- observe $K_2,...,K_m$ are also components in $\xi-\mu'$ and by Lemma \ref{weight formula 1 forward inclusion lemma} for $\mu'\in \wt L(\xi)$ -- we fix the desired nodes from $K_2,\ldots, K_m$ to be $j_2,\ldots,j_m$ . 
Let the supporting 1-dim. weight for $\lambda - \sum_{t\in I\setminus J_1}c_t\alpha_t$ in $\wt L(\gamma)$ \big(given its property (C2)\big) involving $j_2,\ldots, j_m$ be $\nu$.
Finally, for observing $\mu\in \wt L(\gamma)$, the desired 1-dim. weight for its (C2) is $\nu-\alpha_i$; and $\nu-\alpha_i\in \wt L(\gamma)$ by $\mathfrak{sl}_2$-theory as $\nu(\alpha_i^{\vee})= \lambda(\alpha_i^{\vee})<0$.  
\end{proof}      
\section{Theorem \ref{Theorem C} : Characters of simples $L(\lambda\in P^{\pm})$ over rank-2 BKM~$\mathfrak{g}$}\label{Section 5}
            We study  Vermas $M(\lambda)$, and simples $L(\lambda)$ and their characters, for $\lambda\in P^{\pm}$. 
            We comupte maximal vectors in $M(\lambda)_{\mu}$ for solutions $\mu$ to norm-equality \eqref{Eqn norm equality with invariant form}, in the cases in Theorem \ref{Theorem C}.
            We study, solution-data to \eqref{Eqn norm equality with invariant form} in Subsections \ref{Subsection 5.1}--\ref{Subsection Theorem C (I)--(II) proof}.
            Recall the discussion above \eqref{Eqn dot action by any positive root}.
 In the settings of $\mathcal{I}=\mathcal{I}^-$ with (widely treated case in the literature of) all $A_{ij}\neq 0$ \cite{Elizabeth, Naito 2} : 1) The whole ``interior'' $\mathbb{Z}_{\geq 0}\Pi\setminus \big(\bigcup_{i\in \mathcal{I}}  \mathbb{Z}_{\geq 0}\{\alpha_i\}\big)\text{ equals }\Delta^+\setminus \Pi$. 
2) See Note \ref{NOte KK maxl vect bounds}a$'$) for the lower bound from \cite{Kac--Kazhdan}, for the number of linearly independent maximal vectors in $M(\lambda)_{\mu}$ for $\mu\preceq \lambda$ satisfying \eqref{Eqn dot action by any positive root}.
\subsection{Norm equality in rank-2, and monomial maximal vectors in Proposition \ref{Prop maxl vect}(a)}\label{Subsection 5.1}
       The rank-2 settings that we work inside in the entirety of this Section, are as follows:
			\begin{itemize}
				\item $A=A(b,a,c,d)$ defined in Setting-A of \eqref{Eqn negative settings of A} for $b,d\in \mathbb{Z}_{\geq 0}\sqcup\{-2\}$, $a,c\in \mathbb{Z}_{\geq 0}$; widely, \underline{when $b,a=c,d\in \mathbb{N}$}, see Lemma \ref{Lemma norm equations via Casimir action}.
                 Such $A$ are symmetrizable upon scaling $\Pi^{\vee}$.
            \item Fix any $\lambda\in \mathfrak{h}^*$ with $\lambda(\alpha_i^{\vee})\ = \ \frac{A_{ii}}{2} (M_{i}-1)$, for (powers) $M_i\in \mathbb{N}$, with $M_i=1$ if $A_{ii}=0$.
				\item Over $A(2)$, $\rho(\alpha_i^{\vee})=-1$ (negative) $\forall\ i$; differing from KM cases of $\rho(\alpha_i^{\vee})=2$ $\forall$ $i\in \mathcal{I}=\mathcal{I}^+$.
                \end{itemize}
                Lemma \ref{Lemma norm equations via Casimir action} reformulates \eqref{Eqn norm equality with invariant form}, in those rank-2 cases. 
               The elementary forms (N)--(R) of $A$ in it, can help build intuitions on structures of $M(\lambda)$ and $L(\lambda)$s for $\lambda\in P^{\pm}$; here, if $\mathcal{I}^-=\emptyset$, we are back in integrable case $P^{\pm}=P^+$.  
            Even in rank-2, if $\{a,b\}\not\subseteq\mathbb{Z}_{\geq 0}$, the analysis could be more involved bringing-in strong dot-action of arbitrary positive roots; while the remaining cases in $\mathbb{N}$ proceed similar to in $A_2$-type. 
           The simplest case $a=c=0$:\  $\mathfrak{g}$ =  $\mathfrak{sl}_2^{\oplus 2}$ ($A_1\times A_1$)~or~$\mathfrak{sl}_2(\mathbb{C}) \oplus\text{[3-}\text{dim. Heisenberg LA]}$, and $\mathfrak{g}$-Vermas are direct sums of individual Verma lines over rank-1 $\mathfrak{g}_{\{i\}}$s.
            For general $A_{n\times n}$, we have to deal with 2-nd (total monomial) degree equations in $n$ variables.    
			\begin{lemma}\label{Lemma norm equations via Casimir action} Fix a symmetric BKMC matrix $A=A(b,a,c=a,d)_{2\times 2}$ that is not a generalized Cartan matrix; so the BKM LA $\mathfrak{g}=\mathfrak{g}(A)$ is not any KM LA.
         Assume $\mathcal{I}^-\neq \emptyset$.
         Fix weights $\lambda\in  P^{\pm} $ and $\mu=\lambda-X\alpha_1-Y\alpha_2\preceq \lambda$, for $(X,Y)\in \mathbb{Z}_{\geq 0}\times \mathbb{Z}_{\geq 0}$.
            If $M(\lambda)_{\mu}$ has a maximal vector, then one of equations \eqref{Eqn norm equality condition for general mu= (X,Y) in Case (N)}--\eqref{Eqn norm equality condition for general mu= (X,Y) in Case (R)} below must be satisfied by the pair $(X,Y)$ in the respective settings when
             \begin{center}
           $(N)\ \  A = \begin{bmatrix}
				-b\ &\ -a\\
				-a\ &\ -d
				\end{bmatrix} \ \  ; \ \ (H)\ \  A = \begin{bmatrix}
				-b\ &\ -a\\
				-a\ &\ 0
				\end{bmatrix}\ \  ;  \ \ (R)\  \  A = \begin{bmatrix}
				-b\ &\ -a\\
				-a\ &\ 2
				\end{bmatrix}\quad \text{for }a,b,d\in \mathbb{N}\  :$
                \end{center}
                 (N)--(R) arise exactly when $\mathcal{I}=\{1,2\}$ with $\{1\}\subseteq \mathcal{I}^-$, and $\{2\}$ resp. negative/Heisenberg/real.
           \begin{align}
			 \tag{Eq N}\ b X^2\ + \ dY^2\ - \  bM_1 X\ - \ dM_2Y \ +\ 2aXY\ & = \ 0.
             \label{Eqn norm equality condition for general mu= (X,Y) in Case (N)}
			\\ 
			  \tag{Eq H} \ \ X\Big(X\ - \ M_1 \ +\ \frac{2a}{b}Y\Big)\ & = \ 0.
           \label{Eqn norm equality condition for general mu= (X,Y) in Case (H)}
			\\
			    \tag{Eq R} \ b X^2\ - \ 2 Y^2\ - \  bM_1 X\ + \ 2M_2Y \ +\ 2aXY\ & = \ 0.
         \label{Eqn norm equality condition for general mu= (X,Y) in Case (R)}
			\end{align}
                        \end{lemma}
          \begin{proof}[{\bf Proof}]
            By the norm equality due to Casimir action 
            \eqref{Eqn norm equality with invariant form}: \begin{equation}\label{Eqn norm equality simplified}
                (\lambda +2\rho +\mu\ , \ \lambda-\mu)\ = \ 0
            \end{equation}
           \begin{align*}
			 \implies &\  2\left(\lambda+\rho\ , \ X\alpha_1+ Y\alpha_2\right) - \left(X\alpha_1+Y\alpha_2\ , \ X\alpha_1+Y\alpha_2\right) \\
	 = \ & \frac{2(\alpha_1\ , \ \alpha_1)}{A_{11}} \left(\lambda+\rho\ , \  \frac{A_{11}}{(\alpha_1\  , \ \alpha_1)} \alpha_1\right)X\  + \  \underset{(\text{When }A_{22}\neq 0)\qquad \qquad \quad}{\frac{2(\alpha_2\ , \ \alpha_2)}{A_{22}} \left(\lambda+\rho\ , \  \frac{A_{22}}{(\alpha_2\  , \ \alpha_2)} \alpha_2\right)Y} \  -   (\alpha_1 , \ \alpha_1)X^2 \ \\
			&\ \ \ \ - \ (\alpha_2\ , \ \alpha_2)Y^2  \ - \ \frac{2(\alpha_1\ , \  \alpha_1)}{A_{11}}\left(\alpha_2\ , \ \frac{A_{11}}{(\alpha_1\ , \ \alpha_1)} \alpha_1\right)XY\  \\
			=  \ & (\alpha_1\ , \  \alpha_1)X^2\ + \ (\alpha_2\  , \  \alpha_2)Y^2\ - \  (\alpha_1\ , \ \alpha_1)M_1 X\ - \ (\alpha_2\ , \  \alpha_2)M_2Y \ +\ 2(\alpha_1\ , \  \alpha_1)\frac{A_{12}}{A_{11}}XY  =  0 .
			\end{align*}
            Now the equations in the lemma are easily seen to arise.
            \end{proof}
          \begin{itemize}
          \item \underline{Easy solutions to \eqref{Eqn norm equality condition for general mu= (X,Y) in Case (N)}--\eqref{Eqn norm equality condition for general mu= (X,Y) in Case (R)}} 
           $(X,Y)\ =\ (0,0),\ (M_1,0),\ (0,M_2)$, i.e. $\mu=\lambda,\        \lambda    -M_1\alpha_1, \ \lambda-  M_2\alpha_2$.
           \item The solution-set for  \eqref{Eqn norm equality condition for general mu= (X,Y) in Case (H)} = $\big\{(0,Y)\ \big|\  Y\in \mathbb{N}\big\}\cup \Big\{ \big(  M_1-\frac{2a}{b}Y ,\  Y\big) \ \Big|\ \mathbb{N}\ni Y \leq \frac{b}{2a}M_1 \Big\}$.
           \item In several small cases of \eqref{Eqn norm equality condition for general mu= (X,Y) in Case (R)} we checked, it had infinitely many solutions.\smallskip\\
           \underline{Example }: Consider $\lambda=0$ over $\mathfrak{g}\big(A(2,2,2,-2)\big)$. \eqref{Eqn norm equality condition for general mu= (X,Y) in Case (R)} in this case reads as  $X^2-Y^2-X+Y+2XY=0$. 
           So $-(X-Y)^2- (X-Y)+ 2X^2= 0 $ implies $X= \sqrt{\frac{(X-Y)(X-Y+1)}{2}}$.
           Now $X\in \mathbb{N}$ if both $(X-Y)$ and $\frac{X-Y+1}{2}$ are perfect squares.
           This brings us to the famous negative \underline{Pell's equation} $u^2-2v^2=-1$ with infinitely many solutions $(u,v)\in \mathbb{Z}\times \mathbb{Z}$.
           For each such $u$, finding $(X,Y)$ using $X=uv$ and $X-Y=u^2$ suffices for the claim.
           It might be interesting to explore the arising of some other prominent equations in number theory from the norm-equality. 
        \item \eqref{Eqn norm equality condition for general mu= (X,Y) in Case (N)} always has finitely many solutions, which the rest of the paper carefully studies.
        \end{itemize}            
            \begin{lemma}\label{Lemma bddness of sols for two negative nodes}
                All the $\mathbb{Z}_{\geq 0}\times \mathbb{Z}_{\geq 0}$-solutions to \eqref{Eqn norm equality condition for general mu= (X,Y) in Case (N)} lie within the square $\Big[0,\  \max\Big(\frac{b}{2a},  \frac{d}{2a} , 1\Big)\times \max\big(M_1,M_2\big) \Big]^2$. \big(It is also the smallest square containing all the solutions when $b=d=2a$.\big)
            \end{lemma}
            \begin{proof}[{\bf \textnormal Proof.}]
                We fix a solution $(x,y)\in \mathbb{Z}_{\geq 0}^2$ to \eqref{Eqn norm equality condition for general mu= (X,Y) in Case (N)}, and assume w.l.o.g. $M_1\geq M_2$. 
                Let  $c:=\max\Big(\frac{b}{2a}, \frac{d}{2a},1\Big)\geq 1$.                If $\min(x, y)\ >\ cM_1\geq M_1$, then $x(x-M_1)+\frac{d}{b} y(y-M_2)+ \frac{2a}{b}xy>0$ contradicts  \eqref{Eqn norm equality condition for general mu= (X,Y) in Case (N)}.
           Next we assume $\max(x,y)\ > \ cM_1\ \geq \ \min(x,y)$.
                If $\max(x,y)=y$ and  resp. if $\max(x,y)=x$,  by $bx^2+dy(y-M_2)+bx\Big(\frac{2a}{b}y-M_1\Big)>0$ and resp. $dy^2+bx(x-M_1)+dy\Big(\frac{2a}{d}x-M_2\Big)>0$, we see the same contradiction.
                So we must have $0\leq x,y\leq c\max(M_1,M_2)$, as desired. 
                         \end{proof}
			\begin{question}
			    Can one  characterize the solutions to \eqref{Eqn norm equality condition for general mu= (X,Y) in Case (R)} in rank-2?
Can one find some bigger example-cases of $\mathcal{I}=\mathcal{I}^-\sqcup \mathcal{I}^0$ where-in solutions admit explicit descriptions as in $A(b,a,c,0)$ case?
When $b=d=2a$ are all (finitely many) solutions to \eqref{Eqn norm equality condition for general mu= (X,Y) in Case (N)} explicitly describable?
                  The illuminating subcase for it might be the negative-$\mathfrak{sl}_3(\mathbb{C})$ with  for $A(2)$ \big(which Proposition \ref{Prop number theory} studies\big) involving following  simplification \eqref{Eqn symmetric and same lengths case} of \eqref{Eqn norm equality condition for general mu= (X,Y) in Case (N)}, which we focus upon.			\begin{equation}\label{Eqn symmetric and same lengths case}
			\ X^2\ + \ Y^2\ - \  M_1 X\ - \ M_2Y \ +\ XY\ = \ 0. 
			\end{equation} 
			\end{question}
We show Proposition \ref{Prop maxl vect} (a) now, more generally over $\mathfrak{g}\big(A(b,a,c=a,d)\big)$ to build some insight; parts (b) and (c) are proved in Subsections \ref{Subsection Theorem C (III) proof}.
		\begin{lemma}[{Negative $\mathfrak{sl}_3$-theory}]
\label{Lemma maximal vect in solution (1,n) for case (1, M2)}
  		Fix $\mathfrak{g}=\mathfrak{g}\big(A(b,a,a,d)\big)$ for $a\in \mathbb{Z}_{\geq 0}$, $d\in \mathbb{Z}_{\geq 0}\sqcup \{-2\}$ and $b\in \mathbb{N}$; $\mathcal{I}=\{1,2\}$; and $\lambda\in  P^{\pm} $ with powers $M_1,M_2\in \mathbb{N}$.
Note here: $a=0$ includes the two decomposable ($A_1\times A_1$, ``$A_1\times A_0$'' typed) cases of $A_{2\times  2}$; $1\in \mathcal{I}^-$; and $M_2=1$ if $d=0$ (Heisenberg case).          When $d\neq 0$ (i.e., $2\in \mathcal{I}^-\sqcup \mathcal{I}^+$), the following two statements (a) and (b) are equivalent:
            \begin{itemize}
					\item[(a)] $(n, M_2)$  is a solution to norm equality \eqref{Eqn norm equality condition for general mu= (X,Y) in Case (N)}, or  to \eqref{Eqn norm equality condition for general mu= (X,Y) in Case (R)}, for $n =   M_1-\frac{2a}{b}M_2  \in \mathbb{N}$; since $(0, M_2)$ is already solution, notice such an `$n$' is unique.
                   \item[(b)]
                    $f_1^n f_2^{M_2} m_{\lambda}$ is maximal in  $M(\lambda)$; in other words $\dim\big[\mathrm{Hom}_{\mathfrak{g}}\big(M(\lambda-n\alpha_1-M_2\alpha_2),\ M(\lambda)\big)\big]\geq 1$.
	\end{itemize}
Fix any $k\in \mathbb{N}$; note $(0,k)$ satisfies \eqref{Eqn norm equality condition for general mu= (X,Y) in Case (H)}.
For $(n, k)$ satisfying \eqref{Eqn norm equality condition for general mu= (X,Y) in Case (H)}, the equivalence of statements (a) and (b) above is true with $k$ in place of $M_2=1$. 
\end{lemma}
\begin{note}\label{Note negative Sl3-theory, string reversal} 
For $A = A(2,a,a,d)$, (a)$\implies$(b) in Lemma \ref{Lemma maximal vect in solution (1,n) for case (1, M2)} says 
				$f_1^{M_1-M_2 
|\alpha_2(\alpha_1^{\vee})| }f_2^{M_2}m_{\lambda}$ is maximal. 
It is similar to the usual  $\mathfrak{sl}_3(\mathbb{C})=\mathfrak{g}\big(A(-2,1,1,-2)\big)$ theory saying  $f_1^{M_1+M_2|\alpha_2(\alpha_1^{\vee})|}f_2^{M_2}m_{\lambda}$ is maximal, but with the reversal/shrink in the way the string of $f_1$ or $-\alpha_1$ proceeds. 
Next $\frac{2a}{b}$ in $n$ need not be in $\mathbb{Z}_{\geq 0}$.
Note $\lambda-n\alpha_1-M_2\alpha_2$ resembles $s_{\alpha_1}\bullet (s_{\alpha_2}\bullet \lambda) = s_{\alpha_1}\bullet (\lambda-M_2\alpha_2)$. 
In part (b), $\dim \mathrm{Hom}_{\mathfrak{g}}\big(M(\mu), M(\lambda)\big) \geq \dim\mathfrak{g}_{\lambda-\mu}$ (as $\mu$ is non-minimal) by \cite[Section 4]{Kac--Kazhdan}.
And Proposition \ref{Prop maxl vect} computes these dimensions when one of the components of $\mu$ is bounded by $2$.
Are there maximal vectors appearing as a product of negative root vectors in addition to $\mathbb{C}f_1^nf_2^k m_{\lambda}$ in $M(\lambda)_{\lambda-n\alpha_1-k\alpha_2}$? 
			\end{note}
        \begin{proof}[{\bf \textnormal Proof of Lemma \ref{Lemma maximal vect in solution (1,n) for case (1, M2)}, thereby of Proposition \ref{Prop maxl vect}(a).}]
		Suppose $f_1^nf_2^{M_2}m_{\lambda}\neq 0$ is maximal.
         $(b)$ $\implies (a)$ in the lemma, follows by below implication due to 
        $(n, M_2)$ (by Casimir-action) satisfying one of \eqref{Eqn norm equality condition for general mu= (X,Y) in Case (N)}--\eqref{Eqn norm equality condition for general mu= (X,Y) in Case (R)}:
                \begin{center} 				$bn^2\pm  dM_2^2 -   bM_1n\mp  dM_2^2  \pm 2aM_2n =  \ 0 \ \     
				\implies \ \    n\big(bn - bM_1 \pm 2aM_2\big) \ =  \ 0 \ \   \implies  \ \ n \ =\ M_1 \mp \frac{2a}{b}M_2$   
				\end{center}
                 Conversely, fix $n$ in (a).
				Note $e_2f_2^{M_2}m_{\lambda}=0= e_2 f_1^n f_2^{M_2}m_{\lambda}$.
				Now $e_1 \cdot f_1^n f_2^{M_2}m_{\lambda}\  = \ n\big[\lambda(
				\alpha_1^{\vee})- M_2\alpha_2(\alpha_1^{\vee}) - \big(\frac{n-1}{2}\big)\alpha_1(\alpha_1^{\vee})\big]f_1^{n-1}f_2^{M_2}m_{\lambda}\ = \ n\big[\frac{b}{2}(n-M_1)+aM_2 \big]f_1^{n-1}f_2^{M_2}m_{\lambda}=0$ by the value of $n$, proves $(a)\implies (b)$.
                The result in Heisenberg case is similarly provable.
	\end{proof}     

\subsection{$\mathrm{\bf char}\big(L(\lambda\in  P^{\pm})\big)$ over $\mathfrak{g}\big(A(b,a,c=a,d=b)\big)$, $a,b\in \mathbb{N}$, when $(1,1)\text{ or } (2,2)$ satisfy~\eqref{Eqn symmetric and same lengths case and any a}}\label{Subsection Theorem C (I)--(II) proof}
Using the counts of maximal vectors in Proposition \ref{Prop maxl vect} (a)--(c), we quickly show character formulas in Theorem~\ref{Theorem C} of simples $L(\lambda\in P^{\pm})$ after the following observation that motivated the cases in it. 
\begin{lemma}\label{Lemma (k,k) solutions, consequences}
Assume $(k,k)$ to satisfy \eqref{Eqn symmetric and same lengths case and any a}, for a rational number $k>0$.
				Suppose $(X,Y)$ is a solution to \eqref{Eqn symmetric and same lengths case and any a} with both $X, Y\geq k$. Then $X=Y=k$.   
                In particular, such a $k$ is unique.
			\end{lemma}
			\begin{proof}[{\bf \textnormal Proof.}]
				Fix solutions $(X, Y)$ and $(k,k)$ as in the lemma; note in \eqref{Eqn symmetric and same lengths case and any a} the $XY$-coefficient is  $\frac{M_1+M_2}{k}-2$.
		Re-writing \eqref{Eqn symmetric and same lengths case and any a} as $(X-Y)^2+ M_1X\Big(\frac{Y}{k}-1\Big)+ M_2Y\Big(\frac{X}{k}-1\Big)=0$, the result is easily seen. 
			\end{proof}
\begin{cor}\label{Corollary (1,1) is unique sol in 1st Quadrant}
Assume that $(1,1)$ is a solution to \eqref{Eqn symmetric and same lengths case and any a}.
Then it is the only solution in $\mathbb{Z}_{>0}\times \mathbb{Z}_{>0}$.
Furthermore $\max\{M_1, M_2\}>1$, as $2a/b>0$.
If we take $a=0$, then
$(1,0),  (0,1),  (1,1)$ are solutions to \eqref{Eqn symmetric and same lengths case and any a}; we are in setting (N).
There are infinitely many cases where-in $(1,1)$ satisfies \eqref{Eqn symmetric and same lengths case and any a} : for any $b\in 2\mathbb{N}$, $M_1>M_2\in \mathbb{N}$ (any large) and $a=\frac{b}{2}(M_1+M_2-2)$; differing from $\mathfrak{sl}_3(\mathbb{C})$-case.
\end{cor}
			\begin{proof}[{\bf \textnormal Proof of Theorem \ref{Theorem C} parts (I) and (II)}]
            Fix $\lambda\in P^{\pm}$ with powers $M_1,M_2\in \mathbb{N}$.
            Assume that $(1,1)$ satisfies \eqref{Eqn symmetric and same lengths case and any a} for part (I); 
				Corollary \ref{Corollary (1,1) is unique sol in 1st Quadrant} says $(X,Y)=(0,0), (M_1, 0), (0,M_2), (1,1)$ are all the solutions to \eqref{Eqn symmetric and same lengths case and any a}.
                For formulas \eqref{Eqn char formula with missing norm equatlity solution in (1,1) case} and \eqref{Eqn char formula with non-missing norm equatlity solution in (1,1) case}, we compute the $c$-coefs. on the r.h.s. of : 
				\begin{align}\label{R times char L = solutions condition}
				\begin{aligned}
				R\cdot  \mathrm{char} L(\lambda)\
				\end{aligned}\ = \ 
				\begin{aligned}
				\sum\limits_{\substack{\mu=\lambda-X\alpha_1-Y\alpha_2 \\ \text{satisfies }\eqref{Eqn symmetric and same lengths case and any a}}}c(\mu)e^{\mu}\ 
				\end{aligned}\ = \ 
				\begin{aligned} 
				& c(\lambda) e^{\lambda} + c(\lambda-M_1\alpha_1) e^{\lambda-M_1\alpha_1}+ c(\lambda-M_2\alpha_2)\\
				\quad & e^{\lambda-M_2\alpha_2} + c(\lambda-\alpha_1-\alpha_2)e^{\lambda-\alpha_1-\alpha_2}.
				\end{aligned}
				\end{align}
			As in WKB character formula $c(\lambda)=1 \text{ and }c(\lambda- M_1\alpha_1)=c(\lambda-M_2\alpha_2)=-1$; since $e^{-\alpha_t}$ for $t\in \{1,2\}$, occurs exactly $-1$-times in $R$ and $e^{\lambda+(1-M_t)\alpha_t}$ exactly once in $\mathrm{char} L(\lambda)$.
				For $c(\lambda-\alpha_1-\alpha_2)$ here -- as well as for the $c$-coefs. of bigger root sums in the proofs of (II) and (III) -- we appeal to the celebrated denominator identity by \eqref{Eqn WKB character formula}, which reads in our rank-2 setting of (N) as : 
                \begin{equation}\label{Eqn denominator identity}
                    R \ \ =\ \  \prod\limits_{\alpha\in \Delta^+}\big(1-e^{-\alpha}\big)^{\dim \mathfrak{g}_{\alpha}} \ \ = \ \ 1 \ -\  e^{-\alpha_1}\ - \ e^{-\alpha_2}.
                \end{equation}
			So, $c(\lambda-\alpha_1-\alpha_2)\ = \ [1]\dim L(\lambda)_{\lambda-\alpha_1-\alpha_2}\  +\ [-1]\dim L(\lambda)_{\lambda-\alpha_1} \ + \ [-1]\dim L(\lambda)_{\lambda-\alpha_2}$.
			Fix $i\in \{1,2\}$ for which $M_i =\max\{ M_1,M_2\}\geq 2$; notice $\dim L(\lambda)_{\lambda-\alpha_i}=1$.
           Next $\dim L(\lambda)_{\lambda-\alpha_1-\alpha_2}\leq 1$, since $M(\lambda)_{\lambda-\alpha_1-\alpha_2}$ has a unique (up to scalars) maximal vector by Lemma \ref{Lemma maximal vect in solution (1,n) for case (1, M2)}.
	Now $\lambda-\alpha_i\in \wt L(\lambda)$ implies by Lemma \ref{Lemma imaginary root strings} that $\dim L(\lambda)_{\lambda-\alpha_1-\alpha_2}=1$.
				Hence $c(\lambda-\alpha_1-\alpha_2)=
				\begin{cases}
				\ \ \ 0 \ & \text{if }\min\{M_1,M_2\}=1 \\
				-1\ & \text{if }\min\{M_1,M_2\}>1
				\end{cases}$. 
				Plugging-in the $c$-coefs. in \eqref{R times char L = solutions condition} shows the character formulas in (I).

            For part (II), we fix any Weyl vector $\lambda=\rho\in P^{\pm}$, with powers $M_1=M_2=2$. 
            Observe here by \eqref{Eqn symmetric and same lengths case and any a}, $(X-1)^2+(Y-1)^2+\frac{2a}{b}XY =2$ forces $X,Y\in \{0,1,2\}$ to begin with.
            Next while $X,Y\in \mathbb{Z}_{>0}$, either $(X-1)^2$ or $(Y-1)^2$ in the previous line must vanish, as $a>0$.
           When both vanish we see $a=b$, and when exactly one vanishes we see $b=4a$.
            i) When $a=b$ visibly \eqref{Eqn symmetric and same lengths case and any a} has the solution-set $\{(0,0), (1,1), (2,0), (0,2)\}$ ii) when $b=4a$ the solution-set is $\{(0,0), (2,0), (1,2), (2,1), (0,2)\}$;
            iii)~and in all the other cases the solution-set is $\{(0,0), (2,0),(0,2)\}$.
            In case i), $(1,1)$  visibly being a minimal solution --  i.e., there is no other solution $(X, Y)$  with $X,Y\leq 1$ -- Proposition \ref{Prop maxl vect}(b) says $\dim L(\rho)_{\rho-\alpha_1-\alpha_2}=1$, and so by the computations in part (I) we see that $c(\rho-\alpha_1-\alpha_2)=-1$.
            In the rest of the proof we work in case ii).
          We begin by seeing that  $\dim L(\rho)_{\rho-\alpha_1-\alpha_2}=2$ as there are no non-zero solutions to \eqref{Eqn symmetric and same lengths case and any a} below or equaling $(1,1)$.
            By Proposition \ref{Prop maxl vect}(b) $\dim L(\rho)_{\rho-\alpha_1-2\alpha_2}\leq 2$, and we show equality here; by symmetry, the dimensions for $(2,1)$ is also 2.
            Then $c(\rho-\alpha_1-2\alpha_2)= [1]\dim L(\rho)_{\rho-\alpha_1-2\alpha_2} + [-1]\dim L(\rho)_{\rho-2\alpha_2}+ [-1]\dim L(\rho)_{\rho-\alpha_1-\alpha_2}=0$.
            Recall that $M(\rho)_{\rho -\alpha_1-2\alpha_2}= \mathbb{C}\big\{ f_1f_2^2 m_{\rho}, [f_2,f_1]f_2 m_{\rho}, \big[f_2, [f_2,f_1]\big]m_{\rho} \big\}$ the span of PBW monomial Lyndon word root basis vectors.
            Fix a h.w. vector $0\neq v_{\rho}\in L(\rho)_{\rho}$, and note $f_2^2 v_{\rho}=0$.
            Now for the linear independence of $\big\{ [f_2,f_1]f_2 v_{\rho}, \big[f_2, [f_2,f_1]\big]v_{\rho} \big\}\subset L(\rho)_{\rho-\alpha_2-2\alpha_2}$,
            we see for any $c_1,c_2\in \mathbb{C}$ that $e_2 \big(c_1[f_2,f_1]f_2 v_{\rho}+ c_2 \big[f_2, [f_2,f_1]\big]v_{\rho} \big)= ac_1 f_1f_2 v_{\rho} + \big(-c_1\frac{b}{2}+c_2(2a+b)\big)[f_2,f_1]v_{\rho} = 0 $ in $L(\rho)_{\rho-\alpha_1-\alpha_2}$ (which is 2-dim.) by $a,b>0$ forces $c_1=c_2=0$.
            Thus in both cases of ii) and iii) the desired character-numerator is $e^{\rho} - e^{\rho-2\alpha_1}-e^{\rho-2\alpha_2}$, completing the proof of \ref{Theorem C}(II).
            Finally for Corollary \ref{Corollary presentations gfor nice simples} for presentations, the weight-multiplicities of \eqref{Eqn symmetric and same lengths case and any a}-solution spaces of the below quotients equaling those of $L(\lambda)$ computed in the above proofs of both parts (I) and (II), indeed show :\\
       $L(\lambda)\  = \ \underset{\big(\text{for (I) }\min\{M_1,M_2\}=1\ , \ \text{ or (II) }a\neq b\big)}{\frac{M(\lambda)}{\big\langle f_j^{M_j}m_{\lambda}\ \Big|\ j\in \{1,2\}   \big\rangle} }  \ \    ; \ \     L(\lambda)\  = \ \underset{ \big(\text{in all other cases in (I) and (II)}\big)}{\frac{M(\lambda)}{\big\langle f_j^{M_j}m_{\lambda} \ , \  2af_1f_2m_{\lambda} +b(M_2-1)[f_2,f_1]m_{\lambda}  \ \Big|\ j\in \{1,2\}   \big\rangle}}.$ 
       The proof of Proposition \ref{Prop maxl vect}(b) implies the maximality of $2af_1f_2m_{\lambda}+b(M_2-1)[f_2,f_1]m_{\lambda}$.
       \end{proof}
     
       \subsection{Proofs of Theorem \ref{Theorem C}(III) and Proposition \ref{Prop maxl vect}(b)}\label{Subsection Theorem C (III) proof}
       A useful characterization of solutions to \eqref{Eqn symmetric and same lengths case and any a}, for computing c-coefs. in the proof of part (III):			\begin{lemma}\label{Lemma (2,2) sol. characterization}
Fix $A=A(b,a,a,b)$ $a,b\in \mathbb{N}$, and let $M_1\geq M_2$ (mutatis mutandis is the below result if $M_2\geq M_1$).  Assume $(2,2)$ to satisfy \eqref{Eqn symmetric and same lengths case and any a}; so $\frac{2a}{b}=\frac{M_1+M_2}{2}-2$ by Lemma \ref{Lemma (k,k) solutions, consequences}.
				Then:
            
				\begin{center}
					\begin{figure}[H]
						\begin{subfigure}[t]{0.22\textwidth}
							\begin{tikzpicture}[scale=0.5] 
							\draw[step=1, gray, ultra thin] (-0.2,-0.2) grid (4.5,4.5) ;
							\draw [fill=black] (0, 0) circle (5pt) node[anchor=north]{{\color{black}{\tiny{ $(0,0)$ }}}};
							\draw [fill=black] (2, 2) circle (5pt) node[anchor=west]{{\color{black}{\tiny{ $(2,2)$ }}}};
							\draw (1.9, -0.3) node[anchor=west]{{\color{black}{ $\boldsymbol{\cdots}$ }}};
							\draw [fill=black] (4, 0) circle (5pt) node[anchor=north]{{\color{black}{ \tiny{\bf$(M_1,  0)$ }}}};
							\draw (-0.3, 4) node[anchor=north]{{\color{black}{ $\boldsymbol{\vdots}$ }}};
							\draw [fill=black] (0, 4) circle (5pt) node[anchor=south]{{\color{black}{\ \ \ \ \tiny{\bf$(0, M_2)$ }}}};
							\end{tikzpicture}
							\caption{$M_1,M_2\geq 2$}
						\end{subfigure} ~ 
						\begin{subfigure}[t]{0.22\textwidth}
							\begin{tikzpicture}[scale=0.5] 
							\draw[step=1, gray, ultra thin] (-0.2,-0.2) grid (4.3,4.3) ;
							\draw [fill=black] (0, 0) circle (5pt) node[anchor=north]{{\color{black}{\tiny{ $(0,0)$ }}}};
							\draw [fill=black] (2, 2) circle (5pt) node[anchor=south]{{\color{black}{\tiny{ $(2,2)$ }}}};
							\draw (2.6, -0.3) node[anchor=west]{{\color{black}{ $\boldsymbol{\cdots}$ }}};
							\draw (1.3, -0.3) node[anchor=west]{{\color{black}{ $\boldsymbol{\cdots}$ }}};
							\draw (1.3, 0.9) node[anchor=west]{{\color{black}{ $\boldsymbol{\cdots}$ }}};					\draw [fill=black] (4, 0) circle (5pt) node[anchor=south]{{\color{black}{ \tiny{\ \\ \bf$(M_1,  0)$ }}}};
							\draw (-0.3, 3.4) node[anchor=north]{{\color{black}{ $\boldsymbol{\vdots}$ }}};
                            \draw (1.1, 3.4) node[anchor=north]{{\color{black}{ $\boldsymbol{\vdots}$ }}};
							\draw (-0.3, 4.6) node[anchor=north]{{\color{black}{ $\boldsymbol{\vdots}$ }}};
							\draw [fill=black] (0,4) circle (5pt) node[anchor=south]{{\color{black}{\ \ \ \ \tiny{\bf$(0,M_2)$ }}}};
							\draw [fill=black] (1,3) circle (5pt) node[anchor=south]{{\color{black}{ \tiny{\bf$(1,k_2)$ }}}};
							\draw [fill=black] (3,1) circle (5pt) node[anchor=south]{{\color{black}{\tiny{\bf$(k_1,1)$ }}}};
							\end{tikzpicture}
							\caption{$2<k_i<M_i$\\  
                            \hspace*{1cm}$k_2\leq k_1$}
						\end{subfigure}~
						\begin{subfigure}[t]{0.22\textwidth}
							\begin{tikzpicture}[scale=0.5] 
							\draw[step=1, gray, ultra thin] (-0.2,-0.2) grid (4.2,3.3) ;
							\draw [fill=black] (0, 0) circle (5pt) node[anchor=north]{{\color{black}{\tiny{ $(0,0)$ }}}};
							\draw [fill=black] (2, 2) circle (5pt) node[anchor=south]{{\color{black}{\ \ \ \ \  \ \ \   \tiny{ $(2,2)$ }}}};
								\draw [fill=black] (4, 0) circle (5pt) node[anchor=north]{{\color{black}{ \tiny{\bf$(M_1,  0)$ }}}};
			\draw [fill=black] (0, 1) circle (5pt) node[anchor=south]{{\color{black}{\tiny{\bf$(0, 1)$ }}}};
            \draw [fill=black] (1.5, 3) node[anchor=south]{{\color{black}{ \tiny{\bf$s_1$}}}};
							\draw [fill=black] (1, 2) circle (5pt) node[anchor=south]{{\color{black}{\tiny{\bf$(1, 2)$ }\ \ \ \ \ \  \  }}};
							\path[every node/.style={font=\sffamily\small}]
							(2,2) edge[bend left = 90] node [right] {} (2,2)
							(2,2) edge[bend right = 90] node [left] {} (1.4,3)
							(1.6,3) edge[bend right = 90] node [left] {} (1,2);
                            \draw (1.8, -0.2) node[anchor=west]{{\color{black}{ $\boldsymbol{\cdots}$ }}} ;
						\end{tikzpicture}
							\caption{ $4\leq M_1$ even \ \  \ \\ \hspace*{0.5cm} \big( so $\frac{2a}{b}\notin \mathbb{N}$ \big) \  \ \ }
						\end{subfigure}~    
\begin{subfigure}[t]{0.22\textwidth}
\begin{tikzpicture}[scale=0.5] 
							\draw[step=1, gray, ultra thin] (-0.2,-0.2) grid (5.2,3.3) ;
							\draw [fill=black] (0, 0) circle (5pt) node[anchor=north]{{\color{black}{\tiny{ $(0,0)$ }}}};
							\draw [fill=black] (2, 2) circle (5pt) node[anchor=south]{{\color{black}{\ \  \  \  \  \ \tiny{ $(2,2)$ }}}};
							\draw (3.3, -0.2) node[anchor=west]{{\color{black}{ $\boldsymbol{\cdots}$ }}};
							\draw (1.8, -0.2) node[anchor=west]{{\color{black}{ $\boldsymbol{\cdots}$ }}};
                            \draw (1.2, 1.2) node[anchor=west]{{\color{black}{ $\boldsymbol{\cdots}$ }}};
							\draw [fill=black] (5, 0) circle (5pt) node[anchor=north]{{\color{black}{ \tiny{\bf$(M_1,  0)$ }}}};
							\draw [fill=black] (4, 1) circle (5pt) node[anchor=west]{{\color{black}{ \tiny{\bf$\Big(\frac{M_1+3}{2}, 1 \Big)$ }}}};
							\draw [fill=black] (0, 1) circle (5pt) node[anchor=south]{{\color{black}{\tiny{\bf$(0, 1)$ }}}};
							\draw [fill=black] (1, 2) circle (5pt) node[anchor=south]{{\color{black}{\tiny{\bf$(1, 2)$ \ \ \ \ \ \ \  }}}};
							\path[every node/.style={font=\sffamily\small}]
							(2,2) edge[bend left = 90] node [right] {} (2,2)
							(2,2) edge[bend right = 90] node [left] {} (1.4,3)
							(1.6,3) edge[bend right = 90] node [left] {} (1,2);
	\draw [fill=black] (1.5, 3) node[anchor=south]{{\color{black}{ \tiny{\bf$s_1$}}}};
	\path[every node/.style={font=\sffamily\small}]
		(0,1) edge[bend right= 20] node [left] {} (4,1);
				\draw [fill=black] (2.5, 1.2) node[anchor=north]{{\color{black}{ \tiny{\bf$s_1$}}}};
							\end{tikzpicture}
							\caption{$5\leq M_1$ odd\\ \hspace*{0.5cm}$\big( \text{ so }\frac{2a}{b}\in \mathbb{N}\ \big)$ }
						\end{subfigure}              
					\end{figure}
				\end{center}
                are all the possible plots of  the solutions to \eqref{Eqn symmetric and same lengths case and any a} (not drawn to a scale).
                Conversely: i) The characterizations for Cases (C), (D) above is complete and explicit.
               ii) If $\sqrt{M_1}=\sqrt{M_2}=k\in \mathbb{N}$, the solution-plot is as in (B). 
               iii) For any $k\in \mathbb{N}$, when $M_2=3k$ and $M_1=3k+2$, we land in (A). 
			\end{lemma}
			\begin{proof}[{\bf \textnormal Proof.}]
				We fix $a, b , M_1\geq M_2\in \mathbb{N}$ with $\frac{2a}{b}=\frac{M_1+M_2}{2}-2>0$ (possibly half-integer), and fix without further mentioning one solution of \eqref{Eqn symmetric and same lengths case and any a} to be $(2,2)$.
				 \smallskip\\
				\underline{Step 1}. {\it $(1,2)$ is a solution to \eqref{Eqn symmetric and same lengths case and any a} iff $M_2=1$ \big(Cases (C) \& (D)\big)} :
				The reverse implication in this claim is easily seen by substituting $(1,2)$ into \eqref{Eqn symmetric and same lengths case and any a}.
				Conversely, suppose $(1,2 )$ satisfies 
             \eqref{Eqn symmetric and same lengths case and any a}.
				Equation \eqref{Eqn symmetric and same lengths case and any a} at $Y=2$ is $X^2 + (M_2-4)X+ (4-2M_2)=0$, and $X=1,2$ satisfying this equation forces  $M_2=1$. 
                Note, there are no solutions on the $Y$-axis other than $(0,0)$ and $(0,M_2=1)$.
                Suppose \eqref{Eqn symmetric and same lengths case and any a} has a solution with $X=1=M_2$, i.e. $Y^2+\Big(\frac{M_1-3}{2}-1\Big)Y+(1-M_1)=0$ for some $Y= y\in \mathbb{Z}_{\geq 0}\setminus \{2\}$.
                Then observe, $2y=1-M_1\geq 0$ forces $M_1=1$, which cannot happen since the coef. of $XY$ is positive.
                Finally consider $X=2$ line; i.e., solutions to $Y^2+(M_1-4)Y+(4-2M_1)=0$.
                As before, if this equation has a solution in $\mathbb{Z}_{\geq 0}\setminus \{2\}$, then $4-2M_1\geq 0$ contradicts $\frac{2a}{b}=\frac{M_1-3}{2}>0\implies M_1\geq 4$. 
                Now for any solution $(X,Y)$ to \eqref{Eqn symmetric and same lengths case and any a} with $X>2$, Lemma \ref{Lemma (k,k) solutions, consequences} forces $Y\leq 1$.
                Indeed the point $\big(\frac{M_1+3}{2}\in \frac{1}{2}\mathbb{Z}_{>0}\ ,\ 1\big)$ satisfies \eqref{Eqn symmetric and same lengths case and any a} and is the only solution on $Y=1$ line other than $(0,1)$.
                Observe $\frac{M_1+3}{2}\in \mathbb{Z}_{>0}$ iff $5\leq M_1$ is odd, showing Cases (D) as well as (C).
				\smallskip\\
				\underline{Step 2}. {\it For $M_2>1$, suppose $(1, k_2)$ is a solution for some $k_2>2$} \big(both in view of Step~1\big):
				At $X=1$, \eqref{Eqn symmetric and same lengths case and any a} is $Y^2+\left(\frac{M_1-M_2}{2} -2 \right)Y + (1-M_1)=0$.
				So $k_2$ has two possible values $k_2^{(\pm)}\ :=\ \frac{1}{2}\Big[ \Big(2-\frac{M_1-M_2}{2}\Big)\pm \sqrt{\Big(\frac{M_1-M_2}{2}\Big)^2+2M_1+2M_2} \Big]$.
				If both $k_2^{(\pm)}\in \mathbb{Z}_{\geq 0}$, then $k_2^{(+)}k_2^{(-)}
                \geq 0\implies (1-M_1)\geq 0$ forces $M_1=1$, contradicting $M_1\geq M_2>1$.
				So, $\big(1,\ k_2=k_2^{(+)}\big)$ is the only solution to \eqref{Eqn symmetric and same lengths case and any a} with its components in $\mathbb{N}$, but not $\big(1,k_2^{(-)}\big)$.
				We write $k_1^{(+)}$ swapping $M_1,M_2$ in the definition of $k_2^{(+)}$.  
				It can be checked that $X=k_1^{(+)}$ \big(which is $>2$ by Step 1\big)  satisfies $X^2+\left(\frac{M_2-M_1}{2} -2 \right)X + (1-M_2)=0$, and thereby  $\big(k_1=k_1^{(+)}, \ 1\big)$ satisfies \eqref{Eqn symmetric and same lengths case and any a}. 
                The uniqueness of positive $(k_1,1)$ is seen similar to that of $(1,k_2)$.
            We are in (B) when $M_2>1$ and some $(1, k_2)$ \big(equivalently when $(k_1, 1)$\big) satisfies \eqref{Eqn symmetric and same lengths case and any a}. 
                    
In Case (B) it remains to compare $k_i$ and $M_i$ for each $i$.
We first observe that $M_2> 2$ for the existence of solution $(1,k_2>2)$: $M_2=2\implies k_2= \frac{1}{2}\bigg[ 3-\frac{M_1}{2}  \pm \sqrt{ \Big(\frac{M_1}{2}+1\Big)^2 + 4}\bigg]\implies \Big(2k_2+\frac{M_1}{2}-3\Big)^2= \Big(\frac{M_1}{2}+1\Big)^2+4\implies (k_2-2)\Big(k_2-2 +\frac{M_1}{2}+1\Big)=1$, which is not possible as $k_2>2$ and $M_1>0$.
Suppose $k_2\geq M_2$; so $\dim L(\lambda)_{\lambda-k_2\alpha
                _2}=0$.
                Then $\frac{1}{2} \bigg[(2-\frac{M_1-M_2}{2}\Big)\pm \sqrt{\Big(\frac{M_1-M_2}{2}\Big)^2+2M_1+2M_2} \bigg]\geq M_2$ implies $\Big(\frac{M_1-M_2}{2}\Big)^2+(2)M_1+(2)M_2\ \geq \  \Big(\frac{M_1+(M_2-4)}{2}\Big)^2+(M_2)M_1 + (2M_2-4)M_2$, which on term-wise comparison leads to contradiction $M_2\leq 2$.
Note now $\frac{2a}{b}= \frac{M_1+M_2}{2}-2\geq 3-2\geq 1$.
So by Lemma~\ref{Lemma bddness of sols for two negative nodes}, $k_1\leq M_1=\max(M_1,M_2)$.
Suppose $k_1=M_1$; recall $2\leq M_2\leq M_1\in \mathbb{N}$.
                So $Y- M_2+ \Big(\frac{M_1+M_2}{2}-2\Big)M_1=0$ \big(which is \eqref{Eqn symmetric and same lengths case and any a} at $X=M_1$\big) satisfied by $Y=1$ forces $1>\frac{M_2-1}{M_1}=\frac{M_1+M_2}{2}-2=\frac{1}{2}$, forcing $M_2=2$ and $M_1=3$ and thereby $\frac{1}{3}=\frac{1}{2}\Rightarrow\!\Leftarrow$.  
                So $k_1<M_1$.  
                      \smallskip\\
                \underline{Step 3}.
 {\it Suppose there are no solutions $(X,Y)$ to \eqref{Eqn symmetric and same lengths case and any a} with $X=1$,  or with $Y=1$} : 
 By Step 1, $M_1\geq M_2\geq 2$, and by Lemma \ref{Lemma bddness of sols for two negative nodes} the solution-plot is as in Case (A). 
 If $M_1=M_2\geq 2$, then  $\sqrt{M_1}\notin \mathbb{N}$ as there are no solutions of the form $(1, k_2^{(+)})$.
When $M_2=2$, $\frac{2a}{b}>0\implies M_1>2$.

Finally we are left to check claims i)--iii) in the lemma.
In ii), we see By Step 2, $k_1^{(+)}=k_2^{(+)}=k+1$.
Observe in iii), the term inside the square root in the formulas for $k_{1 \text{ or }2}^{(+)}$ is $4M_2+5$, which being congruent to $2 \ \mathrm{mod}\ 3$ cannot be a perfect square, and hence we are in Case (A).
\end{proof}
The proof of Theorem \ref{Theorem C}(III) involves stepping down from $(2,n)$-space, and doing raising computations in lower $(1,k)$-spaces.
Helpful for the latter is proving Proposition \ref{Prop maxl vect}(b)--(c)
			\begin{proof}[{\bf \textnormal Proof of Proposition \ref{Prop maxl vect}(b).}]
            Similar to in the proof of part (a), we work over $\mathfrak{g}(b,a,a,d)$ for $a,..,d\in \mathbb{N}$.
	Fix $\lambda\in P^{\pm} $ and powers $M_1, M_2\in \mathbb{N}$, and assume $(1, n)$ to satisfy \eqref{Eqn symmetric and same lengths case and any a}; inside whose weight space in $M(\lambda)$ we construct a unique (up to scalars) maximal vector.
		Recall, $M(\lambda)_{\lambda-\alpha_1-n\alpha_2}$ has the below basis; where $\ad_{f_2}^{\circ i} (f_1)$  is the adjoint operator  $\ad_{f_2}$ composed $i$-times and acted on $f_1$:
				\[
				\left\{ 
				f_1f_2^n\cdot m_{\lambda}\ , \ \ \ad_{f_2}(f_1)\cdot f_2^{n-1}\cdot m_{\lambda}\ ,\  \ \ldots \ , \  \ \ad_{f_2}^{\circ i}(f_1)\cdot f_2^{n-i}\cdot m_{\lambda}\ ,\ \  \ldots \  , \ \   \ad_{f_2}^{\circ n}(f_1)\cdot m_{\lambda} \right\},
				\]
	by the PBW theorem and using any PBW ordering $\Big\{f_2; \ f_1;\  [f_2, f_1];\  \big[f_2, [f_2, f_1]\big]; \ldots; \ \ad_{f_2}^{\circ n}( f_1); \ldots \big\}$ for $\mathfrak{n}^-$. 
				Let $x=\sum_{i=0}^n c_i [\ad_{f_2}^{\circ i}(f_1)]f_2^{n-i}m_{\lambda}$, for some scalars $c_0,\ldots, c_n\in \mathbb{C}$.
				We solve the simultaneous equations $e_1x = e_2 x =0$ for the solution vectors $0\neq \big(c_0, \ldots , c_n)\in \mathbb{C}^{n+1}$ :
				\begin{align}\label{Eqns e1 action showing c1=0}
				\begin{aligned}[]
				c_0 f_1f_2^n m_{\lambda} \ \  \xrightarrow{e_1}\ \  c_0 \big(\lambda(\alpha_i^{\vee})-n A_{12}\big)f_2^n m_{\lambda}, \qquad 
				c_1 [f_2, f_1]f_2^{n-1} m_{\lambda}  \ \ \xrightarrow{e_1} \ \ c_1(A_{12})f_2^{n}m_{\lambda} \\
				c_i \ad_{f_2}^{\circ i}(f_1)f_2^{n-i} m_{\lambda}  \ \ \xrightarrow{e_1} \ \ c_i(A_{12})\underset{=0}{\cancel{\ad_{f_2}^{\circ i}(\alpha_1^{\vee})}} f_2^{n-i}m_{\lambda}\ = 0 \ \ \ \forall\ i\geq 2.  
				\end{aligned}
				\end{align}
				Thus for $0= e_1 x\ = \ \Big[c_1A_{12}+c_0 \big(\lambda(\alpha_i^{\vee})-n A_{12}\big)\Big]f_2^n m_{\lambda}$, we observe  
				\begin{equation}\label{Eqn c0 c1 relation for e1x=0 for (M1, 1) case}
				c_2,\ldots, c_n\  \text{can be arbitrary } \quad \text{ and }\quad     c_1A_{12}+c_0 \big(\lambda(\alpha_1^{\vee})-n A_{12}\big) =0 .
				\end{equation}              \begin{align*}\begin{aligned}[]
				c_{i}\ad_{f_2}^{\circ i}(f_1)\ f_2^{n-i}m_{\lambda}\ \ &\xrightarrow{e_2}\ \ c_i(-i)\Big(A_{21}+\frac{i-1}{2}A_{22}\Big)\ad_{f_2}^{\circ (i-1)}(f_1)\ f_2^{n-i}m_{\lambda}\\
				& \hspace{3.5cm}+ c_i(n-i)\Big(  \lambda(\alpha_2^{\vee})-\frac{n-i-1}{2}A_{22}\Big)\ad_{f_2}^{\circ i }(f_1)\ f_2^{n-i-1}m_{\lambda},
				\end{aligned}
				\end{align*}
                for all $0\leq i\leq n$. Above, at $i=0$ the first term vanishes, and at $i=n$ the second term vanishes. 
				The vanishing of the coefficients of the summands $\ad_{f_2}^{\circ (i-1)}(f_1)f_2^{n-i}m_{\lambda}$ in $e_2x=0$ force 
				\begin{equation}\label{Eqn c0 c1 relation for e2x=0 for (M1, 1) case}
				c_{i}(-i)\Big(A_{21}+\frac{i-1}{2}A_{22}\Big)+c_{i-1}(n-i+1)\Big(\lambda(\alpha_2^{\vee})-\frac{n-i}{2}A_{22}\Big)\ = \ 0\qquad \forall\ 1\leq i\leq n.
				\end{equation}
				Using $A_{12},A_{22}<0$, iteratively by equation(s) \eqref{Eqn c0 c1 relation for e2x=0 for (M1, 1) case}, we see $c_0=0$ to force $c_i=0$ $\forall\ i\geq 1$.
				So w.l.o.g. we assume $c_0=1$, and choose by \eqref{Eqn c0 c1 relation for e1x=0 for (M1, 1) case} $c_1=n-\frac{\lambda(\alpha_1^{\vee})}{A_{12}}= n-\frac{b}{2a}(M_1-1)$ for $e_1x=0$.
				Next, again as $A_{21}+\frac{i-1}{2}A_{22}<0$ $\forall\ i\geq 1$, observe iteratively by  \eqref{Eqn c0 c1 relation for e2x=0 for (M1, 1) case} for $i\geq 2$, that $c_2,\ldots, c_n$ can be uniquely found.
				 But \eqref{Eqn c0 c1 relation for e2x=0 for (M1, 1) case} for \big($i=1$, $c_0=1$, $M_2$ any fixed\big) says $-c_1(-a)-n \frac{d}{2}(M_2-n)=0$, implying $c_1=\frac{d}{2a}n (M_2-n)$.
		So it remains to see if the value $n-\frac{b}{2a}(M_1-1)$ of $c_1$ found above equals $\frac{d}{2a}n (M_2-n)$, equivalently if $2an-b M_1 +b=dnM_2-dn^2$?
                This is true as we assumed $(1,n)$ to satisfy \eqref{Eqn symmetric and same lengths case and any a}.
				Hence, the above choices of $c_i$s -- depending on the initial $c_0$ -- yield the sought-for (unique) non-zero maximal vector in $M(\lambda)_{\lambda-\alpha_1-n\alpha_2}$, completing the proof of part (b).
            \end{proof}
            The proofs of Proposition \ref{Prop maxl vect}(c) and Theorem \ref{Theorem C}(III), require working with the following bases for weight spaces $M(\lambda)_{\lambda-n\alpha_1-\alpha_2}$ $\forall$ $n\in \mathbb{N}$;  for $M(\lambda)_{\lambda-\alpha_1-2\alpha_2}$; and for  $M(\lambda)_{\lambda-2\alpha_1-2\alpha_2}$ respectively:\\
	(i) $f_1^n f_2 m_{\lambda}\ , \ [f_1, f_2]f_1^{n-1} m_{\lambda}\ , \ \ldots \ ,\  \ \Big[f_1,\big[  \ldots, \big[f_1, [f_1, f_2]\big]\ldots\big] \Big] m_{\lambda}$;
					\\ (ii)  $ 
					f_1f_2^2 m_{\lambda}\ , \ [f_2, f_1]f_2 m_{\lambda}\ , \ \big[f_2, [f_2, f_1]\big] m_{\lambda}$;  
					\\ (iii)   
					$f_1^2f_2^2m_{\lambda}$, $[f_1,f_2]f_1f_2 m_{\lambda}$, $\big[f_2,[f_1,f_2]\big]f_1 m_{\lambda}$,
					$\big[f_1,[f_1,f_2]\big]f_2 m_{\lambda}$, $\Big[f_1,  \big[f_2, [f_1,f_2]\big]\Big] m_{\lambda}$, $([f_1,f_2])^2 m_{\lambda}$.
            \begin{proof}[{\textnormal \bf Proof of Proposition \ref{Prop maxl vect}(c)}]
    Fix a vector $x\in M(\lambda)_{\lambda-2\alpha_1-2\alpha_2}$ with coefficients $c_1, \ldots, c_6$ w.r.t. the basis in (iii) above. 
    We see conditions \eqref{Eqn coefs cond for maxl vect in (2,2) space when (M1, 1)} below on $c_1,\ldots, c_6$, for the  maximality of $x$ via:\\
    $0 \ =\  e_2\cdot x\ = \   c_1\Big( \big(2\lambda(\alpha_2^{\vee})-(-b)\big)f_1^2f_2 m_{\lambda}\ = \ b(2-M_2)f_1^2f_2m_{\lambda}\Big)  + c_2\Big( -a f_1^2f_2 m_{\lambda} \ -\ \frac{b}{2}(M_2-1)   [f_1,f_2]f_1 m_{\lambda}\Big)  + c_3 \Big(  \big(b+2a\big)[f_1,f_2]f_1m_{\lambda}\Big) 
 + c_4\Big(\underset{=0}{\cancel{\big[f_1,[f_1,\alpha_2^{\vee}] \big]}} f_2 m_{\lambda}\ - \ \frac{b}{2}(M_2-1) \big[f_1, [f_1, f_2]\big] m_{\lambda}\Big)  +\\ c_5\Big( - \big(2a+b\big)\big[f_1,[f_1,f_2]\big] m_{\lambda}\Big) + c_6\Big( -af_1 [f_1, f_2]m_{\lambda}\ - \ a[f_1, f_2]f_1 m_{\lambda}\ = \ -a\big[ f_1, [f_1, f_2]\big]m_{\lambda}\ -\ 2a [f_1, f_2]f_1 m_{\lambda}\Big)$. \ And\\
 $0\ = \ e_1\cdot x\ = \ c_1\Big(\big(-b(M_1-1)+4a+b)\big)f_1 f_2^2 m_{\lambda}\ = \ (4a+2b-bM_1)f_1f_2^2 m_{\lambda}\Big)  \ +\  c_2 \Big( a f_2f_1f_2 m_{\lambda}\ - \ \left(a-\frac{b}{2}(M_1-1)\right)[f_2, f_1]f_2 m_{\lambda}\  =\  af_1f_2^2 m_{\lambda} + \frac{b}{2}(M_1-1)[f_2, f_1]f_2 m_{\lambda} \Big)\  +\  c_3\Big( \frac{b}{2}(M_1-1) \big[f_2, [f_2, f_1]\big]m_{\lambda}\Big) 
 \ + \ c_4\Big( - (a+b)[f_2, f_1]f_2m_{\lambda}\ + \  a[f_1,f_2]f_2m_{\lambda} \ = \ - (2a+b)[f_2, f_1]f_2 m_{\lambda} \Big) \ +\ c_5\Big(-(2a+b)\big[f_2, [f_2, f_1]\big]m_{\lambda}\Big) \ +\ c_6\Big(- af_2 [f_2, f_1]m_{\lambda}\ - \ a[f_2, f_1]f_2 m_{\lambda}\ = \ -a\big[ f_2, [f_2, f_1]\big]m_{\lambda}\ -\ 2a [f_2, f_1]f_2 m_{\lambda}\Big)$.
			\begin{align}\label{Eqn coefs cond for maxl vect in (2,2) space when (M1, 1)}
                \begin{aligned}[]
				\ \ &0 =   (4a+2b-bM_1)c_1  +a c_2   =     \frac{b}{2}(M_1-1)c_2-(2a+b)c_4-2ac_6   =    \frac{b}{2}(M_1-1)c_3-(2a+b)c_5  -ac_6  \\
                \ \ & =   b(2-M_2)c_1-ac_2  =  -\frac{b}{2}(M_2-1)c_2+(b+2a)c_3 -2ac_6     =     -\frac{b}{2}(M_2-1)c_4-(2a+b)c_5-ac_6  .  
				\end{aligned}\end{align}
                \begin{equation}\label{Eqn (2,2) norm-cond}
		\text{By }(2,2)\text{ satisfying }\eqref{Eqn symmetric and same lengths case and any a}\quad 	4 - M_1-M_2 + 4\frac{a}{b}\ = \ 0 \implies \frac{b}{2}(1-M_1)+ (2a+b) \ = \ -\frac{b}{2}(1-M_2).
				\end{equation}
				Adding the first and fourth terms in \eqref{Eqn coefs cond for maxl vect in (2,2) space when (M1, 1)} for $c_1,c_2$, yields $\big(4a+4b-bM_1-bM_2\big)c_1=0$, which holds true irrespective of the value of $c_1$ via equation \eqref{Eqn (2,2) norm-cond}.
				
				Suppose $c_1=0$.
				So $c_2=0$.
				Next by the second and fifth terms in \eqref{Eqn coefs cond for maxl vect in (2,2) space when (M1, 1)}, $c_3=-c_4$. 
				Finally, $c_6= \left(1+\frac{b}{2a}\right)c_3$, and $\frac{b}{2}(M_2-1)c_3-(2a+b)c_5-a\left(1+\frac{b}{2a}\right)c_3=0$ implies $c_5=\frac{-1}{2a+b}\left(a+b-\frac{b}{2}M_2\right)c_3$.
				So we got coefs. $(c_1, \ldots, c_6)\neq \overline{0}$ with $c_1=c_2=0$ and $c_3,c_5,c_6$ depending on $c_3$.
				Now assume w.l.o.g. $0\neq c_1=1$.
				In turn as above, fixing any value of $c_3$, we can find all $c_4, c_5, c_6$ as functions of $c_3$ uniquely.
				So, choices of $(c_1, c_3)\in \mathbb{C}^2$ determine the  desired 2-dim. space of maximal vectors.              
            \end{proof}
            \begin{proof}[{\bf \textnormal Proof of Theorem \ref{Theorem C}(III)}]
As in the proofs of parts (I) and (II), it suffices to describe $c$-coefs. here, using the denominator identity \eqref{Eqn denominator identity}. 
Clearly $c(\lambda-M_i\alpha
                _i)=-1$ $\forall$ $i$ in Cases (A)--(D).\smallskip\\
                \underline{Step 1}. {\it For Cases (A) and (B) }: 
              Here there are no  solutions $\mu$ to \eqref{Eqn symmetric and same lengths case and any a} in the poset $\lambda\preceq \mu \precneqq \lambda
                -2\alpha_1-2\alpha_2$.
                So no quotienting happens within $M(\lambda)_{\mu}$ during $M(\lambda)\twoheadrightarrow L(\lambda)$; since if $V_1$ is the quotient of $M(\lambda)$ obtained by killing all the maximal vectors in all the solution spaces (except for $\lambda$) to \eqref{Eqn norm equality with invariant form}, then there are no further maximal vectors in $\big(V_1\big)_{\mu}$ for $\mu$ within that poset.
               These observations proceed similar to \cite[Section 4]{Kac--Kazhdan}.
               So $\dim M(\lambda)_{\mu}=\dim L(\lambda)_{\mu}$ for all such $\mu$.
               Indeed, $\dim M(\lambda)_{\nu}=\dim L(\lambda)_{\nu}$ iff there does not exist a solution $\mu'$ to \eqref{Eqn norm equality with invariant form} with
 $\lambda\precneqq \mu'\precneqq \nu$.
            So by Proposition \ref{Prop maxl vect}(c), we have 
                $c(\lambda-2\alpha
                _1-2\alpha
                _2)= \Big[\underbrace{\dim M(\lambda)_{\lambda
                -2\alpha_1-2\alpha_2}}_{=6}-2\Big]- \dim M(\lambda)_{\lambda-\alpha_1-2\alpha_2}- \dim M(\lambda)_{\lambda-2\alpha_1-\alpha_2} = -2$, showing (A).
           Similarly by Proposition \ref{Prop maxl vect}(b), $c(\lambda-\alpha_1-k_2\alpha_2)=\dim L(\lambda)_{\lambda-\alpha_1-k_2\alpha_2}-\dim M(\lambda)_{\lambda-\alpha_1-(k_2-1)\alpha_2}-\dim M(\lambda)_{\lambda-k_2\alpha_2}= k_2-k_2-1=-1\ =c(\lambda-k_1\alpha_1-\alpha_2)$, showing Case (B).\smallskip\\
            \underline{Step 2}. The rest of the proof is  commonly written for Cases (C) and (D).
            Fix a h.w. vector $0\neq v_{\lambda}\in L(\lambda)_{\lambda}$. Observe $f_2 v_{\lambda}=0$, $L(\lambda)_{\lambda-\alpha_1-2\alpha_2}=\mathbb{C}\big[f_2,[f_2,f_1]\big]v_{\lambda}$, $L(\lambda)_{\lambda-\alpha_1-\alpha_2}=\mathbb{C}[f_1,f_2] v_{\lambda}$.
       So $ c(\lambda-\alpha_1-2\alpha_2)= 1-0-1=0$.
       This shows the phenomenon quoted in Observation \ref{Observation missing solution points in character numerator}, and below we show it for the two other possible solutions $\big((M_1+3)/2,\ 1\big)$ and $(2,2)$ as well.

We fix $k\in \mathbb{N}$ and compute $\dim L(\lambda)_{\lambda-k\alpha_1-\alpha_2}$ \big(for Case (D)\big), which is required for writing the next two $c$-coefs.
Below we will make use of the PBW monomial basis (omitting $f_1^kf_2m_{\lambda}$) in the proof of Proposition \ref{Prop maxl vect}(b) with the order of (nodes 1 and 2, i.e.) $f_1$ and $f_2$ interchanged, and use the computations for $e_i$-actions there-in.
Note by comparing weights and using this PBW basis, that $\big(U(\mathfrak{n}^-)f_2m_{\lambda}\big)\cap M(\lambda)_{\lambda-t\alpha_1-\alpha_2} = \mathbb{C}\big\{ f_1^tf_2m_{\lambda} \big\}$ $\forall$ $t\in \mathbb{Z}_{\geq 0}$. 
Let us assume in h.w.m. $V':=\frac{M(\lambda)}{\big\langle f_2 m_{\lambda}\ , \ f_1^{M_1}m_{\lambda}\big\rangle} = U(\mathfrak{n}^-)v'_{\lambda}$ (say), that there is a $\mathfrak{g}$-submodule 
non-trivially intersecting $V'_{\lambda-k\alpha_1-\alpha_2}$ for some $1\leq k<M_1$; so $\dim L(\lambda)_{\lambda-k\alpha_1-\alpha_2}< \dim V'_{\lambda-k\alpha_1-\alpha_2} \leq  \dim M(\lambda)_{\lambda-k\alpha_1-\alpha_2}$.
Then there exists an $x=\sum_{t=1}^n c_t\ad_{f_1}^{\circ t}(f_2)\ f_1^{n-t}m_{\lambda}\in M(\lambda)_{\lambda-n\alpha_1-\alpha_2}$ for some $n\leq k$ and $c_t\in \mathbb{C}$, s.t. its projection $x'=\sum_{t=1}^n c_t\ad_{f_1}^{\circ t}(f_2)\ f_1^{n-t}v'_{\lambda}\neq 0$ is maximal in
 $V'_{\lambda-n\alpha_1-\alpha_2}$.
So $e_2 x'=0\in V'$ $\implies c_1=0$ \big(as in \eqref{Eqns e1 action showing c1=0}\big) since $f_1^n m_{\lambda}\notin U(\mathfrak{n}^-)\big\{f_2m_{\lambda},\ f_1^{M_1}m_{\lambda}\big\}$.
Importantly for $e_1 x'= 0\ \mod U(\mathfrak{n}^-)f_2m_{\lambda}$ -- note $\big(U(\mathfrak{n}^-)f_2 m_{\lambda}\big)_{\lambda-k\alpha_1-\alpha_2}\neq 0$, and so we do the below  computations in $M(\lambda)$, for which recalling $n\leq k < M_1$ -- we write  $e_1 x=  c_0 f_1^{n-1} f_2 m_{\lambda}$ for some $c_0\in \mathbb{C}$.
As in the $e_2$-action \big(showing $c_1=0$, and as in Proposition \ref{Prop maxl vect}(b)'s proof\big), observe \   $\ad_{f_1}^{\circ i}(f_2)\ f_1^{n-i}m_{\lambda}\ \ \xrightarrow{e_1}$
\begin{center}
 $   \ i\Big(a+\frac{i-1}{2}b\Big)\ad_{f_1}^{\circ (i-1)}(f_2)\ f_1^{n-i}m_{\lambda} \ - \  (n-i)\frac{b}{2}\big(  M_1-(n-i)\big)\ad_{f_1}^{\circ i }(f_2)\ f_1^{n-i-1}m_{\lambda}\ \ \forall\ i\geq 2.$
\end{center}
So in the resulting expansion of  $e_1x$, all the summands involve iterated brackets $\big[f_1,\ldots,[f_1, f_2]\ldots\big]$.
Finally, linear independence of the monomial basis $\left\{ f_1^{n-1}f_2m_{\lambda},\ \ad_{f_1}^{\circ i}(f_2)f_1^{n-i-1}m_{\lambda}\ \big|\ n-1\geq i\geq 1  \right\}$ (and the positivity/non-vanishing of all the first coefs. in the expressions for $e_1$-actions above for each $i$) force $c_0=c_2=\cdots =c_n=0$.
So, there is no non-trivial submodule of $V'$ intersecting $V'_{\lambda-k\alpha_1-\alpha_2}$ for any $k\geq 1$.
Thus, even after the passage $V'\twoheadrightarrow L(\lambda)$, we have 
\begin{equation}\label{Eqn dim L(lambda) for (1,k)-wt spaces in cases (C) and (D)}
\dim L(\lambda)_{\lambda-k\alpha_1-\alpha_2}\ = \  \dim V'_{\lambda-k\alpha_1-\alpha_2}\ = \ \dim M(\lambda)_{\lambda-k\alpha_1-\alpha_2}-1\ =\  k \quad \forall\ k\in \mathbb{Z}_{\geq 0}.
\end{equation} 
Hence when $5\leq M_1$ is odd, we have $c\Big(\lambda-\frac{M_1+3}{2}\alpha_1-\alpha_2\Big)\ =\ \frac{M_1+3}{2}-\frac{M_1+1}{2}-1=0$.

\noindent 
It remains to compute $c(\lambda-2\alpha_1-2\alpha_2)$.
One can see $L(\lambda)_{\lambda-2\alpha_1-2\alpha_2}$ = $\mathbb{C}\big\{\big[f_2,[f_1,f_2]\big]f_1 v_{\lambda},\ [f_1,f_2]^2 v_{\lambda},\\ \Big[f_2,\ \big[f_1, [f_1,f_2]\big]\Big]v_{\lambda}\big\}$, by:  1) $f_2v_{\lambda}=0$. 2) And for the linear independence of those 3 vectors, see-
	\begin{align*}
				\begin{aligned}[]
	        \big[f_2, [f_1,f_2]\big]f_1 v_{\lambda}\ \ \xrightarrow{e_2}\ \ \big(b+2a\big)[f_1,f_2]f_1 v_{\lambda}, \quad   
         				\Big[f_2,\   \big[f_1, [f_1, f_2]\big]\Big] v_{\lambda}\ \ \xrightarrow{e_2}\ & \  \big(b+2a\big)\big[f_1,[f_1,f_2]\big]v_{\lambda},\\	[f_1,f_2][f_1, f_2]v_{\lambda}\ \ \xrightarrow{e_2}\ \ -af_1 [f_1, f_2]v_{\lambda}\  - \ a[f_1, f_2]f_1 v_{\lambda}\ = \ - a\big[ f_1,\  [f_1 & , f_2]\big]v_{\lambda}\ - \ 2a [f_1, f_2]f_1 v_{\lambda}.
				\end{aligned}
				\end{align*}
3) Using $[f_1,f_2]f_1v_{\lambda
}$ and $\big[f_1,[f_1,f_2]\big]v_{\lambda}$ to be linearly independent in $L(\lambda)_{\lambda-2\alpha_1-\alpha_2}$ by \eqref{Eqn dim L(lambda) for (1,k)-wt spaces in cases (C) and (D)}.

So by \eqref{Eqn dim L(lambda) for (1,k)-wt spaces in cases (C) and (D)},  $c(\lambda-2\alpha_1-2\alpha_2)= \dim L(\lambda)_{\lambda-2\alpha_1-2\alpha_2}- \dim L(\lambda)_{\lambda-\alpha_1-2\alpha_2}-\dim L(\lambda)_{\lambda-2\alpha_1-\alpha_2}=3-1-2=0$.
Finally, the weight-multiplicities in quotients $V_1$ and $V'$ dealt with in Steps 1--2, equaling those in $L(\lambda)$, which is explained in the proof, presentations of $L(\lambda)$ in Corollary \ref{Corollary presentations gfor nice simples} are seen to be true as follows: 
1) In Cases (A) and (B), all the maximal vectors in the weight spaces of all the solutions to \eqref{Eqn symmetric and same lengths case and any a} in the plots in Lemma \ref{Lemma (2,2) sol. characterization}, minimally generate the maximal submodule of $M(\lambda)$.
2) In Cases (C) and (D), the minimal generators are $f_1^{M_1}m_{\lambda}, \  f_2m_{\lambda}$.
 \end{proof}	
\section{Theorem \ref{Theorem D character of V(rho)} : Characters of $V$ for top weight the Weyl vector}\label{Section 6 for rho}
We show Theorem \ref{Theorem D character of V(rho)} for $A=A(n)$ \big(rank-$n$ type-$A$ Cartan matrix with all diagonal entries $-2$ instead of $2$s\big).
Note $A(n)$ is non-degenerate.
First, the versions of norm equality \eqref{Eqn norm equality with invariant form} here:
     \begin{lemma}\label{Lemma negative type A norm equality equaion}
          In the notations of Theorem \ref{Theorem D character of V(rho)}, fix any $\mu=\rho-X_1\alpha_1-\cdots-X_n\alpha_n$ satisfying  \eqref{Eqn norm equality with invariant form}. Then $(X_1,\ldots, X_n)$ satisfy \eqref{Eqn norm equality for rho 1}, equivalently \eqref{Eqn norm equality for rho 2}.     
          Moreover, such $\mu$ are finitely many.
       \end{lemma}
       \begin{proof}[{\bf \textnormal Proof}]
      The ``norm'' of simples $\alpha_1,...,\alpha_n$ is $-2$.
   Note, $4(\rho\ ,\  X_1\alpha_1+\cdots +X_n\alpha_n) - (X_1\alpha_1+\cdots +X_n\alpha_n \ , \ X_1\alpha_1+\cdots +X_n\alpha_n)= \ 0 \ = \sum\limits_{i=1}^n 4\frac{(\alpha_1,\alpha_1)}{A_{11}}\Big(\rho\ , \ \frac{A_{ii}}{(\alpha_i,\alpha_i)}\alpha_i\Big) X_i -\sum\limits_{i=1}^{n-1}\frac{2(\alpha_1,\alpha_1)}{A_{11}}\Big(\alpha_{i+1},\frac{A_{ii}}{(\alpha_i,\alpha_i)}\alpha_i\Big)X_iX_{i+1}  \\ -   (\alpha_1,\alpha_1)\big(X_1^2+\cdots +X_n^2\big)$, which imply \eqref{Eqn norm equality for rho 1}.
By writing the sum on the l.h.s. of 2-times \eqref{Eqn norm equality for rho 1} as $X_1^2 + X_n^2+ (X_1+X_2)^2+\cdots (X_{n-1}+X_n)^2 - 4(X_1+\cdots +X_n)= X_1^2-2X_1 +1 + X_n^2-2X_n+1+ \big((X_1+X_2)^2- 2(X_1+X_2)+1\big)+\cdots \big((X_{n-1}+X_n)^2- 2(X_{n-1}+X_n)+1\big)- (n+1)$, we see \eqref{Eqn norm equality for rho 2}.
The finiteness of $(X_1,\ldots, X_n)$ satisfying \eqref{Eqn norm equality for rho 2} follows by $X_i\leq \sqrt{n+1}+1$ $\forall i$ . 
       \end{proof}
       \begin{observation}\label{Observation on d-functions}
 Hereafter, we work with the restriction of quadratic form \eqref{Eqn norm equality for rho 2}:  $d^{(n)}:\mathbb{Z}_{\geq 0}^{\times n}\rightarrow \mathbb{Z}$, defined as follows \big(and use its crucial properties (1)--(2) further below, setting $d^{(0)}(.):=0$\big) -\begin{equation}\label{Eqn decomposition of dn functions}
        d^{(n)}(X_1,\ldots,X_n)\ :=\ (X_1-1)^2 + (X_1+X_2-1)^2+\cdots + (X_{n-1}+X_n-1)^2 + (X_n-1)^2-(n+1).
        \end{equation}  
(1) $d^{(n)}(X_1,\ldots, X_{i-1}, X_i=0, X_{i+1}, \ldots, X_n)=d^{(i-1)}(X_1,$ $\ldots,X_{i-1})+d^{(n-i)}(X_{i+1}, \ldots, X_n)$.\\
(2) $d^{(j)}(1,\ldots,1)=-2$ for all $j$.
We now note the easy, important solutions of $d^{(n)}(X_1,\ldots,X_n)=~0$.        
         \end{observation}
       \begin{observation}\label{Example large number of non-hole solutions for rho}
     (1) \eqref{Eqn norm equality for rho 1} is satisfied by all the tuples in \eqref{Eqn solutions of os and 2s for rho in type A}; e.g. $(0,2,0,0,2,0,0,0,2)$ for $n=9$.
       Their weight spaces in $M(\rho)$ are all 1-dim. and consist of  maximal vectors.
       (2) When $n=3$, one sees only 4 solutions other than $(0,0,0)$ --  for e.g. by the bounds in Lemma \ref{Lemma negative type A norm equality equaion}'s proof --  and moreover all are of the form \eqref{Eqn solutions of os and 2s for rho in type A}.
(3) We see many solutions in addition to \eqref{Eqn solutions of os and 2s for rho in type A}, even for (small) $n\geq 4$-
       For e.g. 33 solutions when $n=5$ (by a Python program); wherein $(0,1,0,1,2)$ and $(2,1,1,0,1)$ are important in Observation \ref{Observation solutions with no maximal vectors in M(rho)}.
       And 93 solutions when $n=6$. 
    Also, suitable permutations of components in each solution-tuple yield more solutions.
   (4) The solution-components grow large for large $n$, e.g. $(X_1,\ldots,X_n)=\big(m,\boxed{0,1},\ldots,\boxed{0,1},0,\ldots,0\big)$ with $m^2-2m$ many $\boxed{0,1}$ and $(2m^2-4m+1)\leq n$.
 A common important property of all the solutions:
 \end{observation}
    \begin{lemma}\label{Lemma at least 2 blocks in d-solutions}
        Fix a solution $(X_1,\ldots, X_n)$ to \eqref{Eqn norm equality for rho 1} or \eqref{Eqn norm equality for rho 2}.
        We call a subsequence $(X_{i}, \ldots, X_{i+k})$ for interval $[i, i+k]\subseteq \{1,\ldots, n\}$ as a block of $(X_1,\ldots, X_n)$, if: i)~$X_t>0$ $\forall$ $t\in [i, i+k]$ and ii)~$X_{i-1}=X_{i+k+1}=0$ when $i-1\geq 1$ and resp. $i+k+1\leq n$.
        Then any solution to \eqref{Eqn norm equality for rho 2} other than ``hole-solutions'' \eqref{Eqn solutions of os and 2s for rho in type A}, must have at least two blocks- all ones $(1,\ldots,1)$ block and another block with at least one entry $>1$.
        In particular, when $n>1$, each solution has at least one 0 entry. 
    \end{lemma}
    \begin{proof}[{\bf \textnormal Proof}]
    Beginning by a warm-up, note that $(2,1,1,0,1)$ has blocks $(2,1,1)$ and $(1)$, and $(0,1,0,1,2)$ has $(1)$ and $(1,2)$. 
Let $(X_1,\ldots, X_n)$ be as in the lemma that is outside \eqref{Eqn solutions of os and 2s for rho in type A}.
Note \eqref{Eqn norm equality for rho 2} has $n+1$ many square terms (in $X_i$s) and the constant term $-(n+1)$.
We assume $n>1$, since $n=1$ has only two solutions $(X_1)=(0)$ and $(2)$. 
Suppose $X_i>0$ $\forall$ $i\in \{1,\ldots, n\}$ i.e. $(X_1,\ldots, X_n)$ is a block; we prove the last claim in the lemma by the method of contradiction.
So $(X_i+X_{i+1}-1)^2>0$ $\forall\ 1\leq i\leq n-1$, and possibly $(X_1-1)^2$ and $(X_n-1)^2$ are 0.
Next $(X_1,\ldots,X_n)\neq (1,\ldots,1)$ by Observation \ref{Observation on d-functions}(2).
So we fix the least stage $i$ for which $X_t=1$ $\forall$ $t\leq i-1$ and $X_i\geq 2$. 
We observe
\begin{equation}\label{Eqn d-value lower bounds}
(X_1+X_2-1)^2+\cdot\cdot+ (X_{i-1}+X_i-1)^2\geq i+2\  \ \&\   \ (X_i+X_{i+1}-1)^2+\cdot\cdot +(X_{n-1}+X_n-1)^2\geq  (n-i).
\end{equation}
So we see the contradiction $0=d^{(n)}(X_1,\ldots, X_n)\geq i+2 + n-i - (n+1)=1$.
Hence $(X_1,\ldots,X_n)$ must have at least 2 blocks.
In view of Observation \ref{Observation on d-functions} (1)--(2), all the blocks in $(X_1,\ldots,X_n)$ cannot be $(1,\ldots,1)$.
Similarly by bounds \eqref{Eqn d-value lower bounds} applicable for non-$(1,\ldots,1)$ blocks above, not all the blocks can be different from $(1,\ldots,1)$ as well.
Hence the proof of the lemma is complete.  
\end{proof}
Following Observation  \ref{Example large number of non-hole solutions for rho}(2)--(4), we explore in the small case $n=4$, if the space of $(2,1,0,1)\\  \leftarrow\!\rightarrow \eta=\rho-2\alpha_1-\alpha_2-\alpha_4$ solution in $M(\rho)$ has maximal vectors?
The answer is no:
    \begin{observation}\label{Observation solutions with no maximal vectors in M(rho)}
Fix any $\big[\eta=(2,1,0,1)\big]$-weighted vector $x=f_4Fm_{\rho}=Ff_4m_{\rho}$ in $M(\rho)$ for $0\neq F\in U(\mathfrak{n}^-)_{-2\alpha_1-\alpha_2}$ (as $f_4$ commutes with $f_1,f_2$), and fix any quotient $V=U(\mathfrak{n}^-)v_{\rho}\twoheadleftarrow M(\rho)$.
(1)~If the projection $0\neq Ff_4v_{\rho}\in V_{\rho-2\alpha_1-\alpha_2-\alpha_4}$ of $x$ in $V$ is maximal, then for $i\in \{1,2\}$ observe $0=e_iFf_4v_{\rho}= f_4(e_iFv_{\rho})$ and $e_4 f_4(e_iFv_{\rho})= -e_iFv_{\rho}$ together force the maximality of $Fv_{\rho}\in V_{\rho-2\alpha_1-\alpha_1}$, and thereby the contradiction that $(2,1,0,0)$ satisfies \eqref{Eqn norm equality for rho 1} at $n=4$.
(2) For $V=M(\rho)$ and $x$ maximal, we explain this failure in more basic viewpoints.
By the $U(\mathfrak{n}^-)$-freeness of $M(\rho)$ we have $0=e_iFf_4m_{\rho}= f_4(e_iF m_{\rho})=  (f_4e_iF)m_{\rho}$ $\forall$ $i\in \{1,2\}$ forces $\mathfrak{n}^+Fm_{\rho}=0$, thereby the same contradiction as in the previous way.
(3) Alternately and more easily, observe following \cite[Theorem 2]{Kac--Kazhdan} that $\eta$ cannot be expressed as a sum $\rho-n_1\beta_1-\cdots -n_k\beta_k$ with properties in \eqref{Eqn dot action by any positive root}: As otherwise, $\alpha_4$ must necessarily equal some $\beta_t$ in that sum, but then $n_t=2> \height_{\{4\}}(\rho-\eta)=1$. 
\end{observation}
 Proposition \ref{Proposition solutions with maximal vector for L(rho)} shows the results in this observation, in general case of blocks of types  $(1,\ldots,1)$ and others.
For raising-up along $(1,\ldots,1)$-block-directions, we rely on the non-degeneracy of (the matrix of) the restrictions $F_{\gamma}$ of the Shapovalov form $F:U\big(\mathfrak{g}(A)\big)\rightarrow U(\mathfrak{h})$ to each homogeneous space in $U(\mathfrak{n}^-)=\bigoplus_{\xi\in \mathbb{Z}_{\geq 0 }\Pi }U(\mathfrak{n}^-)_{-\xi}$.
 \cite[Theorem 1]{Kac--Kazhdan} determines over  Contragredient $\mathfrak{g}(A)$s, the determinant of matrix $[f_{\gamma}]_{\mathcal{P}(\gamma)\times \mathcal{P}(\gamma)}$ for the generalized Kostant partition function (involves root multiplicities) $\mathcal{P}(.)$: 
 \big(See \cite[Subsection 4.1]{Naito 2}, and fix $(.,.)$ preserving identification $\nu:\mathfrak{h}\rightarrow \mathfrak{h}^*$.\big)
\begin{center}
$\text{det}\big([f_{\gamma}]\big)\ = \  \prod_{\beta\in \Delta^+} \prod_{r\in \mathbb{N}}\big[2 \nu^{-1}(\beta) + (\rho, \beta) -r(\beta,\beta)  \big]^{\mathcal{P}(\gamma-r\beta)} $.    \end{center} 
\begin{observation}\label{Observation dual basis by Shapovalov form}
If there is no root $\beta\in \Delta^+$ with $\gamma-\beta\succeq 0$ and with $2[\lambda+\rho](\beta)\neq  j (\beta, \beta)$  $\forall\ j\in \mathbb{N}$, then $[F]_{\gamma}$ is non-degenerate and thereby $M(\lambda)_{\lambda-\gamma} \twoheadrightarrow V_{\lambda-\gamma}\twoheadrightarrow L(\lambda)_{\lambda-\gamma}$ all have the same dimensions for every module $M(\lambda)\twoheadrightarrow V$.
In such a case, given a basis $F_1v_{\lambda},\ldots, F_mv_{\lambda}$ of $V_{\lambda-\gamma}$ for $F_i\in U(\mathfrak{n}^-)_{-\gamma}$, we have its dual basis $E_1v_{\lambda},\ldots, E_m v_{\lambda}$ in $V_{\lambda-\gamma}$ with $F_{\gamma}(F_iv_{\lambda}, E_jv_{\lambda})=\delta_{i,j}$ $\forall$ $1\leq i,j\leq m$.
Equivalently, for the involute anti-automorphism $\sigma$ of $U\big(\mathfrak{g}(A)\big)$ \cite[Section 4]{Naito 2} -- recall, $e_t\xleftrightarrow{\sigma}f_t$ and $h\xrightarrow{\sigma} h\ \forall\ h\in \mathfrak{h}$ -- we have $\sigma(E_i)F_j v_{\lambda}\in  \delta_{i,j}     (V_{\lambda}\setminus \{0\})$.
Important in the present negative setting, is applying these to the weights $\mu=\rho- \sum_{j\in J}\alpha_j\ \leftrightarrow (1,\ldots, 1)_{j\in J}$. 
\end{observation}
\begin{prop}\label{Proposition solutions with maximal vector for L(rho)}
         Fix $\mathfrak{g}\big(A(n)\big) \text{ and }\rho $ as in Lemma \ref{Lemma negative type A norm equality equaion}, and any $\mathfrak{g}\big(A(n)\big)$-h.w.m. $V\twoheadleftarrow M(\rho)$.  
    Fix a solution tuple $(X_1,\ldots, X_n)\leftrightarrow \mu=\rho-X_1\alpha_1-\cdots -X_n\alpha_n$ to 
\eqref{Eqn norm equality for rho 2}. 
         \begin{itemize}
         \item[(a)]  $V_{\mu}$ has a maximal vector iff $(X_1,\ldots,X_n)$ is a $\{0,2\}$-sequence in \eqref{Eqn solutions of os and 2s for rho in type A}.\\
On useful property (to the proof of Theorem \ref{Theorem D character of V(rho)}) of the complimentary tuples is as follows.
         \item[(b)]
       Assume $(X_1,\ldots, X_n)$ is not any $\{0,2\}$-tuple in \eqref{Eqn solutions of os and 2s for rho in type A}, and $J$ be the union of the supports of all of its $(1,\ldots, 1)$-blocks.
       Let $\widehat{J}=J\sqcup \big\{i\ |\ X_i=0\big\}\ \subsetneqq \{1,\ldots, n\}$; so $\widehat{J}$ contains all the nodes adjacent to $J$.
       \big(E.g. $J=\{4,6\}$ and $\widehat{J}=\{3,5,7\}$ for $(2,1,0,1,0,1,0,2)$.\big) Then
            \begin{equation} \label{Eqn dimension product identity for rho}
\dim V_{\rho -\eta - \sum_{j\in J}\alpha_j}\ =\   \dim V_{\rho -\eta} \times  \dim V_{\rho- \sum_{j\in J}\alpha_j} \qquad \text{for all }\ \eta\in \mathbb{Z}_{\leq 0}\Pi_{\supp(\rho-\mu)\setminus \widehat{J}}.
\end{equation}
        \end{itemize}
     \end{prop}
\begin{proof}[{\bf \textnormal Proof}]
Fix a solution $\mu\leftrightarrow (X_1,\dots, X_n)$ to \eqref{Eqn norm equality for rho 2} outside \eqref{Eqn solutions of os and 2s for rho in type A}.
We first discuss some crucial properties of $(1,\ldots, 1)$-blocks, to go into the proof.
We define for convenience $R_{J}:=\sum_{j\in J}\alpha_j$; notice $R_J$ is a root if and only if the Dynkin subgraph on $J$ is connected.

Fix binary $(Y_1,\ldots,Y_n)\neq (0,\ldots,0)$ with its support contained in $J$; so $\sum_{t=1}^n Y_t\alpha_t\preceq R_J$.
By Lemma \ref{Lemma at least 2 blocks in d-solutions} $(Y_1,\ldots,Y_n)$ \big(and every non-zero rootsum lying above it, in below lines\big) does not satisfy  \eqref{Eqn norm equality for rho 2}. 
Hence: 1)~$V_{\rho-\sum_{t=1}^n Y_t\alpha_t}$ does not have any maximal vectors; 2)~so no proper submodule of $M(\rho)$ contains the weight $\rho-\sum_{t=1}^n Y_t\alpha_t$; 3) whence,  the multiplicity of $\rho-\sum_{t=1}^n Y_t\alpha_t$ remains the same across $M(\rho)\twoheadrightarrow V\twoheadrightarrow L(\rho)$.
So, given $0\neq z\in V_{\rho-R_J}$, there exists $z'\in U\big(\mathfrak{n}^+_J\big)_{R_J}$ s.t. $0\neq z'z\in V_{\rho}$; $z'$ is constructable iteratively as an ordered product of $e_j$ $\forall$ $j\in J$ in view of point 1).

Now we show part (a); the reverse implication is manifest by definitions.
For the forward implication for $\mu$ (non-$\{0,2\}$-tuple), we assume $0\neq x=F_2 F_1v_{\rho}\in V_{\mu}$ to be maximal, for homogeneous elements $F_1\in U(\mathfrak{n}^-_J)_{ -R_J}$ and $F_2\in U\big(\mathfrak{n}^-_{\supp(\rho-\mu)\setminus \widehat{J}}\big)_{-\sum_{i\notin \widehat{J}}X_i\alpha_i}$; to show a contradiction.
Note $F_1$ commutes with $F_2$, as $X_i=0$ for every $i\in \supp(\rho-\mu)\setminus J$ that has at least one edge to $J$.
Fix a suitable $F'_1\in U\big(\mathfrak{n}^+_{J}\big)_{R_J}$ with $F_1'F_1v_{\rho}=v_{\rho}$.
Now for $t\in \supp(\rho-\mu)\setminus \widehat{J}$, by the maximality of $x$, we see $e_tF_2F_1v_{\rho}=0= F'_1e_tF_2F_1v_{\rho}=e_tF_2(F_1'F_1v_{\rho})=e_tF_2v_{\rho}$ to show the maximality of $F_2v_{\rho}$.
But it leads to the contradiction that $\rho- 
 \sum_{i\in \supp(\rho-\mu)\setminus \widehat{J}} X_i\alpha_i\ =\mu+R_J$ satisfies \eqref{Eqn norm equality for rho 2}, as its $d^{(n)}$-value is $d^{(n)}(X_1,\ldots,X_n) - d^{(|\widehat{J}|)}\big((X_j)_{j\in \widehat{J}}\big)=-\big[-2\times (\text{the number of Dynkin graph components of }J) \big]\neq 0$. 
 
For part (b), we fix any $\eta$ supported over $\supp(\rho-\mu)\setminus \widehat{J}$ as in the statement; for $J$ as above.
The result is manifest if $|J|=1$.
Fix two linearly independent sets of vectors $G_jv_{\rho}\in V_{\rho-\eta}$ and $F_iv_{\rho}\in V_{\rho -R_J}$ for  $1\leq j\leq \dim V_{\rho-\eta}\text{ and }  1\leq i\leq \dim V_{\rho -R_J}$. 
Suppose $\sum_{(j,i)}d_{j,i}G_jF_i v_{\rho}=0$ in $V$ for some $d_{j,i}$s in $\mathbb{C}$.
Observation \ref{Observation dual basis by Shapovalov form} yields $F_i'\in U(\mathfrak{n}^+)_{+R_J}$ with $F_i'F_jv_{\rho}\in \delta_{i,j}(V_{\rho}\setminus \{0\})$ $\forall$ $1\leq i,j\leq \dim V_{\rho-R_J}$.
For each $i$, by the linear independence of $G_j$s, $0=F_i'\sum_{(j,t)} d_{j,t}G_j F_t v_{\rho}=\sum_{j}d_{j,i}G_{j}v_{\rho}=0 \implies$ $d_{j,i}=0$ $\forall$ $j$.
So $G_jF_iv_{\rho}$ $\forall$ $(j,i)$ are linearly independent, showing \eqref{Eqn dimension product identity for rho}.
\end{proof}

  \begin{proof}[{\bf \textnormal Proof of Theorem \ref{Theorem D character of V(rho)}}]
Fix any h.w.m. $0\neq V\twoheadleftarrow M(\rho)$ as in the theorem.
We simplify $\mathcal{H}_V= \Big\{(H,m_H)\in \mathrm{Indep}(\mathcal{I}) \Big|\  \Big(\prod_{h\in H} f_h^{m_H(h)=2 }\Big)v_{\rho}=0\Big\}$  as $\mathcal{H}_V=\Big\{  H\subseteq \mathcal{I}= \{1,\ldots, n\}  \Big|\ H \text{ is independent} \\ \text{with }\Big(\prod_{h\in H}f_h^2\Big) v_{\rho}=0 \Big\}$.
By Proposition \ref{Proposition solutions with maximal vector for L(rho)} and holes' definition, the module $\frac{\mathrm{ker\big(M(\rho)\twoheadrightarrow V\big)}}{\sum\limits_{H\in \mathcal{H}_V}U(\mathfrak{n}^-)\prod\limits_{h\in H}f_h^{2} \cdot v_{\rho}}$ has no maximal vectors.
So $V=\mathbb{M}(\rho, \mathcal{H}_V)$, and in particular $L(\rho)=\mathbb{M}\big(\rho,\ \big\{ \{1\},\ldots,\{n\}\big\}\big)$; and so they all have explicit (possibly higher length) CS-type presentations.

\noindent
For proving the character formula in Theorem \ref{Theorem D character of V(rho)}, we follow idea \eqref{Eqn char formula principle 1 from norm equality}, and denominator identity:
\begin{equation}\label{Eqn denominator identity for rho setting over -A-n}
\prod_{\alpha\in \Delta^+}\big(1-e^{-\alpha}\big)^{\dim\mathfrak{g}_{\alpha}}\ =\  R(\mathcal{I})\ = \ {\bf S}_0,
\end{equation}
where-in $R(J)\ :=\qquad\quad\ \sum\limits_{\mathclap{\text{independent }\{i_1,\ldots,i_k\}\subseteq J}}\qquad(-1)^ke^{-\alpha_{i_1}-\cdots-\alpha_{i_k}}$ \quad $\forall
\ J\subseteq \mathcal{I}=\mathcal{I}^-$.\smallskip\\
\underline{Step 1} {\it The computations of $c(\mu)$s  for tuples $\mu= \rho-X_1\alpha_1-\cdots -X_n\alpha_n \ \leftrightarrow \  (X_1,\ldots,X_n)$ of 2s and 0s in \eqref{Eqn solutions of os and 2s for rho in type A}} :
We fix such a $\mu$ with $I:=\supp(\rho-\mu)$ independent, and set $\{H\in \mathcal{H}_V^{\min}\ |\ H\subseteq I \}=\{H_1,\ldots,H_k\}$.
Note, $\mu\in \wt V$ or $\mu\notin \wt V$; and $U(\mathfrak{n}^-_I)$ is the polynomial algebra $\mathbb{C}[\{f_i\ |\ i\in I\}]$.
\begin{center}
The goal of Step 1\ \  is the formula\ \ \  $c(\mu)\ =\  
\sum\limits_{\substack{\{i_1,\ldots,i_r\}\subseteq \{1,\ldots, k\}\\ H_{i_1}\cup\cdots\cup H_{i_r}=I }}(-1)^r$, \ \ proved as follows: \end{center}
The strategy in cases below, and in Step 2: a) we work with partial sums in $\mathrm{char}(V)=\  \sum\limits_{\mathclap{\eta\in \wt V}}\ \ \dim (V_{\eta})e^{\eta}$ that are supported $I$; which b) when multiplied by $R(\mathcal{I})$ (or its suitable factors) contain $c(\mu)e^{\mu}$ from $\mathrm{char}(V)$.
Below, we leave easy checks to see b) for the sums chosen by a), to the reader. \\ 
i) Suppose $\mu\in \wt V$.
We immediately see 0 on the r.h.s. of our goal-formula, and we have to show $c(\mu)=0$.
We note $\dim V_{\eta}=1$ for every $\rho\succeq \eta\succeq \mu$, by which observe $\sum\limits_{\rho\succeq \eta\succeq \mu}(1)e^{\eta}=e^{\rho}\prod\limits_{i\in I}\big(1+e^{-\alpha_i}+e^{-2\alpha_i}\big)=e^{\rho}\prod\limits_{i\in I}\frac{\big(1-e^{-3\alpha_i}\big)}{\big(1-e^{-\alpha_i}\big)}$.
Now $e^{\mu}$ does not occur in $e^{\rho}\prod\limits_{i\in  I}\frac{\big(1-e^{-3\alpha_i}\big)}{\big(1-e^{-\alpha_i}\big)}\times R(\mathcal{I})$ since it cannot 
 occur in $e^{\rho}\prod\limits_{i\in I}\frac{\big(1-e^{-3\alpha_i}\big)}{\big(1-e^{-\alpha_i}\big)} R(I)=e^{\rho}\prod\limits_{i\in I}\frac{\big(1-e^{-3\alpha_i}\big)}{\big(1-e^{-\alpha_i}\big)} \prod\limits_{i\in I}(1-e^{-\alpha_i})= e^{\rho}\prod\limits_{i\in I}\big(1-e^{-3\alpha_i}\big)$, as required.\smallskip \\
ii) Suppose $I\in \mathcal{H}_V^{\min}$; so $\mu\notin \wt V$.
Then every $\rho\succeq \eta\succneqq \mu$ has $\dim V_{\eta}=1$.
Proceeding similar to in case i) with replacing $\sum_{\rho\succeq \eta\succeq\mu}(1)e^{\eta}$ by $\Big(\sum_{\rho\succeq \eta\succeq\mu}(1)e^{\eta}\Big)-e^{\mu}$ in the equations therein, shows $e^{\mu}$ to occur in the resulting sum $-1$ times.
Now $c(\mu)=-1$ by the minimality of $I$, and we are done.\\
iii) Now we assume $I\in \mathcal{H}_V\setminus \mathcal{H}_V^{\min}$; so $\mu\notin \wt V$.
Consider\ \ \   $\sum\limits_{\mu\precneqq\eta\in \wt V}e^{\eta} \ = \  \bigg(\sum\limits_{\mu\preceq \eta\preceq \rho}e^{\eta}\bigg)\ - \ \sum\limits_{\mu\preceq \eta\notin \wt V}e^{\eta}\ =\ 
\bigg(\sum\limits_{\mu\preceq \eta\preceq \rho}e^{\eta}\ - \ \sum\limits_{I\supsetneqq  J\in \mathcal{H}_V,\ J_1\subsetneqq I\setminus J}e^{\rho- \sum_{j\in J_1}\alpha_j - \sum_{j\in J}2\alpha_j}\bigg)\ - \ 
\sum\limits_{I\supseteq J\in \mathcal{H}_V} e^{\rho- \sum_{i\in I\setminus J}\alpha_i-\sum_{j\in J}2\alpha_j}$.
Each of the 3 summations in the 2 parenthesis (on either sides of the second equality) contribute to only 0 in $c(\mu)$; by case i) and as $R(\mathcal{I})$ does not contain $e^{-2\alpha_i}$s for $i\in I\setminus (J\sqcup J_1)$.
The sets $J\in \mathcal{H}_V$ in the terminating sum over them in the previous equation, can be systematically replaced by the unions $H(i_1,\ldots, i_r):= H_{i_1}\cup \cdots\cup H_{i_r}$ for all increasing sequences $1\leq i_1<\cdots < i_r\leq k$ by- i) either the inclusion-exclusion principle for supports $I\setminus J$ of single $\alpha_i$s; ii) or taking the Euler characteristic of $\Big(\prod_{i\in I}f_i\Big)$ multiple of the Taylor resolution \cite{Taylor} for monomial ideal $\frac{U(\mathfrak{n}_{I}^-)}{\sum\limits_{t=1}^k U(\mathfrak{n}_I^-)\Big(\prod_{h\in H_t}f_h\Big)}$.
Namely \\
\hspace*{.5 cm} $ \sum\limits_{\mathclap{J\in \mathcal{H}_V,\  J\subseteq I}}\ \ -e^{\rho- \sum_{i\in I\setminus  J}\alpha_i-\sum_{j\in J}2\alpha_j} = \sum\limits_{\substack{0\preceq \gamma \preceq \ \sum_{h\in I\setminus H(i_1,..,i_r)}\alpha_h \\ \emptyset\neq \{i_1<\cdot\cdot<i_r\}\subseteq \{1,..,k\}}}(-1)^re^{\rho - \gamma -\sum_{h\in I\setminus H(i_1,..,i_r)}\alpha_h- 2\sum_{h\in H(i_1,..,i_r)}\alpha_h}$. \\
For each fixed $H(i_1,..,i_r) \subsetneqq I$, observe \ \ $R(I)\qquad \quad\sum\limits_{\mathclap{0\preceq \gamma \preceq \ \sum_{h\in I\setminus H(i_1,..,i_r)}\alpha_h}} \quad e^{\rho - \gamma -\sum_{h\in I\setminus H(i_1,..,i_r)}\alpha_h- 2\sum_{h\in H(i_1,..,i_r)}\alpha_h}\\ 
\hspace*{0.5cm}=\Big(\prod_{i\in I}(1-e^{-\alpha_i})\Big) \Big(\prod_{h\in I\setminus H(i_1,..,i_r)}(1+e^{-\alpha_h})\Big)e^{\rho -\sum_{h\in I\setminus H(i_1,..,i_r)}\alpha_h- 2\sum_{h\in H(i_1,..,i_r)}\alpha_h}\\
\hspace*{0.5cm} =\Big(\prod_{i\in H(i_1,..,i_r)}(1-e^{-\alpha_i})\Big) \Big(\prod_{h\in I\setminus H(i_1,..,i_r)}(1-e^{-2\alpha_h})\Big)e^{\rho -\sum_{h\in I\setminus H(i_1,..,i_r)}\alpha_h- 2\sum_{h\in H(i_1,..,i_r)}\alpha_h}$,\\
does not visibly contain $e^{\mu}$.
This is because, along any $h\in I\setminus H(i_1,.., i_r)$ direction, we do not see its exponential $e^{-2\alpha_h}$ -- but we only see $e^{-\alpha_h}$ and $e^{-3\alpha_h}$ -- in the products after the second equality in the above equations.  
The proof of the formula of $c(\mu)$ at the beginning of Step 1 is complete.
\smallskip \\
\underline{Step 2} {\it $c(\mu)=0$ for $\mu= \rho-X_1\alpha_1-\cdots -X_n\alpha_n \ \leftrightarrow \  (X_1,\ldots,X_n)$ other than those in \eqref{Eqn solutions of os and 2s for rho in type A}}:
We fix such a $\mu$.
Let $T:=\{i\in \supp(\rho-\mu)\ |\ i \text{ is in the support of a }(1,\ldots, 1)\text{-block} \}$ ($T\neq\emptyset$ by Lemma \ref{Lemma at least 2 blocks in d-solutions}), and let $\widehat{T}=T\sqcup  \big\{ i\in \supp(\rho-\mu)\ \big|\  X_i=0 \big\}$.
So $(X_t)_{t\in T}=(1,\ldots,1)$ and $\Big[\mathfrak{g}_{T},\  \mathfrak{g}_{\supp(\rho-\mu)\setminus \widehat{T}}\Big]=0$.
Let $R_T=\sum_{t\in T}\alpha_t$, and note, $\rho-R_T$ does not satisfy \eqref{Eqn norm equality for rho 2} by Lemma \ref{Lemma at least 2 blocks in d-solutions}.
So, importantly, the coef. of $e^{\rho - R_T}$ in $R(T)\times \sum_{J\subseteq T} \Big(\dim V_{\rho-R_J}\Big)e^{\rho - R_J}$ is 0; also notice $\rho -R_T\in \wt V$ as it is not below any hole-weight in $V$.
Now we compute $c(\mu)$ following substeps i)--iii):\\
i) $c(\mu)$ is the coef. of $e^{\mu}$ in $R(\mathcal{I})\times\sum_{J\subseteq T}   \sum_{0\preceq \eta\preceq\ \mu+R_T}    
\dim V_{\rho-\eta}\times \dim V_{\rho-R_J }e^{\rho-\eta-R_J}$ \big(by  \eqref{Eqn dimension product identity for rho}\big).\\
ii) It suffices to replace $R(\mathcal{I})$ (ignoring the positions of $X_i=0$) by $R\big(T\sqcup  (\supp(\rho-\mu)\setminus \widehat{T})\big)=R(T)\times R(\supp(\rho-\mu)\setminus \widehat{T})$, by definitions of blocks and $A$-type Dynkin graph structure. So-\\
iii) The sum in i) is $R(\supp(\rho-\mu)\setminus \widehat{T})\cdot\Big(\sum_{0\preceq \eta\preceq\ \mu+R_T}  \dim V_{\rho-\eta} e^{-\eta}\Big)\cdot 
\Big(R(T)\sum_{J\subseteq T}\dim V_{\rho- R_J}e^{\rho-R_J}\Big)$. 
Now $e^{\mu}$ cannot occur in the product in point iii), since $e^{\rho-R_T}$ is absent in the second parenthesis therein; the latter was observed prior to point i).
This completes the proof of the theorem.
\end{proof}
\section{Proposition \ref{Prop number theory}: Kac--Kazhdan equation with three solutions}\label{Section 7}
We conclude by showing number-theoretic characterization Proposition \ref{Prop number theory} -- of (infinite) cases of Kac--Kazhdan equation \eqref{Eqn dot action by any positive root} over $\mathfrak{g}\big(A(2)\big)$ with unique solutions in $\lambda-\mathbb{N}\Pi$ -- which follows by:
     \begin{prop}
Fix $M_1\leqslant M_2\in \mathbb{N}$, and define $d := \gcd(M_1, M_2)\text{ and } 
A  := M_1^2 + M_2^2 - M_1 M_2$. 
Next let \(\omega = e^{2\pi i / 3}\) be the cube-root of unity, and in $\mathbb{Z}[\omega], \text{ we define } \tau(z) := k + l \pmod{3} \text{ whenever }\  z = k + l \omega$.  
Further, we set- 
\begin{center}
$N:= \begin{cases} M_1 & \text{ if } M_1 < M_2,  \\
\frac{2}{3} M & \text{ if } M_1 = M_2 = M.
\end{cases}\quad
\gamma :=
\begin{cases}
 M_1 + (M_1 - M_2) \omega & \text{ if } M_1 < M_2 ,\\
  M + M \omega & \text{ if }
  M_1 = M_2 = M. 
\end{cases}$
\end{center}
\end{prop}
\noindent
\textnormal{\bf Proof.} It follows by verifying the equivalence of following statements (1)--(9):
\begin{enumerate}
    \item In $\mathbb{N}\times \mathbb{N}$, the equation $X^2 + Y^2 + XY - M_1 X - M_2 Y = 0$ has exactly one solution \((X,Y)=(M_1, M_2 - M_1)\) if \( M_1 < M_2 \), or $(X,Y)=$\(\left( \frac{2}{3} M, \frac{2}{3} M \right)\) if \( M_1 = M_2 = M \).
    \item For \( X \in [1, M_2) \setminus \{ N \}\), \(-3 X^2 + 2 (2 M_1 - M_2) X + M_2^2\) is not a perfect square.
    \item For \( X \in [1, M_2) \setminus \{ N \}\), \( 4 A - (3 X + M_2 - 2 M_1)^2 \) is not of the form \( 3 z^2 \).
    \item If \( 4 A = v^2 + 3 z^2 \), \( M_2 - 2 M_1 < v < 4 M_2 - 2 M_1 \), and \( v \equiv M_1 + M_2 \pmod{3} \), then \( v = 3 N + M_2 - 2 M_1 \).
    \item If there exists \(\beta \in \mathbb{Z}[\omega]\) such that \( \gamma \bar{\gamma} = \beta \bar{\beta} \), \( M_2 - 2 M_1 < 2 \operatorname{Re} \beta < 4 M_2 - 2 M_1 \), and \( \tau(\beta) = \tau(\gamma) \), then \( \beta = \gamma \) or \( \beta = \bar{\gamma} \).
    \item If  there exists \(\beta \in \mathbb{Z}[\omega]\) such that \( \gamma \bar{\gamma} = \beta \bar{\beta} \), \( \tau(\beta) = \tau(\gamma) \), and \( \operatorname{Re} \beta > \operatorname{Re}(\omega \gamma) \) if \( M_1 < M_2 \), or \( M > \operatorname{Re} \beta > -\frac{M}{2} \) if \( M_1 = M_2 = M \), then \( \beta = \gamma \) or \( \beta = \bar{\gamma} \).
    \item If  there exists \(\beta \in \mathbb{Z}[\omega]\) such that \( \gamma \bar{\gamma} = \beta \bar{\beta} \), then the ideal generated by \( \beta \) equals the ideal generated by \( \gamma \) or \( \bar{\gamma} \), and 3 does not divide \( M_1 + M_2 \) unless \( M_2 = M_1 \) or \( M_2 = 2 M_1 \).
    \item If there exists \(\beta \in \mathbb{Z}[\omega]\) such that the ideal generated by \( \gamma \bar{\gamma} \) equals the ideal generated by \( \beta \bar{\beta} \), then the ideal generated by
    \( \beta \) equals the ideal generated by \( \gamma \) or \( \bar{\gamma} \), and 3 does not divide \( M_1 + M_2 \) unless \( M_2 = M_1 \) or \( M_2 = 2 M_1 \).
\item  (i) The prime factors of \( d \) are 2, 3, or of the form \( 6k - 1 \) for some \( k \in \mathbb{N} \).   \\
 (ii) If \( M_1 = M_2 \), then 3 divides \( d \). \quad
(iii) If \( M_2 \notin \{ M_1, 2 M_1 \} \), then 3 does not divide \( d \), and \(\frac{M_1^2 + M_2^2 - M_1 M_2}{d^2}\) is a prime of the form \( 6m + 1 \) for some \( m \in \mathbb{N} \).\hfill $\blacksquare$
\end{enumerate}

\bigskip

  \noindent
\address{(Souvik Pal)} \textsc{Department of Sciences and Humanities, CHRIST University, Bangalore 560 074, India.}
  \textit{E-mail address}: \email{\texttt{pal.souvik90@gmail.com, souvik.pal@christuniversity.in}} \smallskip\\
       	 \address{(G. Krishna Teja) \textsc{Stat. Math. Unit, Indian Statistical Institute Bangalore Center,  560059,  India.}}
	 \  \  	 \textit{E-mail address}: \email{\texttt{tejag@alum.iisc.ac.in, tejag\_pd@isibang.ac.in}}

 \end{document}